\documentclass[11pt,oneside, reqno, a4paper]{amsart} 

\usepackage[english]{babel}
\usepackage[T1]{fontenc}
\usepackage{amsmath}
\usepackage{amssymb}
\usepackage{float}
\usepackage{amsfonts, bm}
\usepackage{mathtools}
\usepackage[utf8]{inputenc}
\usepackage[mathcal]{eucal}
\usepackage{enumerate}
\usepackage{amsthm}
\usepackage{esint}
\usepackage{hyperref}
\usepackage[totalwidth=14cm,totalheight=20.5cm]{geometry}
\usepackage{epsfig,fancyhdr,color}
\usepackage{multicol} 
\usepackage{csquotes}
\usepackage{tikz-cd}

\newtheorem{cor}{Corollary}[section]

\newtheorem{theorem}[cor]{Theorem}

\newtheorem{prop}[cor]{Proposition}

\newtheorem{lemma}[cor]{Lemma}

\theoremstyle{definition}
\newtheorem{defi}[cor]{Definition}
\newtheorem{manualtheoreminner}{Theorem}
\newenvironment{manualtheorem}[1]{%
  \renewcommand\themanualtheoreminner{#1}%
  \manualtheoreminner
}{\endmanualtheoreminner}
\theoremstyle{remark}
\newtheorem{remark}[cor]{Remark}
\newtheorem*{remark*}{Remark}

\newcommand{\R}{\mathbb{R}}

\newcommand{\C}{\mathbb{C}}

\newcommand{\Diff}{\mathrm{Diff}}

\newcommand{\diag}{\mathrm{diag}}

\newcommand{\SL}{\mathrm{SL}}

\newcommand{\vl}{|}
\newcommand{\Ker}{\mathrm{Ker}}

\newcommand{\PSL}{\mathbb{P}\mathrm{SL}}

\newcommand{\Ima}{\mathrm{Im}}

\newcommand{\trace}{\mathrm{tr}}

\newcommand{\II}{\mathrm{I}\mathrm{I}}

\newcommand{\Sp}{\mathrm{Sp}}

\newcommand{\SU}{\mathrm{SU}}
\newcommand{\GL}{\mathrm{GL}}

\newcommand{\CP}{\mathbb{C}\mathbb{P}}
\newcommand{\I}{\mathbf{I}}
\newcommand{\J}{\mathbf{J}}

\newcommand{\SO}{\mathrm{SO}}

\newcommand{\Hom}{\mathrm{Hom}}
\newcommand{\Span}{\mathrm{Span}}

\newcommand{\Id}{\mathrm{Id}}

\newcommand{\CH}{\mathbb{CH}}
\newcommand{\Ree}{\mathrm{Re}}

\newcommand{\pbz}{\partial_{\bar{z}}}
\newcommand{\pw}{\partial_{w}}
\newcommand{\pz}{\partial_{z}}
\newcommand{\pbu}{\partial_{\bar{u}}}
\newcommand{\pv}{\partial_{v}}

\newcommand{\Hit}{\mathrm{Hit}_3(S)}
\newcommand{\QHit}{\mathrm{QHit}_3^{\C}(S)}

\DeclareMathAlphabet{\mathpzc}{OT1}{pzc}{m}{it}

\usepackage{mathtools}

\title[Surfaces in $\CH^{2}_\tau$ and representations in $\SL(3,\C)$]{Complex Lagrangian minimal surfaces, bi-complex Higgs bundles and $\SL(3,\C)$-quasi-Fuchsian representations}

\begin{document}

\author[Nicholas Rungi]{Nicholas Rungi}
\address{NR: Universit\'e Grenoble Alpes (Institut Fourier), Grenoble, France.} \email{nicholas.rungi@univ-grenoble-alpes.fr} 

\author[Andrea Tamburelli]{Andrea Tamburelli}
\address{AT: Department of Mathematics, University of Pisa, Italy.} \email{andrea.tamburelli@libero.it}

\begin{abstract}In this paper we introduce complex minimal Lagrangian surfaces in the bi-complex hyperbolic space and study their relation with representations in $\SL(3,\C)$. Our theory generalizes at the same time minimal Lagrangian surfaces in the complex hyperbolic plane, hyperbolic affine spheres in $\R^3$, and Bers embeddings in the holomorphic space form $\CP^1 \times \CP^1 \setminus \Delta$. If these surfaces are equivariant under representations in $\SL(3,\C)$, our approach generalizes the study of almost $\R$-Fuchsian representations in $\SU(2,1)$, Hitchin representations in $\SL(3,\R)$, and quasi-Fuchsian representations in $\SL(2,\C)$. Moreover, we give a parameterization of $\SL(3,\C)$-\emph{quasi-Fuchsian} representations by an open set in the product of two copies of the bundle of holomorphic cubic differentials over the Teichm\"uller space of $S$, from which we deduce that this space of representations is endowed with a bi-complex structure. In the process, we introduce bi-complex Higgs bundles as a new tool for studying representations into semisimple complex Lie groups.
\end{abstract}

\date{\today}
\maketitle
\setcounter{tocdepth}{1}

\tableofcontents

\section*{Introduction}
\noindent The study of immersed surfaces in space forms has always been a topic of great interest both geometrically and analytically. In fact, much work has been produced concerning the investigation of minimal surfaces in Riemannian symmetric spaces and their analogue in pseudo-Riemannian homogeneous spaces.
The present paper has two main objectives: the first is to study a special class of immersed surfaces, called \emph{minimal complex Lagrangian}, into the bi-complex hyperbolic plane which has not been defined and considered before. We will explain how this theory generalizes at the same time minimal Lagrangians in complex hyperbolic plane (\cite{loftin2013minimal, huang2013holomorphic, loftin2019equivariant}), hyperbolic affine spheres in $\R^3$ (\cite{loftin2001affine},\cite{Labourie_cubic}) and another special class of immersed surfaces into a holomorphic space form studied by Bonsante and El-Emam (\cite{bonsante2022immersions}). The second objective is an application to the so-called \emph{higher rank Teichmüller theory} (\cite{Wienhard_intro}), which concerns the study of representations of surface groups into a non-compact semisimple Lie group $G$ of arbitrary rank, aiming to generalize Teichmüller theory that considers the case $G=\PSL(2,\R)$. Although classically arising with the case of $G$ being real, recently a great amount of interest has developed around the case of complex Lie groups (\cite{wienhard2016representations},\cite{dumas2020geometry},\cite{alessandrini2021projective},\cite{dumas2021uniformization},\cite{alessandrini2023fiber},\cite{farre2024topological}). In this regard, we show how the space of equivariant complex Lagrangian minimal surfaces parameterizes a neighborhood of the Fuchsian locus into the space of quasi-Hitchin representations in $\SL(3,\C)$, consisting of $B^\C$-Anosov representations, with $B^\C$ being a Borel subgroup. Moreover, the dynamical property of Anosov representations will allow us to describe the boundary at infinity of the surface in terms of boundary maps. In the next paragraphs we describe the prerequisites necessary to state and explain in more detail the main results of the article.
\vspace{0.5em}\newline\underline{\emph{Bi-complex hyperbolic plane and complex Lagrangian minimal surfaces}} \vspace{0.5em}\newline Let $\C_\tau$ be the algebra of bi-complex numbers, namely the $4$-dimensional real algebra generated by $\{1,\tau,i,j\}$ where $\tau^2=1, i^2=j^2=-1$ and $j=\tau i=i\tau$. By definition, any number $z\in\C_\tau$ can be written as $z=z_1+\tau z_2$, where $z_1,z_2$ are complex numbers with respect to $i$. Even though $\C_\tau$ is a real algebra endowed with zero divisors, it is possible to define, as in the complex case, a hyperbolic plane modeled on it by taking the quadric in $\C^3_\tau$ of negative norm vectors with respect to a suitable $\C$-bilinear form $\mathbf{q}$ (Section \ref{sec:bicomplex_hyperbolic_space}). It will be denoted by $\CH^2_\tau$ and named the bi-complex hyperbolic plane. This space is a four-dimensional complex manifold and, from a differential geometric perspective, it is endowed with a special metric structure. As in the usual setting, one can identify the tangent space $T_z\CH^2_\tau\subset\C^3_\tau$ with the $\mathbf{q}$-orthogonal to the $\C_\tau$-linear span of the vector $z\in\C^3_\tau$, resulting in an isomorphism $T_z\CH^2_\tau\cong\C^2_\tau$. Therefore, multiplication by $i$ on $\C^2_\tau$ induces a complex structure $\mathbf{I}$ on $\CH^2_\tau$ and multiplication by $\tau$ induces a para-complex structure $\mathbf{P}$. In particular, the product $\mathbf{J}:=\mathbf{I}\mathbf{P}$ defines another complex structure. In addition, the $\C$-bilinear form $\mathbf q$ allows one to define two other objects: a metric $\hat g$ such that $\hat g(\mathbf{P}\cdot,\mathbf{P}\cdot)=-\hat g$ and a symplectic form $\hat\omega$, related by $\hat\omega=\hat g(\cdot,\mathbf{P}\cdot)$, which are both holomorphic with respect to the complex structure $\mathbf{I}$. In the end, the space $\CH^2_\tau$ is endowed with a so-called \emph{holomorphic para-K\"ahler structure} $(\mathbf{I},\mathbf{P},\hat g,\hat\omega)$ (Section \ref{sec:reduction}) whose complex para-holomorphic sectional curvature is negative constant (Section \ref{sec:bicomplexdiffgeometry}). \\ \\ Throughout the paper, we are interested in immersions $\sigma:U\to\CH^2_\tau$, where 
$U$ is a simply-connected domain of real dimension $2$. Using recent techniques developed by Bonsante and El-Emam (\cite{bonsante2022immersions}), we can treat them as having complex codimension 2, in order to introduce a notion of complex Lagrangian which we now briefly describe. By extending in a $\C$-linear way the differential $\sigma_\ast$ on 
$\C TU$ using the complex structure $\mathbf{I}$, one can consider the classical objects of the real case: first fundamental form $h$, Levi-Civita connection $\nabla$ and second fundamental form $\II$, extended by $\C$-linearity. The immersions satisfying $\trace_{\hat g}\II=0$ and $\sigma^*\hat\omega=0$ will be called \emph{minimal complex Lagrangian}, for which we are able to describe their structure equations:

\begin{manualtheorem}A\label{thm:A}
\textit{Let $\sigma:U\to\CH^2_\tau$ be a complex Lagrangian minimal surface, then the complex first fundamental form $h$ and the complex second fundamental form $\II$ satisfy the following structure equations: \begin{align*}
    &d^{\nabla}\II = 0 \ \ \ \ \ \ \ \ \ \ \ &\text{(Codazzi equation)} \\
    &K_{h}+\| \II \|^{2} = -1 \ \ \ \ \ &\text{(Gauss equation),}
\end{align*}
where $K_{h}$ denotes the complex Gaussian curvature of the metric $h$ and $\| \II \|^{2}$ the norm of $\II$.}
\end{manualtheorem}
\noindent Now let $S$ be a closed surface of genus $g\ge 2$ so that, in the above theorem, $U$ is the universal cover of $S$ and the immersion $\sigma$ is equivariant. We will prove that a pair $(h,\II)$ as in the Theorem \ref{thm:A} is equivalent to a $\C$-complex structure $\mathcal{J}$, i.e. an automorphism of $\C TS$ such that $\mathcal{J}^{2}=-1$, and a Codazzi tensor $C$ compatible with $\mathcal J$. If we denote by $\mathcal{C}^\C(S)$ the space of $\C$-complex structures and by $\mathcal{Q}^\C(S)$ the bundle over $\mathcal{C}^\C(S)$ whose fiber in a point $\mathcal{J}$ is given by the space of Codazzi tensors compatible with $\mathcal{J}$, then its total space can be identified with $\mathcal Q_3(S)\times\mathcal Q_3(S)$ the product of two copies of the bundle of holomorphic cubic differentials over the space of complex structures on $S$. In other words, the data $(\mathcal J,C)$ is equivalent to two complex structures $(J_1,J_2)$ inducing the same orientation and two $(J_1,J_2)$-holomorphic cubic differentials $q_1,q_2$ such that $C=q_1+\bar q_2$. We will call the \emph{Hitchin locus} the subset of $\mathcal{Q}^\C(S)$ given by the preimage of the diagonal in $\mathcal Q_3(S)\times\mathcal Q_3(S)$ via the aforementioned identification. 
\begin{manualtheorem}B\label{thm:B}
\textit{There is a neighborhood $\mathcal{U}$ of the Hitchin locus in $\mathcal{Q}_{3}^{\C}(S)$ such that all pairs $(\mathcal{J}, C) \in \mathcal{U}$ can be realized as embedding data of an equivariant complex Lagrangian minimal surface in $\CH^{2}_{\tau}$.}
\end{manualtheorem}
\noindent In contexts similar to ours, the theory of Higgs bundles and the existence and uniqueness of the solution of Hitchin equations have often been used to study the integrability of the Gauss-Codazzi equations. In our setting of complex metrics, we introduce the new notion of $\SL(3,\C)$-bi-complex Higgs bundles and explain how this is related to a complex Lagrangian minimal immersion $\sigma:\widetilde S\to\CH^2_\tau$ in such a way that the analogues of Hitchin equations for bi-complex Higgs bundles are equivalent to the structure equations of $\sigma$. However, since we no longer have existence and uniqueness of solutions to Hitchin equations in this context, an analytical study of the involved differential operators was necessary. On the other hand, we will use $\SL(3,\C)$-bi-complex Higgs bundles in the proof of another result that will be explained in the next paragraphs. This new notion of bi-complex Higgs bundles can be easily generalized to other complex semisimple Lie groups $G_{\C}$ and we believe could be a useful tool for the study of representations in $G_{\C}$. Finally, the open set for which we have existence and uniqueness of complex Lagrangian minimal surfaces in $\CH^2_\tau$ is actually larger than the one given by Theorem \ref{thm:B}, because a similar approach can be applied to the preimage of a different locus inside $\mathcal Q_3^\C(S)$ obtained by picking $J_1=J_2$ and $q_2=-\bar{q}_1$.
\vspace{0.5em}\newline\underline{\emph{A unified point of view}} \vspace{0.5em}\newline
One of the reasons for starting the study of complex Lagrangian minimal immersions in $\CH^2_\tau$ was to find a unified viewpoint with other seemingly unrelated situations which, as we will explain, are simply a particular case of a more general phenomenon. From the point of view of Lie groups, the automorphism group that preserves the holomorphic para-Kähler structure on $\CH^2_\tau$ is isomorphic to $\SL(3,\C)$, which admits two non-compact real forms: $\SL(3,\R)$ and $\SU(2,1)$. It is known in the literature that $\SU(2,1)$ can be seen as the group of biholomorphisms of the complex hyperbolic space $\CH^2$. Actually, there is another hyperbolic space modeled on the algebra of para-complex numbers and denoted by $\mathbb H^2_\tau$, which admits a para-Kähler structure with the group of para-biholomorphisms given by $\SL(3,\R)$. Using the complex structure $\mathbf{J}$ and the para-complex structure $\mathbf{P}$ on $\CH^2_\tau$, it can be seen that $\CH^2$ is realized as the submanifold $\CH^2_\tau\cap(\R^3\oplus j\R^3)$ and $\mathbb H^2_\tau$ is realized as $\CH^2_\tau\cap(\R^3\oplus\tau\R^3)$. \newline This allows us to observe that if we have a complex Lagrangian minimal immersion $\sigma:\widetilde S\to\CH^2_\tau$, equivariant with respect to a representation 
$\rho:\pi_1(S)\to\SL(3,\C)$, whose image is contained in $\CH^2$, then the surface is actually equivariant with respect to a representation $\SU(2,1)$. Moreover, if $\rho$ is sufficiently close to the Fuchsian locus (\emph{almost $\R$-Fuchsian} in $\SU(2,1)$), then it preserves the unique minimal Lagrangian immersion studied by Loftin and McIntosh (\cite{loftin2013minimal}). On the other hand, if the image is contained in $\mathbb H^2_\tau$, then the representation takes values in $\SL(3,\R)$ and preserves a minimal Lagrangian immersion in the para-complex hyperbolic space. In this case, it is equivalent to having a hyperbolic affine sphere in $\R^3$ equivariant with respect to a Hitchin representation. In particular, we obtain a new variational description of hyperbolic affine spheres as stable critical points of an area functional. A similar result was obtained by Hildebrand (\cite{hildebrand2011cross}), by studying minimal Lagrangian immersions in a space called the cross-ratio manifold, which turns out to be isometric to $\mathbb H^2_\tau$. As a final observation, the submanifold given by 
$\CH^2_\tau\cap(\R^3\oplus i\R^3)$ can be identified with the holomorphic space form $\mathbb X:=\{\underline z\in\C^3 \ | \ z_1^2+z_2^2+z_3^2=-1\}$, which has been shown to be isometric to the complex manifold $\CP^1\times\CP^1\setminus\Delta$, being $\Delta$ the diagonal (\cite{bonsante2022immersions},\cite{el2022gauss}). If the image of the complex Lagrangian minimal immersion is contained in $\mathbb X$, then we recover the \emph{Bers embeddings} studied by Bonsante and El-Emam (\cite{bonsante2022immersions}) whose holonomy representation takes value into $\PSL(2,\C)$. The picture to keep in mind can be summarized in the following two diagrams: \\

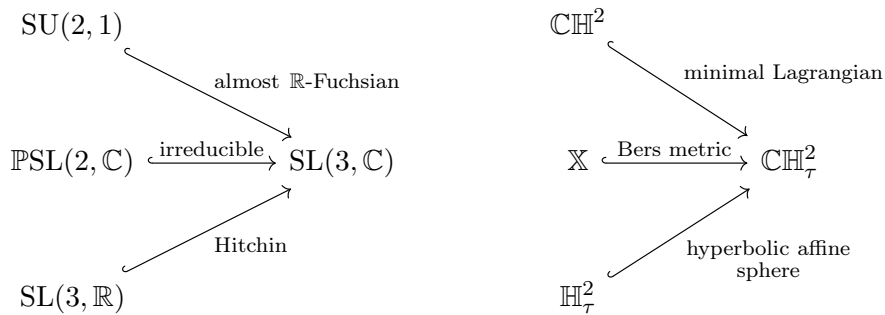
\begin{tikzcd}[row sep=1.1cm, column sep=1.7cm]
{\mathrm{SU}(2,1)} \arrow[rd, "\text{almost} \mathbb \  \R\text{-Fuchsian}", hook]                                      &  & \mathbb{CH}^2 \arrow[rd, "\text{minimal Lagrangian}", hook]             &                    \\
{\mathbb{P}\mathrm{SL}(2,\mathbb C)} \arrow[r, "\text{irreducible}", hook]       & {\mathrm{S}\mathrm{L}(3,\mathbb C)}   & \mathbb X \arrow[r, "\text{Bers metric}", hook]                         & \mathbb{CH}^2_\tau \\
{\mathrm{SL}(3,\mathbb R)} \arrow[ru, "\text{Hitchin}"', hook]                                                   &  & \mathbb H^2_\tau \arrow[ru, "\substack{\text{hyperbolic affine} \\ \text{sphere}}"', hook] &                   
\end{tikzcd}

\vspace{1em}
\noindent In brief, our equivariant complex Lagrangian minimal surfaces in $\CH^{2}_{\tau}$ generalize equivariant minimal Lagrangian surfaces in $\CH^{2}$, hyperbolic affine spheres in $\R^{3}$ and Bers embeddings in $\mathbb{X}$. This allows for a unified treatment of almost $\R$-Fuchsian representations in $\SU(2,1)$, Hitchin representations in $\SL(3,\R)$ and quasi-Fuchsian representations in $\PSL(2,\C)$, as well as their small deformations.\\

\noindent\underline{\emph{Quasi-Hitchin representations and parameterization close to the Fuchsian locus}} \vspace{0.5em}\newline
As previously mentioned, the second motivation for studying complex Lagrangian minimal immersions into $\CH^2_\tau$ comes from higher Teichm\"uller theory, more specifically from the study of quasi-Hitchin representations in complex Lie groups, which we now recall. Let $\eta:\pi_1(S)\to\PSL(2,\R)$ be an injective representation with discrete image. Using the irreducible embedding $\iota:\PSL(2,\R)\hookrightarrow\SL(3,\R)$ we can define a representation $\iota\circ\eta$ into $\SL(3,\R)$. The set of all possible deformations of $\iota\circ\eta$ form a connected component in the corresponding character variety, which is called \emph{Hitchin component}. In fact, it was initially discovered by Hitchin using Higgs bundles techniques (\cite{hitchin1992lie}), proving that it is an open ball of dimension $16g-16$ and contains a copy of Teichm\"uller space coming from the representations $\iota\circ\eta$ as above. A remarkable property of Hitchin representations is that of being $B$-Anosov, where $B<\SL(3,\R)$ is a Borel subgroup (\cite{Labourie2006anosov},\cite{fock2006moduli}). Roughly speaking, a representation $\rho:\pi_1(S)\to\SL(3,\R)$ is $B$-Anosov if there exists a unique $\rho$-equivariant continuous map $\xi$ from $\partial_\infty\pi_1(S)$ to the space of full flags in $\R^3$, which is endowed with a transversality property: for any $x\neq y$ in $\partial_\infty\pi_1(S)$ one has $\xi_x\oplus\xi_y=\R^3$; and a \emph{dynamic preserving} property satisfied whenever the representation is irreducible. A similar definition can be given for representations into the complex Lie group $\SL(3,\C)$, which we will call $B^\C$-Anosov, being $B^\C$ the complexified Borel subgroup. Deformations, inside $\SL(3,\C)$, of Hitchin representations into $\SL(3,\R)$ that remain $B^\C$-Anosov are called \emph{quasi-Hitchin} and form a space denoted with $\QHit$. Even though it is probably already known in the literature, in Section \ref{sec:quasiHitchin} we include a proof that irreducible representations in $\QHit$ are smooth points, thus obtaining, that Hitchin representations in 
$\SL(3,\R)$, almost $\R$-Fuchsian representations in $\SU(2,1)$ and quasi-Fuchsian representations in $\PSL(2,\C)$ post-composed with the irreducible embedding $\iota^\C:\PSL(2,\C)\hookrightarrow\PSL(3,\C)$, which are all examples of quasi-Hitchin representations, are all contained in the smooth locus of the $\SL(3,\C)$-representation variety. \\
\noindent The third main result of the paper, and perhaps the most interesting, concerns a parameterization of the space of $\SL(3,\C)$-\emph{quasi-Fuchsian} representations, obtained by small deformations of representations of the form $\pi_{1}(S) \xrightarrow{\eta} \PSL(2,\R) \xhookrightarrow{\iota} \SL(3,\mathbb{R}) \hookrightarrow \SL(3,\C)$, where $\eta$ is discrete and faithful, which has been object of recent study (\cite{dumas2020geometry}). We define a \emph{holonomy map} $\widetilde{\mathrm{hol}}$, which associates to each point $(\mathcal J,C)\in\mathcal U\subset\mathcal{Q}_3^\C(S)$ the holonomy of the unique equivariant complex Lagrangian minimal immersion determined by the data $(h,\II)$ which is equivalent to $(\mathcal J, C)$. In order to promote this map to a local diffeomorphism, we need to restrict the domain. We consider the subset $\mathcal{R}\subset \mathcal{Q}^{\C}_{3}(S)$ of pairs $(\mathcal{J}, C)$ such that, if $\mathcal{J}=(J_{1},J_{2})$ and $g_{1}$ and $g_{2}$ denote the unique hyperbolic metric in the conformal class of $J_{1}$ and $J_{2}$, then the identity $(S,g_{1}) \rightarrow (S,g_{2})$ is harmonic. We call such pairs $(\mathcal{J}, C)$ \emph{renormalized pairs}. Note that the space of renormalized pairs is still infinite dimensional and the group $\Diff_{0}(S)$ of diffeomorphisms of $S$ isotopic to the identity acts diagonally on $\mathcal{R}$ by pull-back with quotient $Q^{\C}_{3}(S):=Q_{3}(S)\times Q_{3}(S)$ isomorphic to two copies of the bundle of holomorphic cubic differentials over the Teichm\"uller space of $S$. We call \emph{Fuchsian locus} the pre-image in $Q^{\C}_{3}(S)$ of the diagonal inside $\mathrm{Teich}(S)\times \mathrm{Teich}(S) \subset Q_{3}(S) \times Q_{3}(S)$ under the isomorphism described previously. The map $\widetilde{\mathrm{hol}}$ clearly descends to the quotient and defines a map
\[
    \mathrm{hol}: U:=(\mathcal{U}\cap\mathcal{R})/\Diff_{0}(S) \rightarrow \Hom(\pi_{1}(S), \SL(3,\C))/\SL(3,\C) \ . 
\]
\begin{manualtheorem}C\label{thm:C}
\textit{$\SL(3,\C)$-quasi-Fuchsian representations that are sufficiently close to Fuchsian are parameterized by a neighborhood of the Fuchsian locus inside the product $Q_3^\C(S)=Q_3(S)\times Q_3(S)$.}
\end{manualtheorem}
\noindent The proof of the theorem uses Goldman symplectic form $\omega$ (\cite{goldman1984symplectic}). More precisely, we show that the pull-back form $\text{hol}^*\omega$, under the map $\mathrm{hol}$ that associates to the embedding data of an equivariant complex Lagrangian minimal surface its holonomy representation, is non-degenerate at all points $(\mathcal{J},C)$ such that $\mathcal J=(J,J)$ and $C=0$, by using the explicit description of the bi-complex Higgs bundle associated with a complex Lagrangian minimal immersion, and then conclude by the inverse function theorem. \\

\noindent \vspace{0.5em}\underline{\emph{Boundary maps of Anosov representations}} \vspace{0.5em}\newline
The last result we obtain concerns the Anosov property of $\SL(3,\C)$-quasi-Fuchsian representations and their correlation with the complex Lagrangian minimal surface given by Theorem \ref{thm:C}. More specifically, one can define an isometric model of the bi-complex hyperbolic plane as the subspace of 
$\CP^2\times(\CP^2)^*$ formed by all pairs $(v,\varphi)$ such that $\varphi(v)\neq 0$,   where $\varphi$ is considered as a linear functional. Thanks to this description, an analogous notion of boundary at infinity for $\CH^2_\tau$ can be defined, and it is identified with the space of full flags in 
$\C^3$, that is, the homogeneous space $\SL(3,\C)/B^\C$. Therefore, if $\rho$ is an $\SL(3,\C)$-quasi-Fuchsian representation and 
$\xi:\partial_\infty\pi_1(S)\to\SL(3,\C)/B^\C$ is its boundary map, we obtain \begin{manualtheorem}D\label{thm:D}
\textit{The $\rho$-equivariant complex Lagrangian minimal surface $\sigma:\widetilde S\to\CH^2_\tau$ bounds the image of the boundary map $\xi$.}
\end{manualtheorem}

\noindent We believe that this relation about the boundary map $\xi$ and the boundary of the equivariant minimal Lagrangian surface is evidence that this is the correct setting to study quasi-Hitchin representations in $\SL(3,\C)$.\\

\noindent \vspace{0.5em}\underline{\emph{Comparison with other work on representations into complex Lie groups}} \vspace{0.5em}\newline
In recent years, the interest in surface group representations into a complex semisimple Lie group $G^{\C}$ of higher rank has steadily increased. Dumas and Sanders (\cite{dumas2020geometry}, \cite{dumas2021uniformization})  carefully studied the cohomological and analytical properties of the domain of discontinuity $\Omega_{\rho}$ of a $G^{\C}$-quasi-Fuchsian representation $\rho$ acting on flag varieties and conjectured that the quotient $\Omega_{\rho}/\rho(\pi_{1}(S))$ is a fiber bundle over $S$, which they verified for $G^{\C}=\SL(3,\C)$. It is conceivable that the use of our complex Lagrangian minimal surfaces in $\CH^{2}_{\tau}$ combined with the techniques developed by Collier-Tholozan-Toulisse (\cite{CTT}) could simplify the proof of their result considerably. More recently, Alessandrini-Maloni-Tholozan-Wienhard (\cite{alessandrini2023fiber}) proved Dumas and Sanders' conjecture in full generality and described, in particular, the topology of $\Omega_{\rho}/\rho(\pi_{1}(S))$ for quasi-Hitchin representations in $\Sp(4,\C)$ acting on the space of complex Lagrangians in $\C^{4}$. However, in all the aforementioned works, a description of the geometry and topology of the space of $G^{\C}$-quasi-Fuchsian representations is not available. Our work fills this gap, discovering, in addition, a bi-complex structure in the space of $\SL(3,\C)$-quasi-Fuchsian representations close to the Fuchsian locus, which, in our opinion, is very surprising, since neither from the algebraic nor from the Higgs bundle perspective the existence of a para-complex structure is expected. Finally, in the process of writing this paper, we have been informed that El-Emam and Sagman have also been studying the geometry associated to representations of surface groups into $\SL(3,\C)$ (\cite{Bers_SL3}). They define a notion of complex affine spheres in $\C^3$, which naturally correspond to some harmonic maps in $\SL(3,\C)/\SO(3,\C)$ and to new objects, called bi-Higgs bundles, which extend the usual Higgs bundles. These objects allow to define a holomorphic analog of Bers theorem for $\SL(3,\C)$,  namely a holomorphic map from a large open subset of $\Hit\times \overline{\Hit}$ – containing $\Hit \times \mathrm{Teich}(\bar S)$ and $\mathrm{Teich}(S)\times \Hit$ – to the character variety of $\SL(3,\C)$ extending the inclusion of the diagonal. It would be interesting to study the connection between ours and their strategy in the future. At the time of writing, the two works are completely independent.

\subsection*{Outline of the paper}
For the sake of clarity, we will carefully explain where to find the various notions and results that the reader might be interested in. \newline Section \ref{sec:backgroundmaterial} is dedicated to the review of the theory of complex and holomorphic metrics on manifolds, and then to the introduction of differential geometric structures with special curvature properties. Section \ref{sec:bicomplexplane} is entirely dedicated to the introduction and study of the bi-complex hyperbolic plane. In this regard, in Sections \ref{sec:bicomplex_hyperbolic_space} and \ref{sec:incidence_geometry}, we provide two different models of the space; in Section \ref{sec:boundary}, we define its boundary at infinity and give its homeomorphism type; in Section \ref{sec:reduction}, we show how its holomorphic para-Kähler structure can be obtained through an adaptation of the Marsden-Weinstein symplectic reduction; in Section \ref{sec:submanifolds}, we explain how to obtain the submanifolds of $\CH^2_\tau$ mentioned in the introduction, and how the restriction of the holomorphic para-Kähler structure behaves on these submanifolds; in Section \ref{sec:bicomplexdiffgeometry}, we study and show that the complex para-holomorphic curvature of $\CH^2_\tau$ is negative constant. Section \ref{sec:min_Lagrangian} includes the study of complex Lagrangian minimal immersions in $\CH^2_\tau$, and Theorem \ref{thm:A} is proved in Section \ref{sec:embedding_data}. The most analytical part is Section \ref{sec:localcoordinates}, where the Gauss equation and the Laplacian of complex metrics are computed in local coordinates. As for Section \ref{sec:bicomplex_Higgs}, we introduce the so-called bi-complex Higgs bundles and explain how to construct such object from a complex Lagrangian minimal surface. In Section \ref{sec:minimalinpara-complex} we focus on immersions whose image is contained in the para-complex hyperbolic space $\mathbb H^2_\tau$ by showing the equivalence with hyperbolic affine spheres in $\R^3$, and obtaining as a new result that they are stable critical points of a certain area functional. Finally, making use of all the theory developed in the previous sections, Theorem \ref{thm:B} is proved in Section \ref{sec:existencenearFuchsian}, Theorem \ref{thm:C} in Section \ref{sec:parametrFuchsian} and Theorem \ref{thm:D} in Section \ref{sec:boundarymaps}.

\subsection*{Acknowledgement} The authors are grateful to Andrea Seppi for useful conversations, for pointing out the references \cite{hildebrand2011cross, trettel2019families} and for comments on an earlier draft, and to Jérémy Toulisse for a helpful remark on boundary maps.
A.T. acknowledges support from the National Science Foundation with grant DMS-2005551. N.R. is funded by the European Union (ERC, GENERATE, 101124349). Views and opinions expressed are however those of the author(s) only and do not necessarily reflect those of the European Union or the European Research Council Executive Agency. Neither the European Union nor the granting authority can be held responsible for them.

\section{Background material}\label{sec:backgroundmaterial}
\noindent In this section we recall some general notions about complex metrics on manifolds and para-complex structures. Most of the material is classical in the setting of (pseudo)-Riemannian manifolds (see for instance (\cite{kobayashi1996foundations}) and we adapt it here to the context of complex metrics. Part of it can already be found in the literature (\cite{davidov2009hyperhermitian},\cite{baird2011harmonic},\cite{loustau2017bi}). 

\subsection{Complex metrics on complex manifolds}\label{sec:back_complex_metrics}
Let $(M,\mathbf{I})$ be a complex manifold of complex dimension $n$. We think of $\mathbf{I}$ as an integrable almost complex structure on $M$, i.e. an automorphism of $TM$ such that $\mathbf{I}^{2}=-\mathrm{Id}$ for which the Nijenhuis tensor
\[
    N_{\mathbf{I}}(X,Y):=[X,Y]+\mathbf{I}([\mathbf{I}X,Y]+[X,\mathbf{I}Y])-[\mathbf{I}X, \mathbf{I}Y]
\]
vanishes for all $X,Y\in \Gamma(TM)$. It is well-known that this condition is equivalent to the existence of local coordinates $(x_{1}, y_{1}, \dots, x_{n}, y_{n})$ in a neighborhood of every point such that for all $i=1, \dots, n$
\[
    \mathbf{I}\left(\frac{\partial}{\partial x_{i}}\right) = \frac{\partial}{\partial y_{i}} \ \ \ \text{and} \ \ \ \mathbf{I}\left(\frac{\partial}{\partial y_{i}}\right) = -\frac{\partial}{\partial x_{i}} \ .
\]

\noindent A \emph{complex metric} on $(M,\mathbf{I})$ is the datum of a non-degenerate complex-valued symmetric bi-linear form $g_{p}: T_{p}M \times T_{p}M \rightarrow \C$ that varies smoothly in $p$ such that $g_{p}(X, \mathbf{I}Y)=g_{p}(\mathbf{I}X, Y)=ig_{p}(X,Y)$ for all $X,Y\in T_{p}M$. We say that a complex metric is \emph{holomorphic} if for all holomorphic vector fields $Z,Z'$ on $M$ the function $g(Z,Z')$ is holomorphic. \\ 

\noindent A complex metric induces a canonical connection $D$ on $TM$, which has the same properties as the standard Levi-Civita connection in a (pseudo)-K\"ahler manifold: 

\begin{theorem}[\cite{LeBrun}]\label{thm:connection} Given a complex metric $g$ on a complex manifold $(M, \mathbf{I})$, there is a unique connection $D$ on $TM$, called Levi-Civita connection, such that for all $X,Y,Z \in \Gamma(TM)$ the following conditions hold:
\begin{itemize}
    \item $X \cdot g(Y,Z)=g(D_X Y,Z)+g(Y,D_X Z)$ \ \ \ \text{(metric compatibility)}
    \item $[X,Y]=D_X Y - D_Y X$ \ \ \hspace{2cm} \  \ \ \ \  \ \ \ \text{(torsion free).}
\end{itemize}
Moreover, $D\mathbf{I}=0$. 
\end{theorem}

\noindent Moreover, if $g$ is holomorphic the real and imaginary part of $g$ induce two pseudo-Riemannian metrics on $M$ with neutral signature $(n,n)$ and the connection $D$ is the common Levi-Civita connection of both $\Ree(g)$ and $\Ima(g)$. \\

\noindent The notion of Levi-Civita connection for the complex metric $g$ leads to the standard definition of the Riemann curvature tensor
\[
    R^{D}(X,Y)Z:=D_{X}D_{Y}Z-D_{Y}D_{X}Z-D_{[X,Y]}Z \ \ \ \ \ X,Y, Z \in \Gamma(TM) \ .
\]
We will often use the $(4,0)$ Riemann tensor 
\[
    R^{D}(X,Y,Z,W):=g(R^{D}(X,Y)Z, W) \ .
\]
We use the same letter to denote both tensors, as it will be clear from context which one we are using. The Riemann tensor satisfies the usual symmetries:
\begin{align*}
    & R^{D}(X,Y,Z,W)  = -R^{D}(Y,X,Z,W) \\
    & R^{D}(X,Y,Z,W)  = R^{D}(Z,W, X,Y) \\
    & R^{D}(X,Y,Z,W) + R^{D}(Y,Z,X,W) + R^{D}(Z,X,Y,W)=0
\end{align*}
for all $X,Y,Z,W \in \Gamma(TM)$. In particular, since $R^{D}$ is clearly $\C$-linear in the last two entries, i.e. $R^{D}(X,Y,Z,\mathbf{I}W)=R^{D}(X,Y,\mathbf{I}Z, W)=iR^{D}(X,Y,Z,W)$, the Riemann tensor is actually $\C$-linear in all entries.
As a consequence, the sectional curvature of a complex metric $g$, defined as usual as
\[
    \mathrm{Sec}_{g}(X,Y):= - \frac{R^{D}(X,Y,X,Y)}{g(X,X)g(Y,Y)-g(X,Y)^{2}} \ ,
\]
only depends on the complex plane spanned by $X$ and $Y$. Thus, we will refer to $\mathrm{Sec}_{g}$ as \emph{complex sectional curvature} of $g$. Clearly, the complex sectional curvature of a complex plane spanned by $X$ and $Y$ is well defined only when the restriction of $g$ to this plane is non-degenerate. 

\begin{prop} \label{prop:constant_curvature} Let $g$ be a complex metric on $(M, \mathbf{I})$ with constant complex sectional curvature $k$. Let $D$ be the Levi-Civita connection of $g$. Then its Riemann tensor is
\[
    R^{D}(X,Y,Z,W)=k(g(X,Z)g(Y,W)-g(Y,Z)g(X,W)) 
\]  
for all $X,Y,Z,W \in \Gamma(TM)$.
\end{prop}
\begin{proof} The standard proof in the context of (pseudo)-Riemannian manifolds can be adapted word-by-word: the tensor $T(X,Y,Z,W)$ defined as the right-hand side of the formula in the statement satisfies the same symmetries as the Riemann tensor and $T(X,Y,X,Y)=k(g(X,X)g(Y,Y)-g(X,Y)^{2})=R^{D}(X,Y,X,Y)$ for all $X,Y \in \Gamma(TM)$, since $g$ has constant complex sectional curvature $k$. Then one shows, as in the classical case, that two $(4,0)$ tensors that have the same symmetries as the Riemann tensor are uniquely determined by the values they take on quadruples of the form $(X,Y,X,Y)$.    
\end{proof}

\subsection{Positive complex metrics on surfaces}\label{sec:positive_complex_metrics} Complex metrics can be defined on a real manifold $N$ as well: consider the complexified tangent bundle $\C TN$ with the standard complex structure given by the multiplication by $i$ and then define a non-degenerate symmetric $\C$-bilinear form on each $\C T_{p}N$ that changes smoothly with $p$. Theorem \ref{thm:connection} still holds and provides now a Levi-Civita $\C$-bilinear connection $D$. The notions of curvature are then defined exactly in the same way. In particular, when $N$ is a surface, the curvature tensor is uniquely determined by one complex sectional curvature (since there is only one complex plane in $\C TN$), which we call \emph{complex Gaussian curvature}. \\

\noindent In the case of surfaces, the theory of complex metrics has been developed further in \cite{bonsante2022immersions} and a notion of \emph{positive} complex metrics has been introduced that has similarities with the classical conformal geometry of surfaces. \\

\noindent Let $S$ be an oriented surface. The natural inclusion $TS \hookrightarrow \C TS$ factors to a bundle inclusion $\mathbb{P}_{\mathbb{R}}(TS) \hookrightarrow \mathbb{P}_{\C}(\C TS)$ of bundles over $S$. At every $p\in S$, each fiber $\mathbb{P}_{\R}(T_{p}S)$ is mapped homeomorphically into a circle in $\mathbb{P}_{\C}(\C T_{p}S) \cong S^{2}$ whose complement is the disjoint union of two open discs. 

\begin{defi} A \emph{$\C$-complex structure} on $\C TS$ is a $\C$-linear bundle endomorphism $\mathcal{J}: \C TS \rightarrow \C TS$ such that    
\begin{itemize}
    \item $\mathcal{J}^{2} = -\mathrm{Id}$;
    \item for all $p \in S$, the eigenspaces $V_{i}(\mathcal{J})$ and $V_{-i}(\mathcal{J})$ of $\mathcal{J}$ relative to the eigenvalues $\pm i$ have trivial intersection with $T_{p}S$ and the points $\mathbb{P}_{\C}(V_{i}(\mathcal{J}))$ and $\mathbb{P}_{\C}(V_{-i}(\mathcal{J}))$ lie in different connected components of $\mathbb{P}_{\C}(\C T_{p}S) \setminus \mathbb{P}_{\R}(T_{p}S)$.
\end{itemize}
\end{defi}

\noindent Clearly a complex structure on $S$ extends to a $\C$-complex structure. However, not all $\C$-complex structures can be obtained in this way since they may not restrict to an endomorphism of $TS$, in other words $\mathcal{J}(\overline{X}) \neq \overline{\mathcal{J}(X)}$ in general. Nevertheless, a $\C$-complex structure induces two complex structures $J_{1}$ and $J_{2}$ on $S$ defined by the conditions
\[
    V_{i}(\mathcal{J})=V_{i}(J_{2})=\overline{V_{-i}(J_{2})} \ \ \ \text{and} \ \ \ V_{-i}(\mathcal{J})=V_{-i}(J_{1})=\overline{V_{i}(J_{1})}
\]
which uniquely characterize $J_{1}$ and $J_{2}$. It can be checked (\cite{bonsante2022immersions}) that $J_{1}$ and $J_{2}$ induce the same orientation on $S$. Therefore, a $\C$-complex structure determines an orientation on $S$ and the set $\mathcal{C}^{\C}(S)$ of $\C$-complex structures on $S$ compatible with the given orientation is in bijection with $\mathcal{C}(S)\times \mathcal{C}(S)$, i.e. with pairs of complex structures on $S$, via the map that associates to $\mathcal{J}$ the pair $(J_{1}, J_{2})$. $\C$-complex structures that are $\C$-linear extensions of complex structures on $S$ are mapped to the diagonal of $\mathcal{C}(S) \times \mathcal{C}(S)$ under the above correspondence.\\

\noindent Like Riemannian metrics on $S$ determine a complex structure, a positive complex metric $g$ pins down a $\C$-complex structure.

\begin{defi}\label{def:positivecomplexmetric} Let $g$ be a complex metric on $S$ and denote by $\mathrm{Iso}(g_{p})$ the set of isotropic vectors of $g$ on $\C T_{p}S$. We say that $g$ is \emph{positive} if the two points $\mathbb{P}_{\C}(\mathrm{Iso}(g_{p}))$ lie in different connected components of $\mathbb{P}_{\C}(\C T_{p}S) \setminus \mathbb{P}_{\R}(T_{p}S)$ for all $p\in S$.
\end{defi}

\noindent The connection between positive complex metrics and $\C$-complex structures on $S$ is explained in the following result:

\begin{prop}[\cite{bonsante2022immersions}]\label{prop:pos_complex_metrics} Let $g$ be a positive complex metric on an oriented surface $S$. Then there exists a unique $\C$-complex structure $\mathcal{J}$ such that
\begin{itemize}
    \item it induces the fixed orientation on $S$;
    \item it is compatible with $g$ in the sense that $g(\mathcal{J}, \mathcal{J})=g$.
\end{itemize}
\end{prop}

\noindent It follows that if $\mathcal{J}$ is the $\C$-complex structure provided by Proposition \ref{prop:pos_complex_metrics} and corresponds to the pair of complex structures $(J_{1}, J_{2})$, then the positive complex metric $g$ can be written locally as
\[
    g = g\left( \frac{\partial}{\partial w_{1}}, \frac{\partial}{\partial \overline{w}_{2}} \right) dw_{1}d\overline{w}_{2}  \ ,
\]
where $w_{1}$ and $w_{2}$ are local holomorphic coordinates for $J_1$ and $J_{2}$. \\

\noindent A positive complex metric $g$ has also a canonical (up to a sign) area form $d\mathrm{A}_{g}=g(\mathcal{J}\cdot, \cdot)$ characterized by the property that it evaluates to $\pm 1$ on all $g$-orthonormal frames $\{X, Y\}$ of $\C TS$. In local coordinates, if $g=\rho dw_{1}d\overline{w}_{2}$, then
\begin{equation}\label{eq:areaform}
    d\mathrm{A}_{g} = \frac{i}{2}\rho dw_{1}\wedge d\overline{w}_{2} \ . 
\end{equation}

\subsection{Complex para-K\"ahler and complex para-Hermitian structures}\label{sec:complexpara-Kahler}
A \emph{para-complex structure} on a complex $n$-manifold $(M,\mathbf{I})$ is a bundle endomorphism $\mathbf{P}:TM\rightarrow TM$ such that
\begin{itemize}
    \item $\mathbf{P}^{2}=\mathrm{Id}$;
    \item the eigendistributions $\mathcal{D}_{\pm}$ relative to the $\pm 1$ eigenvalues of $\mathbf{P}$ have the same dimension and are integrable.
\end{itemize}
This last condition is equivalent to the vanishing of its Nijenhuis tensor
\[
    N_{\mathbf{P}}(X,Y):=-[X,Y]+\mathbf{P}([\mathbf{P}X,Y]+[X,\mathbf{P}Y])-[\mathbf{P}X, \mathbf{P}Y]
\]
or, equivalently, to the existence of a local product structure, in the sense that around any point $p \in M$ we can find local coordinates $(x_{1}, \dots, x_{n}, y_{1}, \dots, y_{n})$ such that for all $i=1, \dots, n$
\[
    \mathbf{P}\left(\frac{\partial}{\partial x_{i}}\right) = \frac{\partial}{\partial x_{i}} \ \ \ \text{and} \ \ \ \mathbf{P}\left(\frac{\partial}{\partial y_{i}}\right) = -\frac{\partial}{\partial y_{i}} \ .
\]

\noindent Let us now endow $(M, \mathbf{I})$ with a complex metric $g$ with Levi-Civita connection $D$. We say that the para-complex structure $\mathbf{P}$ is compatible with $g$ if $g(\mathbf{P}\cdot, \mathbf{P}\cdot)=-g$. In this case, the form $\omega_{\mathbf{P}}:= g(\cdot, \mathbf{P}\cdot)$ is skew-symmetric and we can extend $g$ to a complex para-Hermitian metric by defining the tensor $q := g + \tau \omega_{\mathbf{P}}$, where $\tau$ is a non-real element such that $\tau^2=1$. 

\noindent In general, a \emph{complex para-Hermitian metric} on $(M, \mathbf{I}, \mathbf{P})$ is the datum of a bilinear form $q_{p}$ on $T_{p}M$ with values in the algebra of bi-complex numbers $\C_{\tau}:=\C[\tau]$ (see Section \ref{sec:bicomplex_numbers} for more information about this algebra) that varies smoothly with $p$ and satisfies the following properties:
\begin{itemize}
    \item $q(\mathbf{I}X, Y)=q(X,\mathbf{I}Y)=iq(X,Y)$ 
    \item $q(\mathbf{P}X, Y)=-q(X, \mathbf{P}Y)=\tau q(X,Y)$ 
    \item $q(X,Y)=\overline{q(Y,X)}^{\tau}$
\end{itemize}
for all $X,Y \in \Gamma(TM)$. It is straightforward to check that, if $q$ is a complex para-Hermitian metric on $(M, \mathbf{I}, \mathbf{P})$, then $g=\frac{1}{2}(q+\bar{q}^{\tau})$ is a complex metric compatible with $\mathbf{P}$ and $\frac{1}{2}(q-\bar{q}^{\tau})$ is a skew-symmetric $\C$-bilinear form that coincides with $\omega_{\mathbf{P}}$. Therefore, all para-Hermitian metrics can actually be obtained from the above construction and we will refer to the quadruple $(\mathbf{I}, \mathbf{P}, g, \omega_{\mathbf{P}})$ as a \emph{complex para-Hermitian structure} on $M$. Moreover, we say that the complex para-Hermitian structure is \emph{holomorphic} if the complex metric $g$ and the symplectic form $\omega_{\mathbf{P}}$ are holomorphic.\\

\noindent We say that a para-complex structure $\mathbf{P}$ on $(M, \mathbf{I}, g)$ is \emph{complex para-K\"ahler} if $\omega_{\mathbf{P}}$ is a non-degenerate symplectic form; if, moreover, $g$ and $\omega_{\mathbf{P}}$ are holomorphic, we call it \emph{holomorphic para-K\"ahler}. As in the classical setting, the K\"ahler condition is equivalent to $D\mathbf{P}=0$.  \\

\noindent A \emph{bi-complex} structure on $M$ is a complex para-Hermitian structure $(\mathbf{I}, \mathbf{P}, g, \omega_{\mathbf{P}})$ such that $\mathbf{J}:=
\mathbf{I}\mathbf{P}=\mathbf{P}\mathbf{I}$ is an integrable complex structure. We say that it is \emph{bi-K\"ahler} if $\mathbf{P}$ (and hence $\mathbf{J}$) is parallel for the Levi-Civita connection of $g$. Moreover, the quadruple $(\mathbf{I}, \mathbf{P}, g, \omega_{\mathbf{P}})$ gives rise to a \emph{holomorphic bi-K\"ahler} structure if, in addition, the metric $g$ is holomorphic. \begin{remark}
Our definition of holomorphic bi-K\"ahler structure coincides with the notion of \emph{holomorphic bi-complex K\"ahler} structure given by Loustau and Sanders (\cite[Definition 2.14]{loustau2017bi}). They also showed that it is equivalent to a structure called \emph{complex bi-Lagrangian} which consists of the data of a complex symplectic form $\omega$ on $(M,\mathbf{I})$ and an ordered pair $(\mathcal{L}_1,\mathcal{L}_2)$ of transverse complex Lagrangian foliations (\cite[Proposition 2.15]{loustau2017bi}). In particular, bi-complex structures appear naturally in the complexification of complex manifolds (\cite[Theorem 3.13]{loustau2017bi})
\end{remark}

\noindent On a complex para-Hermitian manifold $(M, \mathbf{I}, \mathbf{P}, g, \omega_{\mathbf{P}})$ with Levi-Civita connection $D$, we can consider the complex sectional curvature of complex planes spanned by $X$ and $\mathbf{P}X$:
\[
    \mathrm{Sec}_{g}(X,\mathbf{P}X)=\frac{R^{D}(X,\mathbf{P}X, X, \mathbf{P}X)}{g(X,X)^2} \ .
\]
We call it \emph{para-holomorphic complex sectional curvature} and we say that $g$ has constant para-holomorphic complex sectional curvature if $\mathrm{Sec}_{g}(X,\mathbf{P}X)$ is a constant function independent of $X$, whenever $X$ is non isotropic. A result similar to Proposition \ref{prop:constant_curvature} holds in this setting:

\begin{prop} \label{prop:constant_parahol_curvature} Let $(M, \mathbf{I}, \mathbf{P}, g, \omega_{\mathbf{P}})$ be a complex para-Hermitian manifold with constant para-holomorphic complex sectional curvature $k$. Then the Riemann tensor of $g$ can be written as
\begin{align*}
    R^{D}(X,Y,Z,W)=-\frac{k}{4}\big(&g(X,Z)g(Y,W)-g(Y,Z)g(X,W)+g(X,\mathbf{P}Z)g(\mathbf{P}Y,W) \\ 
    &-g(Y,\mathbf{P}Z)g(\mathbf{P}X,W)+2g(X,\mathbf{P}Y)g(\mathbf{P}Z,W)  \big)
\end{align*}
for all $X,Y,Z,W \in \Gamma(TM)$.
\end{prop}
\begin{proof} The same argument that is used in the pseudo-Riemannian setting works (\cite{bonome1998paraholomorphic}). We give only a brief sketch. Let $T(X,Y,Z,W)$ denote the tensor defined by the right hand side of the formula in the statement. It is straightforward to verify that
\begin{align}\label{eq:symmetries}
    &T(X,Y,Z,W)=-T(Y,X,Z,W)=-T(X,Y,W,Z) \notag\\
    &T(X,Y,Z,W)= T(Z,W,X,Y) \\
    &T(X,Y,Z,W)+T(Y,Z,X,W)+T(Z,X,Y,W)=0 \notag\\
    &T(X,Y,Z,W)=-T(\mathbf{P}X, \mathbf{P}Y, Z, W)=-T(X,Y,\mathbf{P}Z, \mathbf{P}W) \notag
\end{align}
for all $X,Y,Z,W \in \Gamma(TM)$. Moreover, we note that $T(X,\mathbf{P}X,X,\mathbf{P}X)=kg(X,X)^{2}$ $=R^{D}(X,\mathbf{P}X,X, \mathbf{P}X)$ for all non-isotropic $X\in \Gamma(TM)$, because $g$ has constant para-holomorphic complex sectional curvature $k$. Then one proves that a $(4,0)$-tensor enjoying the symmetries in Equation \eqref{eq:symmetries} is uniquely determined by the values it takes on quadruples of vectors of the form $(X,\mathbf{P}X, X, \mathbf{P}X)$.
\end{proof}

\noindent In the case of bi-complex structures, Proposition \ref{prop:constant_parahol_curvature} implies that $g$ has also constant \emph{$\mathbf{J}$-holomorphic complex sectional curvature} in the sense that all non-degenerate complex planes spanned $X$ and $\mathbf{J}X$ have the same complex sectional curvature. Indeed, 
\begin{align*}
    R^{D}(X,\mathbf{J}X, X, \mathbf{J}X) &=-\frac{k}{4}\big(g(X,X)g(\mathbf{J}X,\mathbf{J}X)-g(\mathbf{J}X,X)^2+g(X,\mathbf{P}X)g(\mathbf{P}\mathbf{J}X,\mathbf{J}X) \\ 
    & \ \ \ \ -g(\mathbf{J}X,\mathbf{P}X)^{2}+2g(X,\mathbf{P}\mathbf{J}X)g(\mathbf{P}X,\mathbf{J}X)  \big) \\
    &= -kg(X,X)^{2} 
\end{align*}
hence
\[
    \mathrm{Sec}_{g}(X,\mathbf{J}X) = - \frac{{R}^{D}(X,\mathbf{J}X, X ,\mathbf{J}X)}{g(X,X)^2} = k \ 
\]
for all $X\in \Gamma(TM)$ not isotropic for $g$.

\section{Bi-complex hyperbolic plane}\label{sec:bicomplexplane}
\noindent In this section we introduce bi-complex numbers and the bi-complex hyperbolic space $\mathbb{CH}_{\tau}^2$, which generalizes both the complex hyperbolic space $\mathbb{CH}^{2}$ and the para-complex hyperbolic space $\mathbb{H}_{\tau}^2$, and we will describe its basic properties.

\subsection{Bi-complex numbers}\label{sec:bicomplex_numbers}
\noindent We review here some basic properties of bi-complex numbers and fix some notation that we will use throughout the paper. \\

\noindent The algebra of bi-complex numbers $\mathbb{C}_\tau$ is the commutative algebra over $\mathbb{C}$ generated by $1$ and $\tau$, where $\tau$ is a non-complex number such that $\tau^2=1$. It is a $4$-dimensional vector space over the reals with basis $\{1, \tau, i, j:=\tau i\}$. Here we denoted by $i$ the standard imaginary unit in $\mathbb{C}$. Note that $j$ is also a square root of $-1$ since $i$ and $\tau$ commute. \\

\noindent Given a bi-complex number $w=z_{1}+\tau z_{2}$, we will use the standard notation $\bar{w}$ to denote complex conjugation, so that $\bar{w}=\bar{z}_{1}+\tau\bar{z}_{2}$. We then define
\[
    \Ree(w):=\frac{w+\bar{w}}{2}  \ \ \ \ \  \text{and} \ \ \ \ \ \Ima(w):=\frac{w-\bar{w}}{2i} \ . 
\]
These are para-complex numbers, i.e. elements of the real subalgebra $\mathbb{R}_\tau \subset \mathbb{C}_\tau$, with the property that $w=\Ree(w)+i\Ima(w)$. In a similar way, we denote by $\bar{w}^{\tau}$ the conjugate of $w$ with respect to the para-complex unit $\tau$, in other words $\bar{w}^{\tau}:=z_{1}-\tau z_{2}$. We then define
\[
    \Ree_{\tau}(w):=\frac{w+\bar{w}^{\tau}}{2}  \ \ \ \ \  \text{and} \ \ \ \ \ \Ima_{\tau}(w):=\frac{w-\bar{w}^{\tau}}{2\tau} \ 
\]
as the $\tau$-real and $\tau$-imaginary part of $w$. Note that, despite the name, $\Ree_{\tau}(w)$ and $\Ima_{\tau}(w)$ are complex numbers. \\

\noindent The $\tau$-norm of a bi-complex number $w=z_{1}+\tau z_{2}$ is the complex number 
\[
    |w|_{\tau}:=w\bar{w}^{\tau}=z_{1}^2-z_{2}^{2} \ .
\]
The non-vanishing of the $\tau$-norm characterizes invertible bi-complex numbers. 

\noindent For what follows, it will be useful to write a bi-complex number $w=z_{1}+\tau z_{2}$ with respect to the basis of idempotents $\{e_{+}, e_{-}\}$, where
\[
    e_{+}:= \frac{1+\tau}{2} \ \ \ \ \ \text{and} \ \ \ \ \ e_{-}:= \frac{1-\tau}{2} \ .
\]
By a straightforward computation, we can rewrite $w=z^{+}e_{+}+z^{-}e_{-}$ with 
\[
    z^{+}=z_{1}+z_{2} \ \ \ \ \ \text{and} \ \ \ \ \ z^{-}=z_{1}-z_{2} \ .  
\]
\begin{remark}
We want to emphasize that our definition of the algebra of bi-complex numbers might a-priori seem different from what can be found in the literature (see \cite{luna2012bicomplex}), but in fact they are equivalent. It will become clear in the next sections why the definition we have given is useful for our purposes.
\end{remark}
\subsection{Bi-complex hyperbolic plane: the hyperboloid model}\label{sec:bicomplex_hyperbolic_space}

Consider the vector space $\C^3_\tau$ endowed with the $\C$-bilinear para-Hermitian form
\[
    \mathbf{q}(z, w) := z^{t}Q\bar{w}^{\tau} \ ,
\]
where $Q=\diag(1,1,-1)$. 

\begin{defi} The bi-complex hyperbolic plane is
\[
    \CH^{2}_{\tau}=\{ z \in \C^{3}_{\tau} \ | \ \mathbf{q}(z,z)=-1\} / \sim
\]
where $z\sim z'$ if there exists $w\in \C_{\tau}$ of unit $\tau$-norm such that $z=wz'$.
\end{defi}

\noindent It follows immediately from the definition that $\CH^{2}_{\tau}$ is a complex manifold of dimension $4$. Moreover, given $z\in \CH^{2}_{\tau}$, we can identify the tangent space $T_{z}\CH^{2}_{\tau}$ of $\CH^{2}_{\tau}$ with the $\mathbf{q}$-orthogonal complement of $z$ obtaining thus a decomposition 
\begin{equation}\label{eq:decomposition_ambient}
    T_{z}\C^{3}_{\tau}=\C_{\tau}z \oplus T_{z}\CH^{2}_{\tau} \ .
\end{equation}
This shows that $T_{z}\CH^{2}_{\tau}$ is naturally isomorphic to $\C_{\tau}^{2}$, hence it is
equipped with a complex structure $\mathbf{I}$, inherited from the complex coordinates on $\C^{3}_{\tau}$ and corresponding to the multiplication by $i$ on $\C_{\tau}^{2}$, a para-complex structure $\mathbf{P}$ induced by the multiplication by $\tau$, and another complex structure by composition $\mathbf{J}:=\mathbf{I}\mathbf{P}=\mathbf{P}\mathbf{I}$. This explains the reason why we called this space \textit{bi-complex} hyperbolic space. In addition, the restriction of $\Ree_{\tau}(\mathbf{q})$ to $T_{z}\CH^{2}_{\tau}$ defines a holomorphic metric $\hat{g}$ such that $\hat{g}(\mathbf{P} \cdot, \mathbf{P}\cdot)=-\hat{g}$. Therefore, the pairing $\hat{\omega}=\hat{g}(\cdot, \mathbf{P}\cdot)$ defines a holomorphic symplectic form and the $\C_{\tau}$-valued tensor $\hat{q}:=\hat{g}+\tau\hat{\omega}$ defines a holomorphic para-Hermitian metric on $\CH^{2}_{\tau}$. \\ 

\noindent The isometry group of $\CH^{2}_{\tau}$ consists of all $\C$-linear transformations of $\C_{\tau}^{3}$ that preserve $\mathbf{q}$. We denote such group with $\mathrm{U}(2,1,\C_{\tau})$.

\begin{prop}\label{prop:isometries} The map 
\begin{align*}
    \Phi: \mathrm{GL}(3,\C) &\rightarrow \mathrm{U}(2,1,\C_{\tau}) \\ 
        A &\mapsto Ae_{+}+Q(A^{-1})^{t}Qe_{-}
\end{align*}
is an isomorphism of groups. In particular, the subgroup $\mathrm{SU}(2,1,\C_{\tau})$ of matrices with unit determinant preserving $\mathbf{q}$ is isomorphic to $\mathrm{SL}(3,\C)$. 
\end{prop}
\begin{proof}
Let $X \in \mathrm{U}(2,1,\C_{\tau})$. Rewriting each entry of $X$ in the basis of idempotents, we can decompose $X=X_{+}e_{+}+X_{-}e_{-}$ with $X_{\pm} \in \mathfrak{gl}(3,\C)$. Now, the matrix $X$ preserves the $\C$-bilinear $\tau$-Hermitian form $\mathbf{q}$ if and only if $X^{t}Q\overline{X}^{\tau}=Q$. Unraveling this equation we get
\begin{align*}
    Q &=X^{t}Q\overline{X}^{\tau}\\
     &=(X_{+}e_{+}+X_{-}e_{-})^{t}Q(\overline{X_{+}e_{+}+X_{-}e_{-}}^{\tau}) \\
     &=(X_{+}^{t}e_{+}+X_{-}^{t}e_{-})Q(X_{+}e_{-}+X_{-}e_{+}) \tag{$\overline{e_{\pm}}^{\tau}=e_{\mp}$}\\
     &=X_{+}^{t}QX_{-}e_{+}+X_{-}^{t}QX_{+}e_{-} \tag{$e_{+}e_{-}=0$} \\
     &=\frac{1}{2}(X_{+}^{t}QX_{-}+X_{-}^{t}QX_{+})+\frac{\tau}{2}(X_{+}^{t}QX_{-}-X_{-}^{t}QX_{+}) \ ,
\end{align*}
which is satisfied if and only if 
\[
    \begin{cases}
        X_{+}^{t}QX_{-}-X_{-}^{t}QX_{+}=0 \\
        X_{+}^{t}QX_{-}+X_{-}^{t}QX_{+} = 2Q \ .
    \end{cases}
\]
We deduce that $X_{+}^{t}QX_{-}=X_{-}^{t}QX_{+}$ and $X_{-}=Q(X_{+}^{-1})^{t}Q$. This shows that $\Phi$ is surjective. Since $\mathrm{Ker}(\Phi)=\{\Id\}$, the map $\Phi$ is an isomorphism of groups. \\
Finally, observing that 
\begin{align*}
    \det(X)&=\det(X_{+}e_{+}+X_{-}e_{-})=\det(X_{+})e_{+}+\det(X_{-})e_{-} \\
    &=\det(X_{+})e_{+}+\det(X_{+})^{-1}e_{-} \ ,
\end{align*}
a matrix $X\in \mathrm{U}(2,1,\C_{\tau})$ has unit determinant if and only if $\det(X_{+})=1$. Therefore, $\Phi^{-1}(\mathrm{SU}(2,1,\C_{\tau}))=\mathrm{SL}(3,\C)$. 
\end{proof}

\begin{remark} The equation $X^{t}Q\overline{X}^{\tau}=Q$ satisfied by an $X\in \mathrm{U}(2,1, \C_{\tau})$ implies that $\det(X)$ is a unitary bi-complex number. As a result, we can identify $\mathrm{SU}(2,1,\C_{\tau})$ with the quotient of $\mathrm{U}(2,1, \C_{\tau})$ by the action by scalar multiplication of elements of $\C_\tau$ with $\tau$-norm equal to $1$. 
\end{remark}

\noindent It is straightforward to verify that $\mathrm{SL}(3,\C)$ acts transitively on $\CH_{\tau}^{2}$ with stabilizer isomorphic to $\mathrm{GL}(2,\C)$. Thus, we can identify $\CH_{\tau}^{2}$ with the homogeneous space $\mathrm{SL}(3,\C) / \mathrm{GL}(2,\C)$. This suggests a relation between the bi-complex hyperbolic plane and the complex projective plane, which we explain in the next subsection.

\subsection{Bi-complex hyperbolic plane: relation with incidence geometry}\label{sec:incidence_geometry}

We give now another model of $\CH^{2}_{\tau}$ that will play an important role in bringing a connection between minimal Lagrangian surfaces in $\CH^{2}_{\tau}$ and hyperbolic affine spheres (see \cite[\S 8.5]{trettel2019families} for a related discussion on a similar case). \\

\noindent Using the basis of idempotents $\{e_{+}, e_{-}\}$, we can see $\C^{3}_{\tau}$ as the product $\C^{3}e_{+} \oplus \C^{3}e_{-}$. In the sequel, we will identify the second $\C^{3}$ factor with the dual of $\C^{3}$ with the isomorphism being induced by the bilinear form $Q$ via Riesz representation theorem; in other words we will identify
\begin{align*}
    (\C^{3})^{*} &\longleftrightarrow \C^{3} \\
        \varphi &\longleftrightarrow v_{\varphi} \ ,
\end{align*}
where $v_{\varphi}$ is the unique vector in $\C^{3}$ such that $\varphi(w)=v_{\varphi}^tQw$ 
for all $w\in \C^{3}$. 

\begin{prop}\label{prop:incidence} There is a diffeomorphism between $\CH^{2}_{\tau}$ and the set
\[
    \mathbb{P}\mathcal{L}=\{ (v,\varphi) \in \CP^{2} \times (\CP^{2})^* \ | \ \varphi(v)\neq 0\}) \ ,
\]
consisting of pairs of projective classes of transverse lines and planes in $\C^{3}$. 
\end{prop}
\begin{proof} We first show that the change of coordinates
\begin{align*}
       \alpha: \C^{3}_{\tau} \rightarrow \C^{3}e_{+} \oplus \C^{3}e_{-}
\end{align*}
composed with the above identification between the second $\C^{3}$-factor and its dual induces a diffeomorphism between the quadric
\[
    V_{-1}=\{ z \in \C^{3}_{\tau} \ | \ \mathbf{q}(z,z)=-1\}
\]
and 
\[
    \mathcal{L}_{-1}=\{ (v,\varphi) \in \C^{3} \times (\C^{3})^* \ | \ \varphi(v)=-1\} . 
\]
Indeed, if $x, y \in \C^{3}$, then $\alpha(x+\tau y)=(x+y)e_{+}+(x-y)e_{-}$. The functional corresponding to $(x-y)$ is represented by the row vector $(x-y)^{t}Q$. Therefore, if $z=x+\tau y \in V_{-1}$, then 
\begin{align*}
    -1=\mathbf{q}(z,z)&=z^{t}Q\bar{z}^{\tau} \\
            & =x_{1}^{2}-y_{1}^{2}+x_{2}^{2}-y_2^{2}-x_{3}^{2}+y_{3}^{2} \\
            & = (x-y)^{t} Q (x+y) \ ,
\end{align*}
and the last expression equals exactly the evaluation of the functional $(x-y)^{t}Q$ on the vector $(x+y)$, thus the pair $(x+y,(x-y)^{t}Q)$ belongs to $\mathcal{L}_{-1}$. This shows that $\mathcal{L}_{-1}$ is diffeomorphic to $V_{-1}$, as it is clear that the map constructed above is smooth and with smooth inverse.\\
Now, if $z \in V_{-1}$ is multiplied by $w\in \C_{\tau}$ with $|w|_{\tau}=1$, then $w=u_{+}e_{+}+u_{-}e_{-}$ with $u_{+}u_{-}=1$ and the corresponding pair $(x+y, (x-y)^{t}Q, ) \in \mathcal{L}_{-1}$ changes to $(u_{+}(x+y), u_{-}(x-y)^{t}Q)$. Therefore, the diffeomorphism defined before induces one between $\CH^{2}_{\tau}$ and the quotient of $\mathcal{L}_{-1}$ by the action of $\C^{*}$ given by $\lambda \cdot (v, \varphi):= (\lambda v, \lambda^{-1}\varphi)$, which is exactly the set $\mathbb{P}\mathcal{L}$ of the statement. 
\end{proof}

\subsection{Topology and the boundary at infinity of $\CH^{2}_{\tau}$}\label{sec:boundary}
Here we explain the topological type of the bi-complex hyperbolic space and we introduce its boundary at infinity. \begin{prop}
The space $\CH^2_\tau$ is homeomorphic to the tangent bundle $T\CP^2$. In particular, it is simply connected and the (real) para-complex hyperbolic space $\mathbb H^2_\tau$ is homeomorphic to $T\R\mathbb P^2$.
\end{prop}
\begin{proof}
Let us denote with $\vl\vl\cdot\vl\vl$ and with $\langle\cdot,\cdot\rangle$ respectively the Euclidean norm and scalar product on $\R^6$. The space $$\mathcal{V}:=\{(u,v)\in\R^6\times\R^6 \ | \ \vl\vl u\vl\vl=\vl\vl v\vl \vl, \ \langle u,v\rangle=1 \}$$ is homeomorphic to the tangent bundle of $S^5$ via the map \begin{align*}
    \Phi:& \ \mathcal{V}\longrightarrow TS^5 \\ & (u,v)\longmapsto \bigg(\frac{u+v}{\vl\vl u+v\vl\vl},u-v\bigg) \ ,
\end{align*}being the vectors $u+v$ and $u-v$ orthogonal (see also \cite{gadea1989spaces}). After identifying $\R^6$ with $\C^3$ we can think of $\mathcal{V}$ as a subset of $\C^3\times\C^3$. It carries an $S^1$-action given by $\lambda\cdot (u,v):=(\lambda u, \lambda^{-1}v)$, which is nothing but the restriction on $\mathcal{V}$ of the standard $(\C^*\times\C^*)$-action on $\C^3\times\C^3$. The quotient space is identified with the set of projective classes of transverse lines and planes in $\C^3$, which is an equivalent model for $\CH^2_\tau$ (Proposition \ref{prop:incidence}). The above $S^1$-action is induced on $TS^5$ via the homeomorphism $\Phi$ and the corresponding quotient space is identified with $T\CP^2$, where $\CP^2\cong S^5/S^1$.
The claim regarding $\mathbb H^2_\tau$ comes from the observation $\mathbb H^2_\tau=\CH^2_\tau\cap\big(\R^3\oplus\tau\R^3\big)$ (see also Section \ref{sec:submanifolds}).
\end{proof}
\noindent Using the model of $\CH_{\tau}^{2}$ in terms of incidence geometry (see Section \ref{sec:incidence_geometry}), it is natural to define the boundary at infinity of $\CH^{2}_{\tau}$ as the topological boundary of $\CH^{2}_{\tau}$ inside $\CP^{2}\times (\CP^{2})^{*}$. Therefore, we set
\[
    \partial_{\infty}\CH^{2}_{\tau}:=\{(v,\varphi) \in \CP^{2}\times (\CP^{2})^{*} \ | \  \varphi(v)=0\} \ .
\]
We will see in Section \ref{sec:quasi-Hitchin} that this notion of boundary is well-suited for the study of quasi-Hitchin representations in $\mathrm{SL}(3,\C)$, as we will be able to make a connection between the boundary map provided by the Anosov property of quasi-Hitchin representations and the boundary at infinity of an equivariant complex Lagrangian minimal surface in $\CH^{2}_{\tau}$, which we introduce in Section \ref{sec:min_Lagrangian}.

\subsection{The bi-complex hyperbolic plane as a holomorphic para-K\"ahler reduction}\label{sec:reduction} Here we want to realize the space $\CH^{2}_{\tau}$ as a holomorphic para-K\"ahler quotient using the classical Marsden-Weinstein theory. In contrast with the classical symplectic reduction theory (\cite{marsden1974reduction},\cite{weinstein1980symplectic}), the ambient space $\C^{3}_{\tau}$ is endowed with a holomorphic para-K\"ahler structure $\mathbf{q}=\mathbf{g}+\tau\bm{\omega}$, where $\mathbf{g}$ and $\bm{\omega}$ are a holomorphic metric and a holomorphic symplectic form, respectively. Therefore, the procedure will have two main differences: the first one regards the (real) para-K\"ahler reduction, namely the fact that we want to induce a pseudo-Riemannian metric to the quotient, and this requires further verification (\cite[Theorem B.6]{rungi2023pseudo}); the second one is that the structure we want to induce must be holomorphic, so we are looking for a holomorphic Hamiltonian action of a complex Lie group on $\C^{3}_{\tau}$ whose moment map is also holomorphic (see \cite[\S 7.5]{kobayashi1987differential} for more details). \\ \\
Let us consider the function $f:\C^{3}_{\tau} \rightarrow \C$ given by $f(z)=\mathbf{q}(z,z)$. We denote by $\mathcal U$ the subgroup of $\C_{\tau}$ of bi-complex numbers with unitary $\tau$-norm. We notice that the action of $\mathcal U$ by multiplication is holomorphic and preserves the level sets of the function $f$: indeed, if $u\in\mathcal U$, then
\[
    f(uz)=\mathbf{q}(uz,uz)=|u|_{\tau}\mathbf{q}(z,z)= f(z) \ .
\]
Since $u=u_{1}+\tau u_{2}\in \mathcal{U}$ satisfies $u_{1}^{2}-u_{2}^{2}=1$, we can parameterize $\mathcal{U}$ using complex hyperbolic trigonometric functions
\[
    \mathcal{U}=\{ \cosh(v)+\tau\sinh(v) \ | \ v \in \C, \ \Ima(v) \in [0, 2\pi) \} \ .
\]
We observe that the complex metric $\mathbf{g}=\Ree_{\tau}(\mathbf{q})$ restricted to the orbits of $\mathcal{U}$ on the non-zero level sets of $f$ is non-degenerate: indeed, for a $z\in f^{-1}(k)$, we have 
\[
    \frac{d}{dv}_{|_{v=0}} (\cosh(v)+\tau \sinh(v))z = \tau z 
\]
and $\mathbf{g}(\tau z, \tau z) = \Ree_{\tau}(\mathbf{q}(\tau z, \tau z)) = -k \neq 0$. A straightforward computation shows that the action of $\mathcal{U}$ is Hamiltonian with respect to the holomorphic symplectic form $\hat{\omega}$ with moment map $\mu(z)=\frac{1}{2}f(z)$. Indeed, we already computed that the infinitesimal generator of the action of $\mathcal{U}$ is the vector field $U(z)=\tau z$ and for all $v\in T_{z}\mathbb{C}_{\tau}^{3}$ we have
\begin{align*}
    \bm{\omega}(U(z), v) &= \Ima_{\tau}(\mathbf{q}(\tau z, v)) \\ 
    &= \Ree_{\tau}(\mathbf{q}(z,v)) \\
    &= \frac{1}{2} \frac{d}{dt}_{|_{t=0}} \mathbf{q}(z+tv, z+tv) = d\mu_{z}(v) \ .
\end{align*}
Therefore, holomorphic symplectic reduction theory (\cite[\S 7.5]{kobayashi1987differential}) gives, for all $ k \in \C^{*}$, a non-degenerate holomorphic symplectic form $\hat{\omega}_{k}$ on the quotient $f^{-1}(k)/\mathcal{U}$. In particular, the bi-complex hyperbolic space $\CH^{2}_{\tau}=f^{-1}(-1)/\mathcal{U}$ has one that we denote by $\hat{\omega}:=\hat{\omega}_{-1}$. Moreover, from the decomposition
\[
    T_{z}\C^{3}_{\tau} = T_{z}\CH^{2}_{\tau} \oplus T(\mathcal{U} \cdot z) \oplus \tau T(\mathcal{U} \cdot z) \ ,
\]
we see that the para-complex structure given by the multiplication by $\tau$ descends to a para-complex structure $\mathbf{P}$ on $\CH^{2}_{\tau}$, such that $\hat{\omega}(\mathbf{P}\cdot, \mathbf{P}\cdot)=-\hat{\omega}$. Hence, the holomorphic metric $\mathbf{g}=\bm{\omega}(\cdot, \mathbf{P} \cdot)$ also descends to $\CH^{2}_{\tau}$ to a holomorphic metric $\hat{g}$, so that the triple $(\hat{\omega}, \mathbf{P}, \hat{g})$ defines a holomorphic para-K\"ahler structure on $\CH^{2}_{\tau}$. We note, moreover, that vectors tangent to the orbit of the action of $\mathcal{U}$ are $\bm{\omega}$-orthogonal to vectors in the kernel of $d\mu$, in other words to vectors tangent to a level set of $f$. 


\subsection{Remarkable submanifolds in $\mathbb{CH}_{\tau}^2$}\label{sec:submanifolds} Our interest in the bi-complex hyperbolic plane arose as a way of giving a unified point of view for the study of surface group representations into $\mathrm{SL}(3,\C)$. It is well-known that, if $S$ denotes a closed, connected, oriented surface of genus at least $2$, to any Hitchin representation $\rho:\pi_{1}(S) \rightarrow \SL(3,\R)$, there is a unique $\rho$-equivariant affine sphere in $\R^{3}$. As we will see later in Section \ref{sec:hyp_affine_spheres} and partially by work of Hildebrand (\cite{hildebrand2011cross}), hyperbolic affine spheres in $\R^{3}$ are equivalent to minimal Lagrangian surfaces in the (real) para-complex hyperbolic space $\mathbb{H}^{2}_{\tau}:=\CH^{2}_{\tau}\cap \R^{3}_{\tau}$, which is a para-K\"ahler, totally real, and totally geodesic submanifold in $\CH^{2}_{\tau}$. In particular, if $\hat{q}=\hat{g}+\tau\hat{\omega}$ denotes the holomorphic para-Hermitian structure on $\CH^{2}_{\tau}$, then its (real) para-Hermitian restriction to $\mathbb{H}^{2}_{\tau}$ is given by $\Ree(\hat{g})|_{\mathbb{H}^{2}_{\tau}}+\tau\Ree(\hat{\omega})|_{\mathbb{H}^{2}_{\tau}}$ where $\Ree(\hat{g})|_{\mathbb{H}^{2}_{\tau}}$ is a pseudo-Riemannian metric of neutral signature $(2,2)$. \\

\noindent On the other hand, if $\varrho:\pi_{1}(S) \rightarrow \mathrm{SU}(2,1)$ is sufficiently close to a Fuchsian representation (i.e. to one that is faithful and discrete and factors through $\mathrm{SO}_0(2,1)$), there is a unique $\varrho$-equivariant minimal Lagrangian surface in the complex hyperbolic space $\CH^{2}$. This can also be found as a totally geodesic submanifold of $\CH^{2}_{\tau}$ by considering the intersection $\CH^{2}_\tau \cap (\R \oplus i\tau \R)^3$, where $\R \oplus i\tau \R$ denotes the subspace of bi-complex numbers in which the $\tau$-real part is a real number and the $\tau$-imaginary part is a purely imaginary complex number. Note that the complex structure on this copy of $\CH^{2}$ is inherited from $\mathbf{J}$. Let us now explain how the $\C$-linear para-Hermitian structure $\mathbf{q}$ restricts to a (real) $\mathbf{J}$-Hermitian structure on $\CH^{2}$, which induces the usual K\"ahler metric. First we decompose $\mathbf{q}$ with respect to the real basis $\{1,i, j,\tau\}$ of the algebra of bi-complex numbers
\begin{align*}
    \mathbf{q}&=\Ree(\mathbf{g})+i\Ima(\mathbf{g})+\tau\Ree(\bm{\omega})+\tau i\Ima(\bm\omega) \tag{$j=\tau i$} \\ 
 &=\big(\Ree(\mathbf{g})+j\Ima(\bm{\omega})\big)+\tau(\Ree(\bm{\omega})+j\Ima(\mathbf{g})) \\ 
 &=:\tilde{\mathbf{g}}+\tau\tilde{\bm{\omega}} \ .
\end{align*}
It is then clear from the definition of $\CH^2$ as the intersection of $\CH^2_\tau$ with $(\R \oplus \tau i\R)^3$, that $\tilde{\bm\omega}$ vanishes identically when restricted to $\CH^{2}$. 
Then we see that $\tilde{\mathbf{g}}$ defines a $\J$-Hermitian form because
\begin{align*}
    \tilde{\mathbf{g}}(\J\cdot,\J\cdot)&=\Ree(\mathbf{g})(\J\cdot,\J\cdot)+j\Ima(\bm\omega)(\J\cdot,\J\cdot) \\ &=\Ree(\mathbf{g})(\I\mathbf{P}\cdot,\I\mathbf{P}\cdot)+j\Ima(\bm\omega)(\I\mathbf{P}\cdot,\I\mathbf{P}\cdot) \\ &=-\Ree(\mathbf{g})(\mathbf{P}\cdot,\mathbf{P}\cdot)-j\Ima(\bm\omega)(\mathbf{P}\cdot,\mathbf{P}\cdot) \tag{$\mathbf{q}(\I\cdot,\I\cdot)=-\mathbf{q}$} \\ &=\Ree(\mathbf{g})+j\Ima(\bm\omega)=\tilde {\mathbf{g}} \tag{$\mathbf{q}(\mathbf{P}\cdot,\mathbf{P}\cdot)=-\mathbf{q}$}
\end{align*}
From this discussion, we conclude that the $\J$-Hermitian structure on $\CH^2$ is indeed induced from the restriction of the bi-complex structure of the ambient space.\\

\noindent There is another interesting submanifold to consider in $\CH^2_\tau$, which is found by looking at the intersection $\CH^2_\tau\cap \C^{3}$. This space, denoted by $\mathbb X$, is endowed with the complex structure $\mathbf I$ and a holomorphic Riemannian metric $\mathbf g$ given by the restriction of the $\C$-linear para-Hermitian form $\mathbf q$. Explicitly it is realized as the hypersurface $\{\underline z\in\C^3 \ | \ z_1^2+z_2^2-z_3^2=-1\}$ quotient out by all the complex numbers that square to one, namely $\{\pm 1\}$. Since all non-degenerate bilinear form over $\C$ are conjugate to each other, this space is isometric to the space of (maximal, oriented, unparametrized) geodesics of the three-dimensional real hyperbolic space $\mathbb H^3$ first studied by Bonsante and El-Emam (\cite{bonsante2022immersions}) and then by El-Emam and Seppi (\cite{el2022gauss}), which is further identified with $\C\mathbb P^1\times\C\mathbb P^1\setminus\Delta$, being $\Delta$ the diagonal. The subgroup of $\SL(3,\C)$ that preserves $\mathbb{X}$ is isomorphic to $\SO(3,\C)$ and acts by holomorphic isometries for its holomorphic Riemannian metric. A special class of immersed surfaces in $\mathbb X$, called \emph{Bers embeddings}, has been studied in \cite{bonsante2022immersions}, whose holonomy representation takes value in $\SL(2,\C) \cong \SO(3,\C)$: the theory we will describe in Section \ref{sec:min_Lagrangian} will generalize part of their work as well.  \\

\noindent We will see in Section \ref{sec:min_Lagrangian} that hyperbolic affine spheres in $\R^{3}$, minimal Lagrangian surfaces in $\CH^{2}$ and Bers embeddings in $\mathbb X$ are instances of more general minimal Lagrangian surfaces in $\CH^{2}_{\tau}$, thus giving a unified point of view to these seemingly different notions.

\subsection{Differential geometry of $\CH^{2}_{\tau}$}\label{sec:bicomplexdiffgeometry}
We describe here the main geometric features of $\CH^2_\tau$ that will be useful in the rest of the paper. \\

\noindent Recall that $\CH^2_\tau$ is a complex manifold of dimension $4$. We denoted by $\mathbf{I}$ its standard complex structure, seen as a bundle endomorphism of $T\CH^2_\tau$ such that $\mathbf{I}^2=-\mathrm{Id}$. The bi-complex hyperbolic space is also endowed with a holomorphic metric $\hat{g}$. This induces a natural connection $\hat{D}$ of $T\CH^2_\tau$ by Theorem \ref{thm:connection}. Considering the $\mathbf{q}$-orthogonal decomposition
\[
    T_{z}\C^{3}_{\tau} = \C_{\tau} \cdot z \oplus z^{\perp}\ ,
\]
for $z\in \C^{3}_{\tau}$ such that $\mathbf{q}(z,z)=-1$, after identifying $T_{z}\CH^{2}_{\tau}$ with $z^{\perp}$, it is straightforward to check, using the properties that characterize uniquely the Levi-Civita connection of $\hat{g}$, that $\hat{D}$ is exaclty the projection onto $z^{\perp}$ of the standard flat connection $\mathbf{D}$ defined on $T\C_{\tau}^{3}$. 

\begin{lemma}\label{lm:curvature_bicomplex} The bi-complex hyperbolic space has constant para-holomorphic complex sectional curvature $-2$.
\end{lemma}
\begin{proof} Let $z\in \C^{3}_{\tau}$ such that $\mathbf{q}(z,z)=-1$. By the previous discussion, the $\mathbf{q}$-orthogonal splitting $T_{z}\C^{3}_{\tau} = T_{z}\CH^{2}_{\tau} \oplus \C_{\tau} \cdot z$ allows to decompose the flat connection $\mathbf{D}$ as 
\begin{align*}
    \mathbf{D}_{X} Y &= \hat{D}_{X}Y + \hat{\II}(X,Y) \\
    \mathbf{D}_{X}\xi &= -\hat{B}_{\xi} X + \hat{\nabla}^{N}_{X}\xi
\end{align*}
where $X,Y\in \Gamma(T\CH_{\tau})$ and $\xi \in \Gamma(\C_{\tau} \cdot z)$. Therefore, if $\mathbf{R}$ and $\hat{R}$ denote the Riemann curvature tensor of $\mathbf{D}$ and $\hat{D}$ respectively, we have, for all non-isotropic $X\in \Gamma(T_{z}\CH^{2}_{\tau})$, 
\begin{align*}
    0 &= \mathbf{g}(\mathbf{R}(X,\mathbf{P}X)X, \mathbf{P}X)\\ &= \mathbf{g}(\mathbf{D}_{\mathbf{P}X}\mathbf{D}_{X}X-\mathbf{D}_{X}\mathbf{D}_{\mathbf{P}X}-\mathbf{D}_{[\mathbf{P}X,X]}X, \mathbf{P}X) \\
    &= \hat{g}(\hat{D}_{\mathbf{P}X}\hat{D}_{X} X - \hat{D}_{X}\hat{D}_{\mathbf{P}X}X - \hat{D}_{[\mathbf{P}X,X]}, \mathbf{P}X) \\
    & \ \ \ +\mathbf{g}(-\hat{B}_{\hat{\II}(X,X)}\mathbf{P}X, \mathbf{P}X)+\hat{g}(\hat{B}_{\hat{\II}(X,\mathbf{P}X)}, \mathbf{P}X) \\
    &= \hat{g}(\hat{R}(X,\mathbf{P}X)X,\mathbf{P}X)\\
    & \ \ \ - \mathbf{g}(\hat{\II}(X,X), \hat{\II}(\mathbf{P}X,\mathbf{P}X))+\mathbf{g}(\hat{\II}(X,\mathbf{P}X), \hat{\II}(X,\mathbf{P}X)) \\
    &= \mathrm{Sec}_{\hat{g}}(X,\mathbf{P}X)\| X\|^{4} -2\|X\|^{4} 
\end{align*}
and the statement follows. In the last step we use the following identities
\[
    \hat{\II}(X,X)=\|X\|^{2}z  \ \ \ \hat{\II}(X,\mathbf{P}X)=\|X\|^{2}\mathbf{P}z \ \ \ \hat{\II}(\mathbf{P}X,\mathbf{P}X)=-\|X\|^{2}z \ ,
\]
which follow immediately from the definition.
\end{proof}

\begin{remark}\label{rmk:scaling} For reasons that will be apparent later, we re-scale the metric $\hat{g}$ of $\CH^{2}_{\tau}$ so that is has para-holomorphic complex sectional curvature equal to $-4$. Note that, by the discussion after Proposition \ref{prop:constant_parahol_curvature}, this implies that the totally geodesic copy of $\CH^{2}$ inside $\CH^{2}_{\tau}$ defined in Section \ref{sec:submanifolds} also has holomorphic sectional curvature $-4$, which is the most common convention. 
\end{remark}

\noindent From the construction of $\CH^{2}_{\tau}$ as holomorphic para-K\"ahler reduction (see Section \ref{sec:reduction}), the bi-complex hyperbolic space comes equipped with an integrable para-holomorphic structure $\mathbf{P}$ and a symplectic form $\hat{\omega}$, which are related to the holomorphic metric $\hat{g}$ by the equation $\hat{\omega}=\hat{g}(\cdot, \mathbf{P}\cdot)$. Note that, since the symplectic form $\hat{\omega}$ is closed and the para-complex structure $\mathbf{P}$ is integrable, we can conclude, as in the classical setting of K\"ahler manifolds, that $\hat{D}\mathbf{P}=0$. In particular, the eigendistributions $\hat{\mathcal{D}}_{\pm}$ relative to the $\pm 1$ eigenvalues of $\mathbf{P}$ not only are integrable but are also $\hat{D}$-parallel.

\section{Complex Lagrangian minimal surfaces in \texorpdfstring{$\mathbb{CH}_{\tau}^2$}{CH}}\label{sec:min_Lagrangian}
\noindent In this section we study smooth immersions $\sigma: U \rightarrow \CH^{2}_{\tau}$ of a simply connected $2$-dimensional domain $U$. As an immersion of smooth manifolds, $\sigma$ has codimension $6$. However, we use recent techniques introduced by Bonsante and El-Emam (\cite{bonsante2022immersions}) to treat them as surfaces of (complex) codimension $2$. 

\subsection{Embedding data and structural equations}\label{sec:embedding_data} Let $\sigma:U \rightarrow \CH^{2}_{\tau}$ be as above. We extend the differential of $\sigma$ by $\C$-linearity to obtain a map 
\begin{align*}
    \sigma_{*}:\C TU &\rightarrow T\CH^{2}_{\tau} \\
        V+iW &\mapsto d\sigma(V)+\mathbf{I}d\sigma(W)  \ .
\end{align*}
The pull-back $\sigma^{*}\hat{g}$ is a $\C$-bilinear form on $\C TU$. We denote it by $h$ and we call it \textit{complex-valued first fundamental form}. From now on, we will consider only immersions for which $h$ is non-degenerate. \\

\noindent Consider now the pull-back vector bundle $\sigma^{*}(T\CH^{2}_{\tau})$. This is a complex vector bundle over $U$ with complex structure given by
\[
    i\sigma^{*}(V):=\sigma^{*}(\mathbf{I}V) \ \ \  V\in \Gamma(T\CH^{2}_{\tau}) \ .
\]
We endow $\sigma^{*}(T\CH^{2}_{\tau})$ with the pull-back complex bilinear form $g:=\sigma^{*}\hat{g}$. Under the non-degenerate assumption, we can consider $\C TU$ as a complex subbundle of $\sigma^{*}(T\CH^{2}_{\tau})$. Note that the restriction of $g$ to $\C TU$ is exactly the complex-valued first fundamental form $h$. We have a $g$-orthogonal decomposition
\[
    \sigma^{*}(T\CH_{\tau}):=\C TU \oplus \C NU \ ,
\]
where $\C NU$ is, by definition, the $g$-orthogonal complement of $\C TU$ in $\sigma^{*}(T\CH^{2}_{\tau})$. It is a complex bundle of complex rank $2$, which plays the same role as the normal bundle to the immersion. The Levi-Civita connection $\hat{D}$ on $T\CH^{2}_{\tau}$ pulls back to an $\R$-bilinear connection $D$ on $\sigma^{*}(T\CH^{2}_{\tau})$. We can extend it by $\C$-linearity by setting
\[
        D_{\sigma^{*}(V)+i\sigma^{*}(V')}\sigma^{*}(W):=D_{\sigma^{*}(V)}\sigma^{*}(W)+iD_{\sigma^{*}(V')}\sigma^{*}(W)
\]
for all $V,V',W \in \Gamma(T\CH^{2}_{\tau})$. Since $\hat{D}\mathbf{I}=0$, the connection $D$ is actually $\C$-bilinear because
\[
    D_{\sigma^{*}(V)}(i\sigma^{*}(W))=D_{\sigma^{*}(V)}(\sigma^{*}(\mathbf{I}W))=\hat{D}_{V}(\mathbf{I}W)=\mathbf{I}\hat{D}_{V}W=iD_{\sigma^{*}(V)}\sigma^{*}(W) \ ,
\]
for all $V,W \in \Gamma(T\CH^{2}_{\tau})$. It thus makes sense to write $D_{X}Y$ for any $X,Y \in \Gamma(\C TU)$ and we can decompose it into a tangential and normal component
\[
    D_{X}Y = \nabla_{X} Y + \II (X,Y) \ .
\]
It is straightforward to check that $\nabla$ is the Levi-Civita connection of $h$ and $\II$ is a $\C$-bilinear symmetric tensor, called \textit{complex second fundamental form}. 

\begin{defi} We say that an immersion $\sigma: U \rightarrow \CH^{2}_\tau$ is \textit{complex Lagrangian} if the pull-back $\sigma^{*}\hat{\omega}$ of the holomorphic symplectic form on $\C TU$ vanishes identically. 
\end{defi}

\noindent 
In particular, if $\sigma:U \rightarrow \CH^{2}_\tau$ is a complex Lagrangian immersion, then its complex-valued first fundamental form $h$ can be obtained by restriction of the pull-back $q:=\sigma^{*}\hat{q}$, since $\hat{q}=\hat{g}+\tau \hat{\omega}$. In this case the non-degeneracy of $h$ can be verified by looking at the relative position between $\C TU$ and the distributions $\mathcal{D}^{\pm}:=\sigma^{*}(\hat{\mathcal{D}}_{\pm})$:

\begin{lemma}\label{lm:transverse} Let $\sigma: U \rightarrow \CH^{2}_{\tau}$ be a smooth complex Lagrangian immersion. Then $h=\sigma^{*}\hat{g}$ is non-degenerate if and only if the distributions $\mathcal{D}_{\pm}$ are transverse to $\C TU$. 
\end{lemma}
\begin{proof} Since $\sigma^{*}(T\CH^{2}_{\tau})$ can be written as direct sum of $\mathcal{D}_{\pm}$, the subspace $\C TU$ is transverse to both $\mathcal{D}_{\pm}$ if and only if $\C TU$ is a graph of a $\C$-linear invertible map $\mathcal{L}:\mathcal{D}_{+} \rightarrow \mathcal{D}_{-}$. Since $\mathcal{D}_{\pm}$ are the eigenspaces of the para-complex structure $\mathbf{P}$, they are both totally isotropic for the metric $g$. However, because $g$ is non-degenerate, we can find a basis $\{e_{1}^{+}, e_{2}^{+}, e_{1}^{-}, e_{2}^{-}\}$ of $\sigma^{*}(T \CH^{2}_{\tau})$ as complex vector space such that $e_{i}^{\pm} \in \mathcal{D}_{\pm}$ and $g(e_{i}^{+}, e_{j}^{-})=\delta_{ij}$. By construction, we then have $\C TU=\mathrm{Span}( e_{1}^{+}+\mathcal{L}e_{1}^{+}, e_{2}^{+}+\mathcal{L}e_{2}^{+})$. Let $L$ be the matrix associated to $\mathcal{L}$ with respect to the basis $\{e_{1}^{+}, e_{2}^{+}\}$ in the domain and $\{e_{1}^{-}, e_{2}^{-}\}$ in the target. It is straightforward to verify that the restriction of $g$ and $\omega$ to $\C TU$ are represented by the matrices $L+L^{t}$ and $L-L^{t}$, respectively. Since $\sigma$ is a complex Lagrangian immersion, we deduce that $L$ is symmetric and thus $L$ is invertible if and only if the restriction $h$ of $g$ to $\C TU$ is non-degenerate.
\end{proof}

\begin{lemma}\label{lm:normal_Lagrangian} Let $\sigma:U \rightarrow \CH^{2}_{\tau}$ be a complex Lagrangian immersion. Then the restriction of the para-complex structure $\mathbf{P}$ induces a $\C$-linear isomorphism $\mathbf{P}: \C TU \rightarrow \C NU$. The same holds for the complex structure $\mathbf{J}$.
\end{lemma}
\begin{proof} Since $\mathbf{P}$ is an involution, it is sufficient to check that 
$\mathbf{P}X \in \C NU$ for all $X \in \C TU$. Now, by the previous remark, for all $Y \in \C TU$, we have
\[
    g(\mathbf{P}X,Y)=\Ree_{\tau}(q(\mathbf{P}X,Y))=\Ree_{\tau}(\tau q(X,Y))=\Ree_{\tau}(\tau g(X,Y))=0  \ .
\]
Since $\mathbf{J}=\mathbf{P}\mathbf{I}$ and $\mathbf{I}$ is an automorphism of $\C TU$, the same conclusion holds for $\mathbf{J}$.
\end{proof}

\noindent The Levi-Civita connection $D$ on $\sigma^{*}(T\CH^{2}_{\tau})$ induces a connection $\nabla^{N}$ on the normal bundle $\C NU$ by decomposing
\[
    D_{X} \mathbf{P}Y = -B_{\mathbf{P}Y}X + \nabla^{N}_{X} \mathbf{P}Y \ ,
\]
where $-B_{\mathbf{P}Y}X$ is by definition the tangential component of $D_{X} \mathbf{P}Y$ for all $X, Y \in \Gamma(\C TU)$. The \textit{complex shape operator} $B: \C TU \times \C NU \rightarrow \C TU$ is complex bi-linear and, as a consequence of the compatibility of the connection $D$ with the metric $g$, it satisfies
\begin{equation}\label{eq:II-B}
    g(\II(X,Y) , \mathbf{P}Z) = g(B_{\mathbf{P}Z}X,Y)
\end{equation}
for all $X,Y,Z \in \Gamma(\C TU)$.

\begin{defi} Let $\sigma: U \rightarrow \CH^{2}_{\tau}$ be a smooth immersion. We say that $\sigma$ is \textit{minimal} if and only if its complex second fundamental form $\II$ is traceless, namely given an $h$-orthonormal basis $\{X_1,X_2\}$ of $\mathbb CTU$ we have $\II(X_1,X_1)=-\II(X_2,X_2)$.
\end{defi}


    



\begin{prop}\label{prop:rel_tensors}
Let $\II$ and $B$ be the complex second fundamental form and complex shape operator of a smooth, complex Lagrangian minimal immersion into $\CH^{2}_{\tau}$. Then the following identities hold:
\begin{enumerate}[i)]
    \item $\mathbf{P}(\II(X,Y)) = -B_{\mathbf{P}Y} X$ 
    \item $\mathbf{P}(\nabla_{X}Y)= \nabla^{N}_{X} \mathbf{P}Y$ \ ,
\end{enumerate}
for all $X,Y,Z \in \Gamma(\C TU)$. Moreover, the tensors defined by 
\[
    C^\pm(X,Y,Z):=\pm g(\II(X,Y), \mathbf{P}Z)
\]
are totally symmetric. 
\end{prop}
\begin{proof}
Recall that, by definition of the Levi-Civita connection $\nabla$ and the complex second fundamental form, we have $D_{X}Y=\nabla_{X}Y + \II(X,Y)$ for all $X,Y \in \Gamma(\C TU)$. Therefore,
\begin{align*}
    \mathbf{P}( \nabla_{X}Y+\II(X,Y)) & = \mathbf{P}(D_{X}Y) \\
    & = D_{X}(\mathbf{P}Y) \tag{$D\mathbf{P}=0$} \\
    & = -B_{\mathbf{P}Y}X + \nabla^{N}_{X} \mathbf{P}Y .
\end{align*}
By Lemma \ref{lm:normal_Lagrangian}, the operator $\mathbf{P}$ sends tangent vectors to normal vectors and viceversa, thus, by transversality of $\C TU$ and $\C NU$ we can conclude that $\mathbf{P}(\II(X,Y))=-B_{\mathbf{P}Y} X$ and $\mathbf{P}(\nabla_{X}Y)=\nabla^{N}_{X} \mathbf{P}Y$ \ . \\
As for the tensor $C^+$, it is clear that $C^+$ is symmetric in the first two entries, because the complex second fundamental form is a symmetric operator. It remains to show that $C^+$ is symmetric in the second and third entry. This follows from an explicit computation:
\begin{align*}
    C^+(X,Z,Y)&= g(\II(X,Z), \mathbf{P}Y) \\
    &= - g(\mathbf{P}(\II(X,Z)), Y) \tag{$g(\mathbf{P}, \mathbf{P}) = -g$} \\
    & = g(B_{\mathbf{P}Z}X,Y) \tag{relation $i)$}\\
    & = g(\II(X,Y), \mathbf{P}Z) = C^+(X,Y,Z) \ . \tag{Equation \eqref{eq:II-B}}\ 
\end{align*}
\end{proof}

\noindent Because we are only considering immersions for which the complex-valued first fundamental form is non-degenerate, Lemma \ref{lm:transverse} furnishes two decompositions of $\sigma^{*}(T\CH^{2}_{\tau})$:
\[
    \sigma^{*}(T\CH^{2}_{\tau})=\C TU \oplus \mathcal{D}_{+} = \C TU \oplus \mathcal{D}_{-} \ .
\]
These induce two $\C$-bilinear connections $\nabla^{\pm}$ on $\C TU$ by splitting the ambient connection $D$ into a tangential component and a part $\alpha^{\pm}$ belonging to $\mathcal{D}_{\pm}$: in other words,

\begin{equation}\label{eq:decomposition_nablaplusminus}    D_{X}Y=\nabla^{+}_{X}Y+\alpha^{+}(X,Y)=\nabla^{-}_X Y +\alpha^{-}(X,Y)\end{equation}
for all $X,Y \in \C TU$. Note that $\alpha^{\pm}$ are symmetric and $\C$-bilinear. Another important property of the distributions $\mathcal D_{\pm}$ is closely related to the holomorphic para-Hermitian structure $\hat{q}$ of $\CH^2_{\tau}$. In this regard, we introduce the following new tensor $\hat{Q}:=\hat{g}+\hat{\omega}$ whose symmetric part coincides with $\hat{g}$ and the skew-symmetric part with $\hat{\omega}$. In particular, it is parallel with respect to the Levi-Civita connection of $\hat{g}$, i.e. $\hat{D}\hat{Q}=0$. Let $Q=\sigma^{*}\hat{Q}$ be the corresponding tensor on $\sigma^{*}(T\CH^2_\tau)$.

\begin{lemma}
For any $X\in\mathcal D_-, Y\in\mathcal D_{+}$ and $Z\in \sigma^{*}(T\CH^2_\tau)$ we have \begin{equation}\label{eq:KerQandD}
    Q(X,Z)=Q(Z,Y)=0 \ . 
\end{equation}
Moreover, if $X,Y\in \sigma^{*}(T\CH^2_\tau)$ satisfy $Q(X,Z)=Q(Z,Y)=0$ for any $Z\in \sigma^{*}(T\CH^2_\tau)$, then $X\in\mathcal D_-$ and $Y\in\mathcal D_{+}$.
\end{lemma}
\begin{proof}
Let $X\in\mathcal D_-$ and $Z\in \sigma^{*}(T\CH^2_\tau)$, then \begin{align*}
    Q(X,Z)&=g(X,Z)+\omega(X,Z) \\ &=g(X,Z)+g(\mathbf PX,Z) \tag{$\omega(\cdot,\cdot)=g(\mathbf{P}\cdot,\cdot)$} \\ &=g(X,Z)-g(X,Z)=0 \ . \tag{$\mathbf PX=-X$}
\end{align*}
Observing that $Y\in\mathcal D_{+}$ if and only if $\mathbf PY=Y,$ with a similar computation and using that $g(\mathbf P\cdot,\cdot)=-g(\cdot,\mathbf P\cdot)$, we obtain $Q(Z,Y)=0$. 
\newline Finally, suppose that given $X,Y\in \sigma^{*}(T\CH^2_\tau)$ we have $Q(X,Z)=Q(Z,Y)=0$ for any $Z\in \sigma^{*}(T\CH^2_\tau)$. This is equivalent to $$g(X+\mathbf PX,Z)=g(Z,Y-\mathbf PY)=0 \ , \ \ \ \forall Z\in \sigma^*(T\CH^2_\tau) \ ,$$ and since the complex metric $g$ is non-degenerate, we conclude that $\mathbf PX=-X$ and $\mathbf PY=Y$ as required.
\end{proof}
\noindent The above result has direct application to our case of interest.
\begin{lemma} \label{lm:dual_connection} Let $\sigma:U \rightarrow \CH^{2}_\tau$ be a complex Lagrangian immersion with non-degenerate complex-valued first fundamental form $h$. Then
\begin{enumerate}[i)]
    \item the tensors 
    \[
        T^{\pm}(X,Y,Z):=(\nabla^{\pm}_{X} h)(Y,Z)
    \]
    are totally symmetric;
    \item the Levi-Civita connection $\nabla$ of $h$ satisfies
    \[
        \nabla=\frac{1}{2}(\nabla^{+}+\nabla^{-})
    \]
\end{enumerate}
\end{lemma}
\begin{proof}
Let us start with point $i)$: the symmetry of $T^+(X,Y,Z)$ in the second and third entries follows from its definition, in fact for any vector field $X$, the tensor $\nabla^+_X h(\cdot,\cdot)$ is still symmetric. On the other hand, since $\sigma$ is a Lagrangian immersion we have $Q=h$, therefore \begin{align*}
    T^+(X,Y,Z)&=(\nabla^+_X Q)(Y,Z) \\ &=X\cdot Q(Y,Z)-Q(\nabla^+_XY,Z)-Q(Y,\nabla^+_XZ) \\ &=X\cdot Q(Y,Z)-Q(D_XY-\alpha^+(X,Y),Z)-Q(Y,D_XZ-\alpha^+(X,Z)) \\ &= Q(\alpha^+(X,Y),Z)+Q(Y,\alpha^+(X,Z)) \tag{$DQ=0$} \\ &= Q(\alpha^+(X,Y),Z) \ , \tag{$\alpha^+(X,Z)\in\mathcal D_+$ and Equation (\ref{eq:KerQandD})}
\end{align*} in particular $T^+$ is also symmetric in the first and second entries. The proof for the symmetry of $T^-$ is analogous.\newline Regarding point $ii)$ we already know that the connection $\frac{1}{2}(\nabla^++\nabla^-)$ is torsion-free, thus we only need to prove that it preserves the induced metric. First notice that \begin{align*}
    X\cdot h(Y,Z)&=X\cdot Q(Y,Z) \\ &=Q(D_XY,Z)+ Q(Y,D_XZ) \tag{$DQ=0$} \\ &= Q(\nabla_X^-Y+\alpha^-(X,Y),Z)+Q(Y,\nabla_X^+Z+\alpha^+(X,Z)) \\ &= Q(\nabla_X^-Y,Z)+ Q(Y,\nabla_X^+Z) \tag{Equation (\ref{eq:KerQandD})} \\ &= h(\nabla_X^-Y,Z)+h(Y,\nabla_X^+Z) \ .
\end{align*}Then, finally \begin{align*}
    \frac{1}{2}(\nabla_X^++\nabla_X^-)(h)(Y,Z)&=\frac{1}{2}(X+X)\cdot h(Y,Z)-\frac{1}{2}h(\nabla_X^+Y+\nabla_X^-Y,Z)+ \\ \ \ \ \  &-\frac{1}{2}h(Y,\nabla_X^+Z+\nabla_X^-Z) \\ &=X\cdot h(Y,Z)-X\cdot h(Y,Z)=0 \ ,
\end{align*}where in the last line we used the above computation.

\end{proof}

\noindent As a consequence, the tensors $C^\pm$ defined in Proposition \ref{prop:rel_tensors} can be written in terms of the connections $\nabla^{\pm}$:

\begin{cor}\label{cor:C_and_nablas} Under the same assumptions as Lemma \ref{lm:dual_connection} we have
\[
    2C^\pm(X,Y,Z)=T^{\pm}(X,Y,Z)
\]
\end{cor}
\begin{proof} We prove the relation for $T^{+}$; the one for $T^{-}$ easily follows using Lemma \ref{lm:dual_connection} part $ii)$ since $\nabla h =0$. \\
Averaging the decompositions
\[
    D_{X}Y=\nabla^{+}_{X}Y+\alpha^{+}(X,Y)=\nabla^{-}_{X}Y+\alpha_{-}(X,Y) \ ,
\]
we deduce that for all $X,Y \in \Gamma(\C TU)$ we have
\begin{align*}
    D_{X}Y&=\frac{1}{2}(\nabla^{+}_{X}Y+\nabla^{-}_{X}Y)+\frac{1}{2}(\alpha^{+}(X,Y)+\alpha_{-}(X,Y)) \\
    &=\nabla_{X}Y+\frac{1}{2}(\alpha^{+}(X,Y)+\alpha_{-}(X,Y)) \tag{Lemma \ref{lm:dual_connection} part $ii)$} \ .
\end{align*}
Therefore, $\frac{1}{2}(\alpha^{+}(X,Y)+\alpha_{-}(X,Y))$ is the normal component of $D_{X}Y$ and thus coincides with $\II(X,Y)$. \\
Let us now consider the $\mathrm{End}(\C TU)$-valued 1-form
\[
    A^{+}:=\nabla^{+}-\nabla \ .
\]
Then, for all $X,Y,Z \in \Gamma(\C TU)$, 
\begin{align*}
    C^+(X,Y,Z)&= -h(\mathbf{P}(\II(X,Y)),Z) \\
            &= -\frac{1}{2}h(\mathbf{P}(\alpha^{+}(X,Y))+\mathbf{P}(\alpha^{-}(X,Y))) \\
            &= -\frac{1}{2}h(\alpha^{+}(X,Y)-\alpha^{-}(X,Y)) \\
            & = \frac{1}{2}h(D_{X}Y-\nabla^{+}_{X}Y-D_{X}Y+\nabla^{-}_{X}Y,Z) \\
            &=-\frac{1}{2}h(-\nabla^{+}_{X}Y+2\nabla_{X}Y-\nabla^{+}_{X}Y,Z) \tag{$\nabla^{-}=2\nabla-\nabla^{+}$} \\
            &=g(A^{+}(X)Y,Z) \ .
\end{align*}
Since $C^+$ is totally symmetric, it follows that $A^{+}(X)Y=A^{+}(Y)X$ for all $X,Y\in \C TU$ and that $A^{+}(X)$ is $h$-self-adjoint for all $X \in \C TU$. Therefore, 
\begin{align*}
    T^{+}(X,Y,Z)&=(\nabla^{+}_{X}h)(Y,Z)\\
    &=(\nabla^{+}_{X}-\nabla_{X})(h)(Y,Z) \tag{$\nabla h=0$} \\
    &= h(A^{+}(X)Y,Z)+h(Y,A^{+}(X)Z) = 2C^+(X,Y,Z) \ .
\end{align*}
\end{proof}

\begin{remark}\label{rmk:A_and_II} From the relation $C^+(X,Y,Z)=h(A^{+}(X)Y,Z)$ proved in Corollary \ref{cor:C_and_nablas}, we deduce that $A^{+}=-\mathbf{P}\II$. 
\end{remark}

\begin{theorem}\label{thm:Gauss_Codazzi} The complex-valued first fundamental form $h$ and the complex second fundamental form $\II$ of a smooth complex Lagrangian minimal immersion satisfy
\begin{align*}
    &d^{\nabla}\II = 0 \ \ \ \ \ \ \ \ \ \ \ &\text{(Codazzi equation)} \\
    &K_{h}+\| \II \|^{2} = -1 \ \ \ \ \ &\text{(Gauss equation),}
\end{align*}
where $K_{h}$ denotes the complex Gaussian curvature of the metric $h$ and the norm of $\II$ is defined as
\[
    \| \II \|^{2} = g(\II(e_{1},e_{1}), \II(e_{1},e_{1}))+g(\II(e_{2},e_{1}), \II(e_{2},e_{1})) \ ,
\]
with $\{e_{1}, e_{2}\}$ being an orthonormal basis of $\C TU$.
\end{theorem}
\begin{proof} Let $X,Y \in \Gamma(\C TU)$. The decompositions $D_{X}Y=\nabla_{X}Y+\II(X,Y)$ and $D_{X} \xi = -B_{\xi}X + \nabla^{N}_{X} \xi$ for all $X,Y \in \Gamma(\C TU)$ and $\xi \in \Gamma(\C NU)$ imply that
\begin{align*}
\begin{split}
    R^{D}(X,Y)Z &= D_{X}D_{Y}Z - D_{Y}D_{X}Z - D_{[X,Y]}Z \\
    &= D_{X}(\nabla_{Y}Z+\II(Y,Z))-D_{Y}(\nabla_{X}Z+\II(X,Z))  \\
    & \ \ \ - \nabla_{[X,Y]}Z-\II([X,Y],Z) \\
    &= \nabla_{X}\nabla_{Y}Z + D_{X}(\II(Y,Z)) + \II(X, \nabla_{Y}Z)  \\
    & \ \ \ - \nabla_{Y}\nabla_{X}Z-D_{Y}( \II(X,Z)) - \II(Y, \nabla_{X}Z) \\
    & \ \ \ - \nabla_{[X,Y]}Z-\II([X,Y],Z) \\
    & = R^{\nabla}(X,Y)Z - B_{\II(Y,Z)}X + B_{\II(X,Z)}Y \\
    & \ \ \ + \nabla^{N}_{X}(\II(Y,Z)) - \nabla^{N}_{Y}(\II(X,Z)) + \II(X, \nabla_{Y}Z) - \II(Y,\nabla_{X}Z) \\
    & \ \ \ - \II(\nabla_{X}Y, Z) + \II(\nabla_{Y}X,Z) \\
    &= \underbrace{R^{\nabla}(X,Y)Z - B_{\II(Y,Z)}X + B_{\II(X,Z)}Y}_{ \in \C TU} + \underbrace{d^{\nabla}\II(X,Y,Z)}_{\in \C NU} \ .\\
\end{split}
\end{align*}
On the other hand, by Proposition \ref{prop:constant_parahol_curvature}, Lemma \ref{lm:curvature_bicomplex} and Remark \ref{rmk:scaling}, we can write
\begin{align*}
    R^{D}(X,Y)Z &=\underbrace{g(X,Z)Y-g(Y,Z)X}_{\in \C TU}+\underbrace{g(X,\mathbf{P}Z)\mathbf{P}Y-g(Y,\mathbf{P}Z)\mathbf{P}X+2g(X,\mathbf{P}Y)\mathbf{P}Z}_{=0} 
\end{align*}
because $\mathbf{P}(\C TU) = \C NU$. Equating the two expressions above, we conclude that $d^{\nabla}\II=0$. Moreover, if $\{X,Y\}$ is an $h$-orthonormal basis of $\C TU$, then
\begin{align*}
    K_{h}&=-R^{\nabla}(X,Y,X,Y) \\
         &=-R^{D}(X,Y,X,Y)-g(B_{\II(Y,X)}X, Y) + g(B_{\II(X,X)}Y,Y) \\
         &= -1 - g(\II(Y,X), \II(Y,X))+g(\II(X,X), \II(Y,Y)) \\
         &= -1 - g(\II(Y,X), \II(Y,X))-g(\II(X,X), \II(X,X)) \tag{$\trace(\II)=0$} \\
         &= -1 - \| \II\|^{2}
\end{align*}
and the proof is complete.
\end{proof}

\begin{remark}\label{rmk:X}
We observe that if the complex Lagrangian minimal surface $\sigma(U)$ is totally geodesic, namely $\II=0$, and the first fundamental form is a positive complex metric $h$ (Definition \ref{def:positivecomplexmetric}), then it satisfies $K_h=-1$. Therefore, it can be realized as the special class of immersed surfaces in $\mathbb X$ studied by Bonsante and El Emam (\cite[Theorem 6.9]{bonsante2022immersions}).
\end{remark}

\subsection{Computations in local coordinates}\label{sec:localcoordinates} We will now restrict to the case where the complex-valued first fundamental form $h$ is a positive complex metric. We saw in Section \ref{sec:positive_complex_metrics} that, under this assumption, we can find a unique compatible $\C$-complex structure $\mathcal{J}$ corresponding to a pair $(J_{1},J_{2})$ of complex structures on $S$ and local holomorphic coordinates $z$ and w on $S$ such that $h=2e^{2\psi}dzd\bar{w}$. This implies that the vector fields $\partial_{\bar{z}}$ and $\partial_{w}$ form an $h$-isotropic basis of $\C TU$. On the other hand, since $\C TU$ can also be generated by $\partial_{z}$ and $\partial_{\bar{z}}$, the vector field $\partial_{w}$ can be written as a $\C$-linear combination of $\partial_{z}$ and $\partial_{\bar{z}}$. Indeed, 
\[
    \partial_{w}=(\partial_{w}z)(\partial_{z}+\bar{\mu}\partial_{\bar{z}}) \ ,
\]
where $\mu$ is the \emph{Beltrami differential} of $w$ with respect to the local coordinate $z$, in other words
\[
    \mu=-\frac{\partial_{\bar{z}}w}{\partial_{z}w} \ .
\]
In particular, we observe that $\partial_{\bar{z}}$ and $\partial_{w}$ do not commute. In fact, the following holds:

\begin{lemma} \label{lm:commutator} Let $\partial_{\bar{z}}$ and $\partial_{w}$ be the vector fields defined above. Then
\[
    [\partial_{\bar{z}}, \partial_{w}] = \partial_{\bar{z}}(\log(\partial_{w}z))\partial_{w} -\partial_{w}(\log(\partial_{\bar{z}}\bar{w}))\partial_{\bar{z}} \ .
\]
\end{lemma}
\begin{proof} Let $f:S\rightarrow \C$ be a smooth function. In the computation that follows we will make repeated use of the following identities, which are easy consequences of the definition of the Beltrami differential $\mu$:
\[
    \bar{\mu}= - \frac{ \pz \bar{w}}{\pbz \bar{w}} \ \ \ \ \ \ \ \  \bar{\mu}_{\bar{z}}= -\pz(\log(\pbz \bar{w}))-\bar{\mu}\pbz(\log(\pbz \bar{w})) = -\frac{\pw(\log(\pbz \bar{w}))}{\pw z}
\]
By definition of Lie bracket of vector fields, we have
\begin{align*}
    [\pbz , \pw](f) &= [\pbz, (\pw z)(\pz+\bar{\mu}\pbz)](f) \\
                    &= \pbz[ (\pw z)(\pz+\bar{\mu}\pbz)(f) ] - (\pw z)(\pz+\bar{\mu}\pbz)f_{\bar{z}} \\
                    &= \pbz [ (\pw z)(f_{z}+\bar{\mu}f_{\bar{z}}) ] - (\pw z)(f_{z\bar{z}}+\bar{\mu}f_{\bar{z}\bar{z}}) \\
                    &= (\pw z)(f_{\bar{z}z}+\bar{\mu}_{\bar{z}}f_{\bar{z}}+\bar{\mu}f_{\bar{z}\bar{z}}) + (\pbz\pw z)(f_{z}+\bar{\mu}f_{\bar{z}})-(\pw z)(f_{z\bar{z}}+\bar{\mu}f_{\bar{z}\bar{z}}) \\
                    &=(\pw z)\bar{\mu}_{\bar{z}}f_{\bar{z}}+ (\pbz\pw z)f_{z}+(\pbz\pw z)\bar{\mu}f_{\bar{z}} \\
                    &=-\pw(\log(\pbz \bar{w}))f_{\bar{z}}+(\pbz\pw z)f_{z}+(\pbz\pw z)\bar{\mu}f_{\bar{z}} \\
                    &=-\pw(\log(\pbz \bar{w}))f_{\bar{z}}+(\pbz\pw z)\left( \frac{f_{w}}{\pw z}-\bar{\mu}f_{\bar{z}} \right) +(\pbz\pw z)\bar{\mu}f_{\bar{z}} \tag{$\pz = \frac{\pw}{\pw z}-\bar{\mu}\pbz$} \\
                    &=-\pw(\log(\pbz \bar{w}))f_{\bar{z}} + \pbz(\log(\pw z))f_{w}
\end{align*}
and the statement follows.
\end{proof}

\begin{defi} Let $h$ be a positive complex metric with Levi-Civita connection $\nabla$. Given a smooth function $\phi:S\rightarrow \C$, we define the $h$-Laplacian of $\phi$ as
\[
    \Delta_{h}\phi:=\trace_{h}(\nabla^{2}\phi)=\trace_{h}( (X,Y) \mapsto \nabla_{X}\nabla_{Y}\phi - \nabla_{\nabla_{X}Y}\phi) \ .
\]
\end{defi}

\begin{prop}\label{prop:laplacian} In local coordinates, if $h=2e^{2\psi}dzd\bar{w}$, we have
\[
    \Delta_{h}\phi= \frac{e^{-2\psi}}{(\partial_{\bar{z}}\bar{w})(\partial_{w}z)}\bigg[\phi_{w\bar{z}}+\phi_{\bar{z}w}-\partial_{w}(\log(\partial_{\bar{z}}\bar{w}))\phi_{\bar{z}}-\partial_{\bar{z}}(\log(\partial_{w}z))\phi_{w}\bigg] \ . 
\]  
\end{prop}
\begin{proof} We start by computing the Christoffel symbols of the Levi-Civita connection of $h$. We set
\begin{align*}
    \nabla_{\pbz}\pbz &= \Gamma_{\bar{z}\bar{z}}^{\bar{z}}\pbz + \Gamma_{\bar{z}\bar{z}}^{w}\pw \ \ \ \ \ \ \nabla_{\pw}\pbz = \Gamma_{w\bar{z}}^{\bar{z}}\pbz+\Gamma_{w\bar{z}}^{w}\pw \\
    \nabla_{\pbz}\pw &= \Gamma_{\bar{z}w}^{\bar{z}}\pbz + \Gamma_{\bar{z}w}^{w}\pw \ \ \ \ \ \nabla_{\pw}\pw = \Gamma_{ww}^{\bar{z}}\pbz+\Gamma_{ww}^{w}\pw \ .
\end{align*}
Since $\nabla$ is compatible with the metric, we have
\[
    0=\pbz(h(\pbz,\pbz))=\frac{1}{2}h(\nabla_{\pbz}\pbz, \pbz) = \frac{1}{2}h(\Gamma_{\bar{z}\bar{z}}^{\bar{z}}\pbz + \Gamma_{\bar{z}\bar{z}}^{w}\pw, \pbz) = \frac{1}{2}\Gamma_{\bar{z}\bar{z}}^{w}e^{2\psi}(\pw z)(\pbz \bar{w})
\]
and we conclude that $\Gamma_{\bar{z}\bar{z}}^{w}=0$. With a similar computation, we get that $\Gamma_{\bar{z}w}^{\bar{z}}=\Gamma_{w\bar{z}}^{w}=\Gamma_{ww}^{\bar{z}}=0$. Moreover,
\begin{align*}
    \Gamma_{\bar{z}\bar{z}}^{\bar{z}}h(\pbz, \pw)+\Gamma_{\bar{z}w}^{w}h(\pbz, \pw) &= h(\nabla_{\pbz}\pbz, \pw)+h(\pbz, \nabla_{\pbz}\pw) \\
    & = \pbz(h(\pbz, \pw)) = \pbz(e^{2\psi}(\pbz \bar{w})(\pw z)) \\
    & = e^{2\psi}[2\psi_{\bar{z}}(\pbz \bar{w})(\pw z) + (\pbz\pbz \bar{w})(\pw z) + (\pbz \bar{w})(\pbz\pw z)]
\end{align*}
and dividing both sides by $h(\pbz, \pw)=e^{2\psi}(\pbz \bar{w})(\pw z)$ we obtain 
\begin{equation}\label{eq:eq1}
    \Gamma_{\bar{z}\bar{z}}^{\bar{z}}+\Gamma_{\bar{z}w}^{w} = 2\psi_{\bar{z}}+\pbz(\log(\pbz \bar{w}))+\pbz(\log(\pw z)) \ . 
\end{equation}
Similarly, 
\begin{align*}
    \Gamma_{w\bar{z}}^{\bar{z}}h(\pbz, \pw)+\Gamma_{ww}^{w}h(\pbz, \pw) &= h(\nabla_{w}\pbz, \pw)+h(\pbz, \nabla_{\pw}\pw) \\
    & = \pw(h(\pbz, \pw)) = \pw(e^{2\psi}(\pbz \bar{w})(\pw z)) \\
    & = e^{2\psi}[2\psi_{w}(\pbz \bar{w})(\pw z) + (\pw\pbz \bar{w})(\pw z) + (\pbz \bar{w})(\pw\pw z)]
\end{align*}
so that
\begin{equation}\label{eq:eq2}
    \Gamma_{w\bar{z}}^{\bar{z}}+\Gamma_{ww}^{w} = 2\psi_{w}+\pw(\log(\pbz \bar{w}))+\pw(\log(\pw z)) \ . 
\end{equation}
In addition, since $\nabla$ is torsion-free, we have
\[
    \Gamma_{\bar{z}w}^{w} \pw - \Gamma_{w\bar{z}}^{\bar{z}}\pbz = \nabla_{\pbz}\pw- \nabla_{\pw}\pbz = [\pbz , \pw] \ .
\]
Comparing this last equation with Lemma \ref{lm:commutator}, we deduce that 
\begin{equation}\label{eq:Christoffel1}
    \Gamma_{\bar{z}w}^{w} = \pbz(\log(\pw z)) \ \ \ \text{and} \ \ \ \Gamma_{w\bar{z}}^{\bar{z}} = \pw(\log(\pbz \bar{w})) \ 
\end{equation}
and, replacing these Christoffel symbols into Equation \eqref{eq:eq1} and \eqref{eq:eq2}, we finally get
\begin{equation}\label{eq:Christoffel2}
    \Gamma_{\bar{z}\bar{z}}^{\bar{z}} = 2\psi_{\bar{z}}+\pbz(\log(\pbz \bar{w})) \ \ \ \text{and} \ \ \ \Gamma_{ww}^{w} = 2\psi_{w}+\pw(\log(\pw z)) \ .
\end{equation}
It is now straightforward to verify that the matrix representing $\nabla^{2}\phi$ in the basis $\{\pbz,\pw\}$ is
\[
    \nabla^{2}\phi = \begin{pmatrix} 
                            \phi_{\bar{z}\bar{z}} - [2\psi_{\bar{z}} + \pbz(\log(\pbz \bar{w}))]\phi_{\bar{z}} & \phi_{\bar{z}w} - \pbz(\log(\pw z))\phi_{w} \\
                            \phi_{w\bar{z}}-\pw(\log(\pbz \bar{w}))\phi_{\bar{z}} & \phi_{ww} - [2\psi_{w} + \pw(\log(\pw z))]\phi_{w}
                        \end{pmatrix}
\]
and, since $h$ in the same basis can be expressed by
\[
    h = e^{2\psi}(\pbz \bar{w})(\pw z)\begin{pmatrix} 0 & 1 \\ 1 & 0 \end{pmatrix} \ ,
\]
we conclude that 
\begin{align*}
    \Delta_{h}\phi &= \trace_{h}(\nabla^{2}\phi) = \trace(h^{-1}\nabla^{2}\phi) \\
                    &=\frac{e^{-2\psi}}{(\partial_{\bar{z}}\bar{w})(\partial_{w}z)}\bigg[\phi_{w\bar{z}}+\phi_{\bar{z}w}-\partial_{w}(\log(\partial_{\bar{z}}\bar{w}))\phi_{\bar{z}}-\partial_{\bar{z}}(\log(\partial_{w}z))\phi_{w}\bigg] \ .
\end{align*}

\end{proof}

\begin{remark} This is the same operator defined in \cite{hol_ext}, which extends holomorphically the Laplacian of a Riemannian metric on $S$.
\end{remark}

\noindent Similar to the case of Riemannian metrics on $S$, the Laplacian operator is related to its compatible $\C$-complex structure $\mathcal{J}$:

\begin{prop}\label{prop:laplacian2} Let $h$ be a positive complex metric on $S$ with compatible $\C$-complex structure $\mathcal{J}$ and area form $dA_{h}$. Then for any smooth function $\phi:S \rightarrow \C$, we have
\[
    d(d\phi \circ \mathcal{J})= (\Delta_{h}\phi)dA_{h} \ .
\]
\end{prop}
\begin{proof} Since
\begin{align*}
    (d\phi \circ \mathcal{J})(\pbz) &= d\phi(\mathcal{J}\pbz)=i\phi_{\bar{z}} \\
    (d\phi \circ \mathcal{J})(\pw) &= d\phi(\mathcal{J}\pbz)=-i\phi_{w} \ ,
\end{align*}
the $1$-form $d\phi \circ \mathcal{J}$ can be written in local coordinates as
\[
    d\phi \circ \mathcal{J} = -i\frac{\phi_{w}}{\pw z}dz+i\frac{\phi_{\bar{z}}}{\pbz \bar{w}}d\bar{w} \ .
\]
Now, if $\alpha = Adz+B\bar{w}$ is any $1$-form, then
\begin{align*}
    d\alpha(\pbz, \pw) &= \pbz(\alpha(\pw))-\pw(\alpha(\pbz))-\alpha([\pbz, \pw]) \\
                       &= \pbz(A(\pw z)) - \pw(B(\pbz \bar{w})) -\alpha([\pbz, \pw]) \\
                    &=A_{\bar{z}}(\pw z) + A(\pbz\pw z) -B_{w}(\pbz \bar{w}) -B(\pw\pbz\bar{w}) \\
                         &\ \ \ - \alpha(\partial_{\bar{z}}(\log(\partial_{w}z))\partial_{w} -\partial_{w}(\log(\partial_{\bar{z}}\bar{w}))\partial_{\bar{z}}) \tag{Lemma \ref{lm:commutator}} \\
                         &= A_{\bar{z}}(\pw z) - B_{w}(\pbz \bar{w}) \ .
\end{align*}        
Therefore, the local expression of $d\alpha$ is 
\[
    d\alpha = \left[\frac{B_{w}}{\pw z} -\frac{A_{\bar{z}}}{\pbz\bar{w}}\right]dz\wedge d\bar{w} \ .
\]
Applying this formula to $\alpha=d\phi\circ \mathcal{J}$, where $A=-i\frac{\phi_{w}}{\pw z}$ and $B=i\frac{\phi_{\bar{z}}}{\pbz \bar{w}}$, we get
\begin{align*}
    d(d\phi\circ \mathcal{J}) &= \frac{i}{(\pw z)(\pbz \bar{w})} \bigg[ \phi_{\bar{z}w}+\phi_{w\bar{z}}-\pbz(\log(\pbz \bar{w}))\phi_{\bar{z}} - \pbz(\log(\pw z))\phi_{w} \bigg]dz\wedge d\bar{w} \\
     &= (\Delta_{h}\phi)dA_{h} \tag{Proposition \ref{prop:laplacian} and Equation \eqref{eq:areaform}} 
\end{align*}
as claimed.
\end{proof}

\noindent Finally, the formula for the curvature of a conformal Riemannian metric on a surface can be extended to positive complex metrics:

\begin{theorem}\label{thm:curvature} Let $h=2e^{2\psi}dzd\bar{w}$ be a positive complex metric on $S$. Then its complex Gaussian curvature can be computed as
\[
    K_{h}=-\Delta_{h}\psi = -\frac{1}{2}\Delta_{h}\log(h) \ .
\]
\end{theorem}
\begin{proof} We first compute the Riemann tensor in local coordinates:
\begin{align*}
    R^{\nabla}(\pbz, \pw, \pbz, \pw)&=h(\nabla_{\pbz}\nabla_{\pw}\pbz, \pw)-h(\nabla_{\pw}\nabla_{\pbz}\pbz, \pw)-h(\nabla_{[\pbz, \pw]}\pbz, \pw) \\
    &=h(\nabla_{\pbz}(\Gamma_{w\bar{z}}^{\bar{z}}\pbz), \pw)-h(\nabla_{\pw}(\Gamma_{\bar{z}\bar{z}}^{\bar{z}}\pbz, \pw) - h(\nabla_{\Gamma_{\bar{z}w}^{w}\pw-\Gamma_{w\bar{z}}^{\bar{z}}\pbz}\pbz, \pw) \\
    &=h(\pbz, \pw)\left[ \pbz \Gamma_{w\bar{z}}^{\bar{z}} - \pw \Gamma_{\bar{z}\bar{z}}^{\bar{z}} + \Gamma_{w\bar{z}}^{\bar{z}}\Gamma_{\bar{z}\bar{z}}^{\bar{z}}-\Gamma_{\bar{z}w}^{w}\Gamma_{w\bar{z}}^{\bar{z}} \right]  \ .
\end{align*}
Then, by Equation \eqref{eq:Christoffel1} and \eqref{eq:Christoffel2}, we find that
\begin{align*}
    K_{h} & = \frac{R^{\nabla}(\pbz, \pw, \pbz, \pw)}{h(\pbz, \pw)^2} = \frac{1}{h(\pbz, \pw)}\left[ \pbz \Gamma_{w\bar{z}}^{\bar{z}} - \pw \Gamma_{\bar{z}\bar{z}}^{\bar{z}} + \Gamma_{w\bar{z}}^{\bar{z}}\Gamma_{\bar{z}\bar{z}}^{\bar{z}}-\Gamma_{\bar{z}w}^{w}\Gamma_{w\bar{z}}^{\bar{z}} \right] \\
    &= \frac{1}{h(\pbz, \pw)} \bigg[ \pbz\pw(\log(\pbz\bar{w}))-2\psi_{w\bar{z}} - \pw\pbz(\log(\pbz \bar{w})) \\
    &  \ \ \ \ \ \ \ \ \ \ \ \ \ \ \ \ + \pw(\log(\pbz \bar{w}))\Big(2\psi_{\bar{z}} +\pbz\log(\pbz \bar{w}) -\pbz(\log(\pw z))\Big)\bigg] \\
    & = \frac{-1}{h(\pbz, \pw)}[ 2\psi_{w\bar{z}} -2\psi_{\bar{z}}\pw(\log(\pbz \bar{w}))] \\
    &= \frac{-1}{h(\pbz, \pw)}[ \psi_{w\bar{z}} +[\pw, \pbz](\psi) +\psi_{\bar{z}w} -
    2\psi_{\bar{z}}\pw(\log(\pbz \bar{w})) ] \\
    &= \frac{-1}{h(\pbz, \pw)}[\psi_{w\bar{z}} +\psi_{\bar{z}w} -
    \pw(\log(\pbz \bar{w}))\psi_{\bar{z}} - \pbz(\log(\pw z))\psi_{w}] \tag{Lemma \ref{lm:commutator}} \\
    &= \frac{-1}{e^{2\psi}(\pbz \bar{w})(\pw z)}[\psi_{w\bar{z}} +\psi_{\bar{z}w} -
    \pw(\log(\pbz \bar{w}))\psi_{\bar{z}} - \pbz(\log(\pw z))\psi_{w}] \\
    &= -\Delta_{h}\psi \tag{Proposition \ref{prop:laplacian}}
\end{align*}
\end{proof}

\begin{cor}\label{cor:curvature_conformal} Let $h=e^{2\phi}g$ be a positive complex metric conformal to $g$. Then
\[
    K_{h}=e^{-2\phi}(K_{g}-\Delta_{g}\phi)
\]
\end{cor}
\begin{proof} In local coordinates, $g=2e^{2\psi}dzd\bar{w}$ so that
\[
    K_{h} = -\Delta_{h}(\psi+\phi) = -e^{-2\phi}\Delta_{g}(\psi+\phi) = e^{-2\phi}(K_{g}-\Delta_{g}\phi) \ .
\]
    
\end{proof}

\noindent We now turn our attention to finding a local expression for the tensor $C^{+}(\cdot,\cdot,\cdot)=-h(\mathbf{P}\II(\cdot,\cdot),\cdot)$ introduced in Proposition \ref{prop:rel_tensors} (see also Remark \ref{rmk:A_and_II}). We start by showing a compatibility relation between $C^{+}$ and $\mathcal{J}$:

\begin{lemma}\label{lm:C_compatible} The tensor $C^{+}$ satisfies 
\[
    C^{+}(\mathcal{J} \cdot, \mathcal{J}\cdot, \mathcal{J}\cdot) = -C^{+}(\mathcal{J}\cdot, \cdot, \cdot) \ .
\]
\end{lemma}
\begin{proof} For any $X,Y,Z \in \Gamma(\C TU)$, we have
\begin{align*}
    C^{+}(\mathcal{J}X, \mathcal{J}Y, \mathcal{J}Z) &=-h(\mathbf{P}\II(\mathcal{J}X, \mathcal{J}Y), \mathcal{J}Z)  \\
    &=h(A^{+}(\mathcal{J}X)\mathcal{J}Y, \mathcal{J}Z) \tag{$A^{+}=-\mathbf{P}\II$} \\
    &=-h(\mathcal{J}A^{+}(\mathcal{J}X)\mathcal{J}Y, Z) \ . \tag{$h(\mathcal{J},\mathcal{J})=h$}
\end{align*}
Comparing this last expression with 
\[
    -C^{+}(\mathcal{J}X,Y,Z)=-h(A^{+}(\mathcal{J}X)Y,Z) \ ,
\]
we see that the statement is proved if we show that $A^{+}(\cdot)$ anti-commutes with $\mathcal{J}$. Now, since $\II$ is traceless, also $A^{+}(\cdot)$ is traceless, and we already know that $A^{+}(\cdot)$ is $h$-self-adjoint from the proof of Corollary \ref{cor:C_and_nablas}. This implies that, in the $h$-isotropic basis $\{\pbz, \pw\}$, the endomorphism $A^{+}(\cdot)$ is represented by a matrix of the form
\[
    A^{+}(\cdot) = \begin{pmatrix} 0 & * \\ ** & 0 \end{pmatrix} \ , 
\]
where $*$ and $**$ denote a arbitrary $\C$-valued smooth functions. Since, in the same basis, $\mathcal{J}=\diag(i,-i)$, anti-commutativity of $A^{+}(\cdot)$ and $\mathcal{J}$ follows by an easy direct computation.
\end{proof}

\begin{theorem}\label{thm:pair_hol_diff} There is a $J_{1}$-holomorphic cubic differential $q_{1}$ and a $J_{2}$-holomorphic cubic differential $q_{2}$ such that
\[
    C^{+}=q_{1}+\overline{q}_{2} \ .
\]
As a consequence, in local coordinates
\[
    C^{+} = \alpha(z)dz^{3}+\overline{\beta(w)}d\bar{w}^{3} , 
\]
for some $\alpha(z)$ and $\beta(w)$ holomorphic functions in their respective variable.
\end{theorem}
\begin{proof} We do this computation in $h$-isotropic normal coordinates $\{\pbu, \pv\}$, which exist by applying a change of basis to the usual normal coordinates of the Levi-Civita connection (\cite{nomizu1994affine}). This means that locally $h=\rho du d\bar{v}$ and all Christoffel symbols vanish at a given point $p\in S$. Note that the $\C$-complex structure $\mathcal{J}$ must satisfy $\mathcal{J}\pbu = -i\pbu$ and $\mathcal{J}\pv=i\pv$ because it is compatible with $h$. Since $C^{+}$ is a $(0,3)$-tensor, it must be written as
\[
    C^{+}=a(u,\bar{v})du^{3}+b(u,\bar{v})dud\bar{v}^{2}+c(u,\bar{v})du^{2}d\bar{v}+d(u,\bar{v})d\bar{v}^{3}
\]
for some smooth functions $a,b,c$ and $d$. On the other hand, Lemma \ref{lm:C_compatible} implies that $b=c=0$. Indeed, 
\begin{align*}
    ic(z,\bar{w})(\pw z)^{2}(\pbz \bar{w}) &=iC^{+}(\pbz, \pw, \pw)=-C^{+}(\mathcal{J}\pbz, \pw, \pw) \\
    &=C^{+}(\mathcal{J}\pbz, \mathcal{J}\pw, \mathcal{J}\pw) = -iC^{+}(\pbz, \pw, \pw) \\
    &=-ic(z,\bar{w})(\pw z)^{2}(\pbz \bar{w}) \ 
\end{align*}
and the computation for $b$ is analogous. \\
Now, since $h^{-1}C^{+}=-\mathbf{P}\II$, the Codazzi equation of $\II$ and $\nabla\mathbf{P}=\nabla h=0$ imply that $d^{\nabla}C^{+}=0$, in other words
\[
    (\nabla_{\pbu} C^{+})(\pv,\cdot, \cdot) = (\nabla_{\pv} C^{+})(\pbu,\cdot, \cdot) \ .
\]
Because we are working in normal coordinates, the equation above evaluated at $p$ is equivalent to 
\[
    \pbu (C^{+}(\pv,\cdot, \cdot))(p) = \pv(C^{+}(\pbu,\cdot, \cdot))(p) \ .
\]
As a consequence, 
\[
    \begin{cases}
        \pbu(C^{+}(\pv,\pv, \pv))(p)=0 \\
        \pv(C^{+}(\pbu, \pbu,\pbu)(p)=0  
    \end{cases} \Rightarrow 
    \begin{cases}
        \pbu(a(u,\bar{v})(\pv u)^{3})(p) =0 \\
        \pv(d(u,\bar{v})(\pbu \bar{v})^{3})(p)=0
    \end{cases}
\]
Since the vector fields $\pbu$ and $\pv$ commute in $p$ the last two equations simplify into
\[
    \begin{cases}
        0=\pbu(a(u,\bar{v}))(p) = \partial_{\bar{v}}a(u,\bar{v})(p)\pbu \bar{v}(p) \\
        0=\pv(d(u,\bar{v}))(p) = \partial_{u}d(u,\bar{v})(p)\pv u(p)
    \end{cases}
\]
which imply that $\pbu(a(u,\bar{v}))(p)=0$ and $\pv(d(u,\bar{v}))(p)=0$. This means that $a$ is a $J_{1}$-holomorphic function and $d$ is $J_{2}$-anti-holomorphic. Therefore, $C^{+}$ is the sum of a $J_{1}$-holomorphic and a $J_{2}$-anti-holomorphic cubic differential. Since this last statement is independent of the chosen coordinate system, the theorem is proved.
\end{proof}

\noindent We conclude this section by relating the norm of the second fundamental form $\II$ to the holomorphic cubic differentials $q_{1}$ and $q_{2}$.

\begin{prop}\label{prop:norm_C} Let $C^{+}=-h^{-1}\mathbf{P}\II=q_{1}+\bar{q}_{2}$. Then
\[
    \| \II \|^{2} = -8\|C^{+}\|_{h}^{2} \ ,      
\]
where the norm of $C^{+}$ is defined as
\[
    \|C^{+}\|_{h}^{2} := \frac{q_{1}\bar{q}_{2}}{h^{3}} \ .
\]
\end{prop} 
\begin{proof} In local $h$-isotropic coordinates $\{\pbz, \pw\}$, we write
\[
    h=2e^{2\psi}dzd\bar{w} \ \ \ \ \text{and} \ \ \ \ C^{+} = q_{1}+\overline{q}_{2}=\alpha dz^{3}+ \overline{\beta}d\bar{w}^{3} \ .
\]
Let $s^{2}=h(\pbz, \pw)=e^{2\psi}(\pw z)(\pbz \bar{w})$ and consider the $h$-orthonormal frame 
\[
    \bigg\{e_{1}:=\frac{1}{s\sqrt{2}}(\pbz+\pw), e_{2}:=\frac{i}{s\sqrt{2}}(\pbz-\pw)\bigg\} \ .
\]
Now,
\[
    \II(e_{1},e_{1})=h(\mathbf{P}\II(e_{1},e_{1}),e_{1})\mathbf{P}e_{1}+g(\mathbf{P}\II(e_{1},e_{1}),e_{2})\mathbf{P}e_{2}
\]
and, since $\mathbf{P}e_{i}$ have norm $-1$, we get 
\[
    \| \II(e_{1},e_{1})\|^{2}=-h(\mathbf{P}\II(e_{1},e_{1}), e_{1})^{2}-h(\mathbf{P}\II(e_{1},e_{1}),e_{2})^{2} \ .
\]
Similarly,
\[
    \| \II(e_{1},e_{2})\|^{2}=-h(\mathbf{P}\II(e_{1},e_{2}),e_{1})^{2}-h(\mathbf{P}\II(e_{1},e_{2}),e_{2})^{2}  \ .
\]
Each of these terms can be expressed using $\alpha$ and $\beta$. For instance,
\begin{align*}
    -h(\mathbf{P}\II(e_{1},e_{1}),e_{1})&=C^{+}(e_{1},e_{1},e_{1}) \\
        &=\frac{1}{2\sqrt{2}s^{3}}C^{+}(\pbz+\pw, \pbz+\pw, \pbz+\pw) \\
        &=\frac{1}{2\sqrt{2}s^{3}}[(\pbz \bar{w})^{3}\overline{\beta}+(\partial_{w}z)^{3}\alpha]
\end{align*}
and, with similar computations,
\begin{align*}
    -h(\mathbf{P}\II(e_{1},e_{1}),e_{2})&=\frac{i}{2\sqrt{2}s^{3}}[(\pbz \bar{w})^{3}\overline{\beta}-(\partial_{w}z)^{3}\alpha] \\
    -h(\mathbf{P}\II(e_{1},e_{2}),e_{1})&=\frac{i}{2\sqrt{2}s^{3}}[(\pbz \bar{w})^{3}\overline{\beta}-(\partial_{w}z)^{3}\alpha] \\ 
    -h(\mathbf{P}\II(e_{2},e_{2}),e_{2})&=-\frac{1}{2\sqrt{2}s^{3}}[(\pbz \bar{w})^{3}\overline{\beta}+(\partial_{w}z)^{3}\alpha] \ .
\end{align*}
Therefore,
\begin{align*}
    \|\II\|^{2} &= \| \II(e_{1},e_{1})\|^{2} + \| \II(e_{1},e_{2})\|^{2} =
            -\frac{(\partial_{w}z)^{3}(\pbz \bar{w})^{3}\alpha\overline{\beta}}{s^{6}} \\
            &= -\frac{\alpha\overline{\beta}}{e^{6\psi}}= -8 \frac{q_{1}\overline{q}_{2}}{h^3} = -8\|C^{+}\|^{2}_{h} \ .
\end{align*}
\end{proof}

\begin{cor}\label{cor:Gauss_coord} The Gauss equation of a smooth complex Lagrangian minimal immersion with positive complex-valued first fundamental form $h=2e^{2\psi}dzd\bar{w}$ and with holomorphic cubic differentials $q_{1}$ and ${q}_{2}$ is given in local isotropic coordinates by
\[
    \Delta_{h}\psi+q_{1}\overline{q}_{2}e^{-6\psi}-1=0
\]
\end{cor}
\begin{proof} It follows immediately from Theorem \ref{thm:Gauss_Codazzi}, Theorem \ref{thm:curvature} and Proposition \ref{prop:norm_C}.   
\end{proof}

\subsection{Bi-complex Higgs bundles}\label{sec:bicomplex_Higgs} In this section we introduce $\SL(n,\C)$-bi-complex Higgs bundles and use them to show that solving the Gauss-Codazzi equations found in Theorem \ref{thm:Gauss_Codazzi} is sufficient to find a complex Lagrangian minimal surface in $\CH^{2}_{\tau}$. This notion can be easily extended to any semisimple complex Lie group, but we restrict to $\SL(n,\C)$ in order to keep the length of this paper somewhat limited.

\noindent Let $S$ be a closed connected oriented surface endowed with a $\C$-complex structure $\mathcal{J}$. It is useful to consider the bi-complexified tangent space $\C_{\tau} TS:= \C TS\otimes_{\C}\C_{\tau}$ and cotangent space $\C_{\tau}T^{*}S:= \C T^{*}S \otimes_\C \C_{\tau}$. We will call sections of $\C_{\tau}TS$ and of $\C_{\tau}T^{*}S$ bi-complex vector fields and bi-complex $1$-forms, respectively. The $\C$-complex structure $\mathcal{J}$ extends to $\C_\tau TS$ by $\C_{\tau}$-linearity so that we can decompose $\C_{\tau}TS$ into $\mathcal{J}$-eigenspaces
\[
    \C_{\tau}TS = V_{i}(\mathcal{J}) \oplus V_{-i}(\mathcal{J}) = \C_{\tau}TS^{(1,0)} \oplus \C_{\tau} TS^{(0,1)} \ .
\]
A similar decomposition can be obtained for $\C_{\tau}T^{*}S$ by considering the action of the adjoint of $\mathcal{J}$. 

\begin{defi} A bi-complex vector bundle of rank $k$ over $S$ is a complex vector bundle $E$ over $S$ with fiber $\C^{2k}$ endowed fiberwise with a para-complex structure $\mathbf{T}$ such that locally, over an open set $U$ of $S$, we have a diffeomorphism
\begin{align*}
    \varphi_{U}: E_{|_{U}} \rightarrow U\times \C^{k}_{\tau} 
\end{align*}
such that for all $p \in U$, its restriction $\varphi_{p}$ at each fiber is a para-complex $\C$-linear isomorphism, i.e.  $\varphi_{p}(\mathbf{T}v)=\tau\varphi_{p}(v)$ for all $v\in E_{p}$. 
\end{defi}

\begin{remark} We will often think of a bi-complex vector bundle of rank $k$ simply as a vector bundle with fiber $\C^{k}_{\tau}$. Moreover, we remark that the fibers of a bi-complex vector bundle are naturally equipped with a $\C$-complex structure induced by the multiplication by $j\in \C_{\tau}$. 
\end{remark}

\noindent It follows immediately from the definition that transition functions of a bi-complex vector bundle of rank $k$ belong to a subgroup of $\GL(2k,\C)$ isomorphic to $\GL(k,\C)\times \GL(k,\C) \cong \GL(k,\C_{\tau})$. 

\begin{defi} We say that a bi-complex vector bundle $E$ over $S$ is $\mathcal{J}$-holomorphic if its transition functions $g_{\alpha\beta}: U_{\alpha}\cap U_{\beta} \rightarrow \GL(k,\C_{\tau})$ are $\mathcal{J}$-holomorphic, in the sense that they satisfy
\[
    dg_{\alpha\beta}(\mathcal{J}v) = jdg_{\alpha\beta}(v) \ \ \ \ {\text{for all} \ v\in \C TS}  \ .
\]
\end{defi}

\noindent On a $\mathcal{J}$-holomorphic bi-complex vector bundle we can define an analog of a $\bar{\partial}$-operator as follows. First, given a smooth function $f: U\subset S \rightarrow \C^{k}_{\tau}$, we denote by $\bar{\partial}^{\mathcal{J}}f$ the $\mathcal{J}$-antiliner part of $df$, in other words
\[
    \bar{\partial}^{\mathcal{J}}f:= \frac{1}{2}\bigg( df - jdf(\mathcal{J}) \bigg) \ . 
\]
Now, if $s=\sum_{i=1}^{k}s_{i}e_{i}$ is a local section of $E$, where $s_{i}:U\rightarrow \C^{k}_{\tau}$ are smooth functions and $\{e_{1}, \dots, e_{k}\}$ is a basis of $\C^{k}_{\tau}$ as $\C_{\tau}$-module, since the transition functions of $E$ are $\mathcal{J}$-holomorphic, the operator
\[
    \bar{\partial}^{E}(s):=\sum_{i=1}^{k}(\bar{\partial}^{\mathcal{J}}s_{i}) \otimes e_{i}\
\]  
does not depend on the local trivialization and thus it is defined globally as a map $\bar{\partial}^{E}: \Gamma(E) \rightarrow \Omega^{0,1}(S,\C_{\tau}) \otimes \Gamma(E)$. We say that a section $s$ of $E$ is $\mathcal{J}$-holomorphic if $\bar{\partial}^{E}s=0$.

\begin{defi} A $\C_{\tau}$-orthogonal structure $Q$ on a bi-complex vector bundle $E$ is a smooth choice of non-degenerate $\C_\tau$-bilinear inner products on each fiber.     
\end{defi}

\begin{defi}\label{def:bi-complexHIggs} An $\SL(n,\C)$-bi-complex Higgs bundle on $(S, \mathcal{J})$ is a triple $(E,Q,\phi)$ such that
\begin{itemize}
    \item $E$ is a $\mathcal{J}$-holomorphic bi-complex vector bundle of rank $n$ with $\bigwedge^{2n}E\cong \mathcal{O}$;
    \item $Q$ is a $\C_{\tau}$-orthogonal structure on $E$;
    \item $\phi$ is a $\mathcal{J}$-holomorphic section of $\mathrm{End}_{0}(E, \mathbf{T}) \otimes \C_\tau T^{*}S$ such that $\phi^{t}Q+Q\phi=0$, which we call \emph{bi-complex Higgs field}.
\end{itemize}
Here $\mathrm{End}_{0}(E,\mathbf{T})$ denotes the bundle of traceless $\C$-linear automorphismsm of $E$ that preserve $\mathbf{T}$, or, equivalently, after indentifying each fiber of $E$ with $\C^{n}_{\tau}$, the bundle of traceless $\C_{\tau}$-linear automorphisms of $E$.
\end{defi}

\noindent From an $\SL(n,\C)$-bi-complex Higgs bundle, it is possible to build a flat $\mathrm{SL}(n,\C)$-connection following ideas similar to the classical theory of Higgs bundles.

\begin{defi} A complex $\mathcal{J}$-Hermitian metric $H$ on a bi-complex vector bundle $E$ is a smooth choice of non-degenerate $\C$-bilinear $\mathcal{J}$-Hermitian inner products on each fiber. We say that a $\C_{\tau}$-linear connection $D^{H}$ on $E$ is compatible with $H$ if
\[
    V(H(s_{1},s_{2})) = H(D^{H}_{V}s_{1}, s_{2})+H(s_{1}, D^{H}_{\overline{V}^{\tau}}s_{2})
\]
for all sections $s_{1},s_{2}$ of $E$ and for all bi-complex vector fields $V$.
\end{defi}

\begin{theorem} Let $E$ be a $\mathcal{J}$-holomorphic vector bundle endowed with a complex $\mathcal{J}$-Hermitian metric. Then there exists a unique $\C_{\tau}$-linear connection $D^{H}$ on $E$, called \emph{Chern connection}, that is compatible with $H$ and whose $(0,1)$-part coincides with $\bar{\partial}^E$.
\end{theorem}
\begin{proof} Let $V$ be a $\mathcal{J}$-holomorphic bi-complex vector field. Then $D^{H}_{\overline{V}^{\tau}}=\bar{\partial}^{E}_{\overline{V}^{\tau}}$, so the connection $D^{H}$ is completely determined once we find $D^{H}_{V}$. Since $D^{H}$ must be compatible with the metric, we must have
\[
    H(D^{H}_{V}s_{1}, s_{2})=V(H(s_{1},s_{2}))-H(s_{1}, \bar{\partial}^{E}_{\overline{V}^{\tau}}s_{2})
\]
for all sections $s_{1}$ and $s_{2}$ of $E$. Since $H$ is non-degenerate, this equation characterizes $D^{H}_{V}$ uniquely.
\end{proof}

\noindent Given an $\SL(n,\C)$-bi-complex Higgs bundle $(E,Q,\phi)$ we look for a complex $\mathcal{J}$-Hermitian metric $H$ on $E$ that is compatible with $Q$ (i.e. $H$ commutes with $Q$ if seen as isomorphisms between $E$ and $E^{*}$) such that the connection 
\[
    \nabla = D^{H}+\phi+\phi^{*H}
\]
is flat, in other words $H$ satisfies Hitchin's equation
\begin{equation}\label{eq:Hitchin}
        F^{\nabla}=F^{H}+\phi\wedge \phi^{*H} = 0 \ .
\end{equation}
Clearly, in this setting Hitchin (\cite{hitchin1987self}) and Simpson's (\cite{simpson_Higgs}) work cannot be applied and the existence of such metric $H$ is not guaranteed. However, if a solution to Equation \eqref{eq:Hitchin} can be found, then, using the same argument as in the classical Higgs bundle theory, the holonomy of $\nabla$ lies in a subgroup of $\GL(n,\C_{\tau})$ isomorphic to $\SL(n,\C)$. \\

\noindent We will now see how we can associate to an equivariant complex Lagrangian minimal surface in $\CH^{2}_{\tau}$ an $\SL(3,\C)$-bi-complex Higgs bundle in such a way that Gauss-Codazzi equations correspond to Hitchin's equation and to the $\mathcal{J}$-holomorphicity of the bi-complex Higgs field. \\

\noindent Let $\sigma:\tilde{S} \rightarrow \CH^{2}_{\tau}$ be an equivariant complex Lagrangian minimal surface. We assume that its complex first fundamental form $h$ is a positive complex metric and we fix local holomorphic coordinates $z$ and $w$ such that locally $h=2e^{2\psi}dzd\bar{w}$ for some complex-valued smooth function $\psi$. We will still denote with $\mathcal{J}$ the $\C$-complex structure compatible with $h$. By an argument similar to that in Lemma \ref{lm:special_lift}, we can find a lift $\tilde{\sigma}: \tilde{S} \rightarrow \C_{\tau}^{3}$ such that $\mathbf{q}(\tilde{\sigma},\tilde{\sigma})=1$ and $\mathbf{q}(\tilde{\sigma}_{\bar{z}}, \tilde{\sigma})=\mathbf{q}(\tilde{\sigma}_{w},\tilde{\sigma})=0$. Moreover, it will be convenient to set
\[
    s^{2}=\mathbf{q}(\tilde{\sigma}_{\bar{z}},\tilde{\sigma}_{w})=e^{2\psi}(\pbz \bar{w})(\pw z) \ .
\]
We also write the tensor $C^{+}$ in local coordinates as $C^{+}=\alpha dz^{3}+\bar{\beta} d\bar{w}^{3}$. Now, the set $\{e_{1}:=\tilde{\sigma}_{\bar{z}}/s, e_{2}:=\tilde{\sigma}_{w}/s, \tilde{\sigma}\}$ is a reference frame for $\C_{\tau}^{3}$ as a $\C_{\tau}$-module. The pull-back of the standard flat connection on $\C_{\tau}^{3}$ can be represented in this frame by the endomorphisms-valued $1$-form $\Omega:=F^{-1}dF$, where $F$ is the matrix whose columns are given by the vectors $e_{1}, e_{2}$, and $\tilde{\sigma}$ in this order. Exploiting the geometry of complex Lagrangian minimal surfaces, one can compute that $\Omega= \frac{\hat{A}}{\pw z} dz + \frac{\hat{B}}{\pbz \bar{w}}d\bar{w}$, where
\[
    \hat{A} = \begin{pmatrix}
        -\psi_{w}+\frac{1}{2}\pw\log(\pbz \bar{w})-\frac{1}{2}\pw\log(\pw z) & -\frac{1}{s^{2}}\tau \alpha (\pw z)^{3} & 0 \\
        0 & \psi_{w}-\frac{1}{2}\pw\log(\pbz \bar{w})+\frac{1}{2}\pw\log(\pw z) & s \\
        s & 0 & 0 
    \end{pmatrix} \ \ \ 
\]
and
\[
    \hat{B} = \begin{pmatrix}
        \psi_{\bar{z}}+\frac{1}{2}\pbz\log(\pbz \bar{w})-\frac{1}{2}\pbz\log(\pw z) & 0 & s \\
        -\frac{1}{s^{2}}\tau \bar{\beta}(\pbz \bar{w})^{3} & -\psi_{\bar{z}}-\frac{1}{2}\pbz\log(\pbz \bar{w})+\frac{1}{2}\pbz\log(\pw z) & 0 \\
        0 & s & 0 
    \end{pmatrix}\ .
\]
We explain how to find the first column of $\hat{A}$; the other entries of the two matrices can be found with similar computations and are left to the interested reader. First, we observe that $\hat{A}=\Omega(\pw)=F^{-1}F_{w}$, thus the entries of its first column are the coordinates of $\pw e_{1}$ with respect to our reference frame. Simple linear algebra shows (recall that $e_{1}$ and $e_{2}$ are isotropic and their pairing is equal to $1$, whereas $\mathbf{q}(\tilde{\sigma},\tilde{\sigma})=-1$) that
\[
    \pw e_{1} = \mathbf{q}(\pw e_{1}, e_{2})e_{1}+\mathbf{q}(\pw e_{1}, e_{1})e_{2}-\mathbf{q}(\pw e_{1}, \tilde{\sigma})\tilde{\sigma} , 
\]
thus it is enough to compute the $\mathbf{q}$-pairing between $\pw e_{1}$ and the vectors of the frame:
\begin{enumerate}[i)]
    \item $\mathbf{q}(\pw e_{1}, \tilde{\sigma})=\pw \mathbf{q}(e_{1}, \tilde{\sigma})-\mathbf{q}(e_{1}, se_{2})=-s$
    \item $\mathbf{q}(\pw e_{1}, e_{1})=-\frac{1}{s}s_{w}\mathbf{q}(e_{1}, e_{2})+\frac{1}{s^{2}}\mathbf{q}(\tilde{\sigma}_{w\bar{z}},\tilde{\sigma}_{\bar{z}})=0$ because 
    \begin{itemize}
        \item $2\Ree_{\tau}(\mathbf{q}(\tilde{\sigma}_{w\bar{z}}, \tilde{\sigma}_{\bar{z}})) = \mathbf{q}(\tilde{\sigma}_{w\bar{z}}, \tilde{\sigma}_{\bar{z}}) + \mathbf{q}( \tilde{\sigma}_{\bar{z}},\tilde{\sigma}_{w\bar{z}}) = \pw \mathbf{q}(\tilde{\sigma}_{\bar{z}},\tilde{\sigma}_{\bar{z}}) =0$ ;
        \item $ \! \begin{aligned}[t] -\Ima_{\tau}(\mathbf{q}(\tilde{\sigma}_{w\bar{z}},\tilde{\sigma}_{\bar{z}})) &= \Ree_{\tau}(\mathbf{q}(\tilde{\sigma}_{w\bar{z}},\tau\tilde{\sigma}_{\bar{z}})) = \Ree_{\tau}(\mathbf{q}(\tilde{\sigma}_{w\bar{z}},\mathbf{P}\tilde{\sigma}_{\bar{z}})) \\ &= g(\II(\pw, \pbz), \mathbf{P}\tilde{\sigma}_{\bar{z}}) = C^{+}(\pw, \pbz,\pbz)=0 \end{aligned}$
    \end{itemize}
    \item $\mathbf{q}(\pw e_{1}, e_{2}) = -\frac{1}{s}s_{w}\mathbf{q}(e_{1},e_{2})+\frac{1}{s^{2}}\mathbf{q}(\tilde{\sigma}_{w\bar{z}}, \tilde{\sigma}_{w}) = -\frac{1}{2}\pw \log(s^{2})+\pw\log(\pbz \bar{w})$, because, by Lemma \ref{lm:commutator}, $\mathbf{q}(\tilde{\sigma}_{w\bar{z}},\tilde{\sigma}_{w})=\mathbf{q}(\tilde{\sigma}_{\bar{z}w},\tilde{\sigma}_{w})+s^{2}\pw\log(\pbz \bar{w})$ and $\mathbf{q}(\tilde{\sigma}_{\bar{z}w},\tilde{\sigma}_{w})$ can be computed as follows:
    \begin{itemize}
        \item $\Ree_{\tau}(\mathbf{q}(\tilde{\sigma}_{\bar{z}w},\tilde{\sigma}_{w})) = \frac{1}{2}\big(\mathbf{q}(\tilde{\sigma}_{\bar{z}w},\tilde{\sigma}_{w})+ \mathbf{q}(\tilde{\sigma}_{w}, \tilde{\sigma}_{\bar{z}w})\big) = \pbz \mathbf{q}(\tilde{\sigma}_{w}, \tilde{\sigma}_{w})=0$ ;
        \item $ \! \begin{aligned}[t] -\Ima_{\tau}(\mathbf{q}(\tilde{\sigma}_{\bar{z}w},\tilde{\sigma}_{w})) &= \Ree_{\tau}(\mathbf{q}(\tilde{\sigma}_{\bar{z}w},\tau\tilde{\sigma}_{w}))= g(\II(\pbz, \pw), \mathbf{P}\tilde{\sigma}_{w}) \\ &= C^{+}(\pbz, \pw,\pw)=0 \end{aligned} $
    \end{itemize}
\end{enumerate}

\noindent Flatness of the standard flat connection translates into the Maurer-Cartan equation for the form $\Omega$
\[
    d\Omega + \Omega \wedge \Omega = 0 \ .
\]
In terms of the matrices $\hat{A}$ and $\hat{B}$ this can be written as (see also the computations in the proof of Proposition \ref{prop:laplacian2})
\[
    \hat{B}_{w}-(\pw\log(\pbz\bar{w}))\hat{B}-\hat{A}_{\bar{z}}+(\pbz\log(\pw z))\hat{A}+\hat{A}\hat{B}-\hat{B}\hat{A}=0
\]
A tedious but elementary calculation shows that only the entries $(1,1)$, $(2,2)$ (which is actually equal to entry $(1,1)$ up to the sign), $(1,2)$ and $(2,1)$ of the matrix on the left-hand side of the equation above are non-zero and thus we obtain the system 
\[
    \begin{cases}
        s^{2}(\Delta_{h}\psi + 8\|C^{+}\|_{h}^{2}-1)=0 \\
        \frac{1}{s^{2}}\tau\alpha_{\bar{z}}(\pw z)^{3}=0 \\
        \frac{1}{s^{2}}\tau\bar{\beta}_{w}(\pbz \bar{w})^{3}=0 \ .
    \end{cases}
\]
By Theorem \ref{thm:Gauss_Codazzi}, Theorem \ref{thm:pair_hol_diff}, Proposition \ref{prop:norm_C} and Corollary \ref{cor:Gauss_coord}, these are equivalent to the Gauss-Codazzi equations for the complex Lagrangian minimal surface $\sigma$ in $\CH^{2}_{\tau}$. Therefore, we have proved the following:

\begin{theorem}\label{thm:Gauss_Codazzi_sufficient} Let $h$ be a positive complex metric with compatible $\C$-complex structure $\mathcal{J}$ and let $C$ be a $(0,3)$-tensor such that $C(\mathcal{J}\cdot, \mathcal{J}\cdot, \mathcal{J}\cdot)=-C(\mathcal{J}\cdot, \cdot, \cdot)$. If they satisfy the equations
\[
    \begin{cases}
        d^{\nabla}C=0 \\
        K_{h} - 8\|C\|_{h}^{2}+1=0
    \end{cases}
\]
then there is a unique (up to global isometries) equivariant complex Lagrangian minimal immersion $\sigma:\tilde{S}\rightarrow \CH^{2}_{\tau}$ with embedding data $h$ and $\II=-\mathbf{P}h^{-1}C$. 
\end{theorem}

\noindent Finally, we describe the bi-complex Higgs bundle related to an equivariant complex Lagrangian minimal surface $\sigma:\tilde{S}\rightarrow \CH^{2}_{\tau}$. Let $\tilde{\sigma}$ be its lift to $\C^{3}_{\tau}$ introduced before. The pull-back bundle 
\[
    \tilde{E}:=\tilde{\sigma}^{*}(T\C^{3}_{\tau})
\]
is naturally a bi-complex vector bundle $E$ over $\tilde{S}$ that splits as a direct sum of the trivial bundle $\underline{\C_{\tau}}$, which has $\tilde{\sigma}$ as global non-zero section, and a rank $2$ bi-complex vector bundle isomorphic to $\C_{\tau}T\tilde{S}$. The $\C$-complex structure on $\tilde{E}$ is given by the multiplication by $j$ on $\underline{\C_{\tau}}$ and by the $\C_{\tau}$-linear extension of the $\C$-complex structure on $\C T\tilde{S}$ compatible with the complex first fundamental form $h$ of $\sigma$. We extend $h$ by $\C_{\tau}$-bilinearity to $\C_{\tau} T\tilde{S}$. This allows us to define an isomorphism 
\begin{align*}
    \hat{h}: \C_{\tau}T\tilde{S}^{(0,1)} &\rightarrow K_{\mathcal{J}}:=\C_{\tau}T^{*}\tilde{S}^{(1,0)} \\
    \overline{V}^{\tau} &\mapsto h(\cdot, \overline{V}^{\tau})
\end{align*}
with inverse
\begin{align*}
    \hat{h}^{-1}: K_{\mathcal{J}} &\rightarrow \C_{\tau}T^{*}\tilde{S}^{(0,1)} \\
    dv &\mapsto \frac{\overline{V}^{\tau}}{\|\overline{V}^{\tau}\|_{h}^{2}} \ ,
\end{align*}
whenever $dv(V)=1$. In particular, we obtain an isomorphism $\C_{\tau}T\tilde{S} \cong K_{\mathcal{J}} \oplus K_{\mathcal{J}}^{*}$ such that both summands are isotropic for $h$ and are paired canonically by $h$. In particular, $\tilde{E}$ has trivial determinant. We can then extend $h$ to a $\C_{\tau}$-orthogonal structure $\tilde{Q}$ on $\tilde{E}$ by declaring $\underline{\C}_{\tau}$ orthogonal to $\C_{\tau}T\tilde{S}$ and $\tilde{Q}(\tilde{\sigma}, \tilde{\sigma})=-1$. The pull-back of the standard flat connection on $\C^{3}_{\tau}$ defines a flat connection $\tilde{\nabla}$ on $\tilde{E}$. In the local frame $\{e_{1}, e_{2}, \tilde{\sigma}\}$ of $\tilde{E}=\C_{\tau}T\tilde{S} \oplus \underline{\C_{\tau}}$, the matrix connection of $\tilde{\nabla}$ is expressed exactly by the $\mathrm{End}_{0}(E)$-valued $1$-form $\Omega$ computed before. Moreover, it follows immediately from the definition that in this frame
\[
    \tilde{Q} = \begin{pmatrix}
         0 & 1 & 0 \\
         1 & 0 & 0 \\
         0 & 0 & -1
    \end{pmatrix} \ . 
\]
Let $\tilde{H}$ be the $\mathcal{J}$-Hermitian metric obtained by extending $h$ to a $\mathcal{J}$-hermitian metric on $\C_{\tau}T\tilde{S}$ and demanding that $\tilde{\sigma}$ be $\tilde{H}$-orthogonal to $\C_{\tau}T\tilde{S}$ and with norm $1$. In other words, in the above frame,
\[
    \tilde{H} = \begin{pmatrix}
         0 & 1 & 0 \\
         1 & 0 & 0 \\
         0 & 0 & 1
    \end{pmatrix} \ . \ 
\]
We split $\tilde{\nabla}$ into its $\tilde{H}$-symmetric part $\tilde{\Phi}$ and its $\tilde{H}$-skew-symmetric part $\tilde{D}_{\tilde{H}}$. We can then further decompose into its $(1,0)$ and its $(0,1)$ part to get
\[
    \tilde{\nabla}=\tilde{D}_{\tilde{H}}^{(1,0)}+\tilde{D}_{\tilde{H}}^{(0,1)}+\tilde{\Phi}^{(1,0)}+\tilde{\Phi}^{(0,1)} \ .
\]
Since $\tilde{\nabla}$ is flat, the equations
\[
    \begin{cases}
        \tilde{D}_{\tilde{H}}^{(0,1)}\Phi^{(1,0)}=0 \\
        F^{\tilde{D}_{\tilde{H}}}+\tilde{\Phi}^{(1,0)}\wedge \tilde{\Phi}^{(0,1)}=0 \ .
    \end{cases}
\]
hold. In particular, the $(0,1)$-part of the connection $\tilde{D}_{\tilde{H}}$ defines on $\tilde{E}$ a $\mathcal{J}$-holomorphic structure for which $\tilde{\phi}=\tilde{\Phi}^{(1,0)}$ is $\mathcal{J}$-holomorphic. Because the immersion $\tilde{\sigma}$ is equivariant under the action of a representation $\rho:\pi_{1}(S) \rightarrow \SL(3,\C)$, all these objects descend to $S$ and, indicating their projections with the same symbol but without tilde, the triple $(E,Q,\phi)$ defines an $\SL(3,\C)$-bi-complex Higgs bundle over $S$.

\begin{remark}
Under the isomorphism $\SL(3,\C)\cong\SU(2,1,\C_\tau)$ described in Proposition \ref{prop:isometries}, our definition of $\SL(3,\C)$-bi-complex Higgs bundle can be re-prashed as follows: an $\SU(2,1,\C_\tau)$-bi-complex Higgs bundle on $(S,\mathcal J)$ is a triple $(E,\beta,\gamma)$ such that \begin{enumerate}
    \item[$\bullet$] $E$ is a $\mathcal J$-holomorphic bi-complex vector bundle of rank $3$ with $\bigwedge^6 E\cong\mathcal O$ endowed with a splitting $E=V\oplus L$, where $V$ is of rank $2$ and $L$ of rank $1$; \item[$\bullet$] $\beta$ and $\gamma$ are respectively $\mathcal J$-holomorphic sections of $\Hom(L,V)\otimes K_\mathcal J$ and $\Hom(V,L)\otimes K_\mathcal J$.
\end{enumerate}In particular, the Higgs field $\phi$ of Definition \ref{def:bi-complexHIggs} is recovered as $\phi=\begin{psmallmatrix}
    0 & \beta \\ \gamma & 0
\end{psmallmatrix}$. 
\end{remark}


\section{A digression on minimal Lagrangians in \texorpdfstring{$\mathbb{H}^{2}_{\tau}$}{H}}\label{sec:minimalinpara-complex}

\noindent In this section we are interested in complex Lagrangian minimal immersions $\sigma:U\to\CH^2_\tau$ whose image is contained in the (totally real) sub-manifold $\mathbb H^2_\tau=\CH^2_\tau\cap\R^3_\tau$ (see Section \ref{sec:submanifolds}). The holomorphic para-K\"ahler structure $(\hat g,\hat\omega,\mathbf P)$ on $\mathbb C\mathbb H^2_\tau$ restricts to $\mathbb H^2_\tau$ and gives rise to a (real) para-K\"ahler metric, still denoted with $(\hat g,\hat\omega,\mathbf P)$. In particular, we obtain a (real) para-hermitian form $\hat q=\hat g+\tau\hat\omega$. Throughout the discussion we will continue to keep the same notations as in Section \ref{sec:embedding_data} for the induced tensors on the surface, the connections, and the distributions. The only difference is that now they will all be real-valued tensors computed over real vector fields, and the various connections will be $\R$-linear. Indeed, the tensors defined in Section \ref{sec:embedding_data}, in particular the complex-valued first fundamental form and the complex second fundamental form, are simply the $\C$-linear extensions of the tensors that we will be using here. \\

\subsection{Relation with hyperbolic affine spheres}\label{sec:hyp_affine_spheres}
The main result of this section consists in showing that the datum of a space-like minimal Lagrangian in $\mathbb H^2_\tau$ is the same as that of a hyperbolic affine sphere in $\R^3$. The construction is geometric and a similar description was obtained by Hildebrand (\cite{hildebrand2011cross}) in terms of incidence geometry (discussed in Section \ref{sec:incidence_geometry}). We also show the correlation between the tensors induced on the minimal Lagrangian and the cubic differentials of the associated hyperbolic affine sphere and its dual. \\ 

\noindent Given a space-like minimal Lagrangian $\sigma:U\to\mathbb H^2_\tau$ we can consider the following tensors on $\sigma(U)$ $$C^{\pm}(X,Y,Z)=(\nabla^{\pm}_Xh)(Y,Z),\qquad A^{\pm}=\nabla^{\pm}-\nabla,$$ where $h$ is the first fundamental form of the immersion required to be positive-definite, $\nabla$ is the Levi-Civita for $h$ and $\nabla^{\pm}$ are the connections described in (\ref{eq:decomposition_nablaplusminus}). In particular, they are related by $C^\pm(X,Y,Z)=h(A^\pm(X)Y,Z)$. \\ \\
Let $f\!: U \to \R^{3}$ be an immersion with $\tilde\xi\!:U\to\R^{3}$ a transverse vector field to $f(U)$. This means that for all $x\in U$ we have a splitting: $$f^{*}(T\R^{3}) = T_xU + \R\tilde\xi_x \ . $$ 
Let $D^\R$ denote the pull-back to $f^{*}(T\R^{3})$ of the standard flat connection on $\R^3$ and suppose the structure equations of the immersed surface are given by:
\begin{equation}\begin{aligned}\label{eq:structurequations}
D^\R_XY &= \bar\nabla_XY + g_B(X,Y)\xi \\
D^\R_X\xi &= -S(X)
\end{aligned}\end{equation}
where $\bar\nabla$ is a torsion-free connection on $U$ called the \emph{Blaschke connection}, $\xi$ is the \emph{affine normal} of the immersion (see \cite[\S 3.1]{loftin2001affine} for example), $g_B$ is a Riemannian metric on $U$ called the \emph{Blaschke metric} and $S$ is an endomorphism of $TU$ called the \emph{affine shape operator}.
\begin{defi}
Let $f:U \rightarrow \R^{3}$ be an immersion with structure equations given by (\ref{eq:structurequations}). Then $f(U)$ is called a \emph{hyperbolic affine sphere} if $S=\Id_{TU}$ and $\xi=f$.
\end{defi} 

\noindent Recall that the \emph{conormal map} (\cite{nomizu1994affine}) associated with the hyperbolic affine sphere $f:U\to\R^3$ is a map $\nu:U\to(\R^3)^*$ such that for any $p\in U$, the element $\nu_p$ satisfies $$\nu_p(f(p))=-1 \qquad \text{and} \qquad \nu_p(df_{p}(X))=0, \ \forall X\in T_pU \ .$$ \begin{prop}[\cite{loftin2001affine}]
Let $f:U\to\R^3$ be a hyperbolic affine sphere. Then the image of the conormal map is a hyperbolic affine sphere in the dual space. In particular, it is equipped with a dual affine connection $\bar{\nabla}^*$, and an affine metric $g^*_B$ which makes $\nu$ an isometry between the affine sphere $f:U \rightarrow \R^{3}$ and its dual. 
\end{prop}

\noindent Given $f$ as above, if $\nabla^{g_B}$ denotes the Levi-Civita connection of the Blaschke metric, then we can consider the following End$(TU)$-valued $1$-form $A:=\bar\nabla-\nabla^{g_B}$, known as the \emph{Pick form}.\begin{prop}[\cite{benoist2013cubic}]
Let $C$ the $(0,3)$-tensor defined as 
\[
    C(X,Y,Z)=g_B(A(X)Y,Z)
\]
for $X,Y,Z\in\Gamma(TU)$, then the following holds: \begin{enumerate}
    \item[i)] $C$ is totally-symmetric;
    \item[ii)] The endomorphism part of $A$ is trace-less and $g_B$-symmetric; \item[iii)] The element $q:=C(\cdot,\cdot,\cdot)-iC(J\cdot,\cdot,\cdot)$ is a complex cubic differential on $U$ which is holomorphic with respect to the complex structure $J$ induced by the conformal class of $g_B$.
\end{enumerate} 
\end{prop}\noindent
Following the same procedure but using the tensors $C^*, A^*$ induced on the dual hyperbolic affine sphere, we obtain that these differ by one sign from those associated with $f$. In particular, the cubic differential associated with $C^*$ is exactly $-q$ (\cite{loftin2001affine}). \newline 

\noindent Suppose we have $f:U\to\R^3$ a hyperbolic affine sphere and its dual $f^*:= U\to(\R^3)^*$. According to the definition, for any $p\in U$ we can identify $f^*(p)$ with the linear functional $\R^3\to\R$ such that $\Ker(f^*(p))=T_{f(p)}f(U)$ and $f^*(p)\big(f(p)\big)=-1$. Let us decompose the space $\R^3_\tau$ as $\R^3e_+\oplus\R^3e_-$, where $\{e_+,e_-\}$ is the basis of idempotents, and, as already explained in Section \ref{sec:incidence_geometry}, we identify $\R^3e_-$ with the dual of $\R^3$ via Riesz representation theorem and the bilinear form $Q=\diag(1,1,-1)$. In particular, there exists a unique vector $v_p\in\R^3$ such that $f^*(p)(w)=v_p^tQw$, namely this is given by $v_p=Qf^*(p)^t$ if we think of $f^*(p)$ as a row vector. Then, we can define the map \begin{align*}
    \tilde\sigma: \ &U\to\R^3e_+\oplus\R^3e_- \\ &p\mapsto(f(p),f^*(p)) \ .
\end{align*}
The image $\tilde\sigma(U)$ is contained in the quadric $\{ z\in\C^3_\tau \ | \ \mathbf{q}(z,z)=-1  \} \cap \R^{3}_{\tau}$ and hence its projection under the action of $\mathcal{U}$ is contained in $\mathbb{H}^{2}_{\tau}=\CH^{2}_{\tau}\cap \R^{3}_{\tau}$. We denote by $\sigma:U\to\mathbb{H}^2_\tau$ the induced immersion to the quotient.

\begin{lemma}\label{lem:Blaschke_metric}
The pull-back of the para-hermitian form $\hat q$ on $\sigma(U)$ coincides with the Blaschke metric of the hyperbolic affine sphere $f:U\to\R^3$. 
\end{lemma}\begin{proof}
Let $\{e_1,e_2\}$ be an orthonormal basis of $T_{f(p)}f(U)$ (and hence of $T_{f^*(p)}f^*(U)$) with respect to the Blaschke metric $g_B$. It is sufficient to prove that $(\tilde\sigma^*\mathbf{q})(e_i,e_j)=\delta^i_j$. In fact, we can deduce directly that $\sigma^*\hat q$ is a real tensor and it is  equal to $g_B$. 
Let $p\in U$. Then, since $\mathrm d_p\tilde\sigma(e_i)=\big(\mathrm d_pf(e_i), \mathrm d_pf^*(e_i)\big)$ gives a basis of $T_{\sigma(p)}\sigma(U)$, it is sufficient to prove that \begin{equation}\mathbf{q}\big((\mathrm d_pf(e_i), \mathrm d_pf^*(e_i)),(\mathrm d_pf(e_j), \mathrm d_pf^*(e_j))\big)=\delta^i_j, \qquad i,j=1,2 \ .\end{equation}In the basis of idempotents $\{e_+=\frac{1+\tau}{2},e_-=\frac{1-\tau}{2}\}$ the para-hermitian form can be written as \begin{align*}\mathbf{q}\big((\mathrm d_pf(e_i), \mathrm d_pf^*(e_i)),(\mathrm d_pf(e_j), \mathrm d_pf^*(e_j))\big)= \ &\mathrm d_pf^*(e_i)\big(\mathrm d_pf(e_j)\big)e_-+ \\ &\ +\mathrm d_pf^*(e_j)\big(\mathrm d_pf(e_i)\big)e_+ \ .\end{align*}Let $\eta:\R^3e_+\oplus\R^3e_-\to\R$ be the pairing $\eta(v,\varphi):=\varphi(v)$, where the elements of $\R^3e_-$ are identified with linear functionals as explained above. This allows us to rewrite $$(\tilde\sigma^*\mathbf{q})(e_i,e_j)=\eta\big(\mathrm d_pf(e_j), \mathrm d_pf^*(e_i)\big)e_-+\eta\big(\mathrm d_pf(e_i), \mathrm d_pf^*(e_j)\big)e_+ \ .$$
Taking the derivative of the equation $\eta(f(p),f^*(p))=-1$ in the $e_i$ direction, we obtain \begin{align*}
    0&=\frac{\mathrm d}{\mathrm de_i}\big(\eta(f(p),f^*(p))\big) \\ &=\eta(\mathrm d_pf(e_i),f^*(p))+\eta(f(p),\mathrm d_pf^*(e_i)) \\ &=\eta(f(p),\mathrm d_pf^*(e_i)) \ . \tag{$\Ker(f^*(p))=T_{f(p)}f(U)$}
\end{align*}Deriving again in the $e_j$ direction \begin{align*}
    0&=\frac{\mathrm d^2}{\mathrm de_j\mathrm de_i}\big(\eta(f(p),f^*(p))\big) \\ &=\eta\big(D_{e_j}^\R\mathrm d_pf(e_i),f^*(p)\big)+\eta\big(\mathrm d_pf(e_i),\mathrm d_pf^*(e_j)\big)+\eta\big(\mathrm d_pf(e_j),\mathrm d_pf^*(e_i)\big) \\ & \ \ \ +\eta\big(f(p),D_{e_j}^\R\mathrm d_pf^*(e_i)\big) \\ &=\frac{\mathrm d}{\mathrm de_j}\big(\eta(\mathrm d_pf(e_i),f^*(p))\big)+\eta\big(\mathrm d_pf(e_j),\mathrm d_pf^*(e_i)\big)+\eta\big(f(p),D_{e_j}^\R\mathrm d_pf^*(e_i)\big) \\ &=\eta\big(\mathrm d_pf(e_j),\mathrm d_pf^*(e_i)\big)+\eta\big(f(p),D_{e_j}^\R\mathrm d_pf^*(e_i)\big) \ .
\end{align*}As a consequence we deduce the following two equations \begin{equation}\begin{aligned}\label{eq:derivative_e_j}
    &\eta\big(\mathrm d_pf(e_j),\mathrm d_pf^*(e_i)\big)=-\eta\big(f(p),D_{e_j}^\R\mathrm d_pf^*(e_i)\big) \\ &\eta\big(\mathrm d_pf(e_i),\mathrm d_pf^*(e_j)\big)=-\eta\big(D_{e_j}^\R\mathrm d_pf(e_i),f^*(p)\big) \ .
\end{aligned}\end{equation} Since $f$ and $f^*$ are both hyperbolic affine spheres, they must satisfy the structure equations (\ref{eq:structurequations}), which in our case become 
\begin{align*}
    &D^\R_{e_j}\mathrm d_pf(e_i)=\bar\nabla_{e_j}e_i+g_B(e_i,e_j)f(p)=\bar\nabla_{e_j}e_i+\delta^i_jf(p) \ , \qquad &\bar\nabla_{e_j}e_i\in T_{f(p)}f(U) \\ &D^\R_{e_j}\mathrm d_pf^*(e_i)=\bar\nabla_{e_j}^*e_i+g_B(e_i,e_j)f^*(p)=\bar\nabla_{e_j}^*e_i+\delta^i_jf^*(p) \  \qquad &\bar\nabla_{e_j}^*e_i\in T_{f^*(p)}f^*(U) \ .
\end{align*}
From here we can deduce  that
\begin{equation}\begin{aligned}\label{eq:pairing}
    &\eta\big(D_{e_j}^\R\mathrm d_pf(e_i),f^*(p)\big)=\eta\big(\bar\nabla_{e_j}e_i+\delta^i_jf(p),f^*(p)\big)=-\delta^i_j \ , \\ &\eta\big(f(p),D_{e_j}^\R\mathrm d_pf^*(e_i)\big)=\eta\big(f(p),\bar\nabla^*_{e_j}e_i+\delta^i_jf^*(p)\big)=-\delta^i_j \ .
\end{aligned}\end{equation}In the end, we can compute the term we were interested in \begin{align*}
    (\tilde\sigma^*\mathbf{q})(e_i,e_j)&=\eta\big(\mathrm d_pf(e_j), \mathrm d_pf^*(e_i)\big)e_-+\eta\big(\mathrm d_pf(e_i), \mathrm d_pf^*(e_j)\big)e_+ \\\ &=-\eta\big(f(p),D_{e_j}^\R\mathrm d_pf^*(e_i)\big)e_--\eta\big(D_{e_j}^\R\mathrm d_pf(e_i),f^*(p)\big)e_+ \tag{Eq. (\ref{eq:derivative_e_j})} \\ &=\delta^i_je_-+\delta^i_je_+ \tag{Eq. (\ref{eq:pairing})} \\ &=\delta^i_j \ .
\end{align*}
\end{proof}
\begin{theorem}
The immersed surface $\sigma:U\to\mathbb H^2_\tau$ is space-like, Lagrangian and transverse to the distributions $\mathcal{D}_\pm$.
\end{theorem}\begin{proof}
According to Lemma \ref{lem:Blaschke_metric} we have $\sigma^*\hat q=\sigma^*\hat g+\tau\sigma^*\hat\omega=g_B$, which implies $\sigma^*\hat\omega=0$, or in other words that $\sigma(U)$ is a Lagrangian surface. The immersion is then clearly space-like. Transversality to the distributions $\mathcal{D}_\pm$ is a consequence of Lemma \ref{lm:transverse}.
\end{proof}

\noindent The next step is to perform the opposite construction to the one just described: given a space-like minimal Lagrangian immersion $\sigma:U\to\mathbb H^2_\tau$ we want to define a hyperbolic affine sphere in $\R^3$ (identified with $\R^3e_+$) and its dual in $(\R^3)^*$ (identified with $\R^3e_-$). First, we can find a lift $\tilde\sigma:U\to\R^3e_+\oplus\R^3e_-$ of $\sigma$ such that $\mathbf{q}(\tilde{\sigma}, \tilde{\sigma})=-1$, which is still isotropic, in the sense that $\tilde\sigma^*\bm{\omega}=0$, because vectors tangent to the orbit of $\mathcal{U}$ are $\bm{\omega}$-orthogonal to the level sets of $\mathbf{q}$ (see Section \ref{sec:reduction}). Let $f^{+}:=\pi_{+}\circ \tilde{\sigma}:U\rightarrow \R^{3}e_{+}$ and $f^{-}:=\pi_{-}\circ \tilde{\sigma}:U\rightarrow \R^{3}e_{-}$ be the projections of $\tilde{\sigma}$ onto $\R^{3}e_{\pm}$. As usual, we identify $\R^{3}e_{-}$ with $(\R^{3})^{*}$ using the bilinear form $Q$, so that we think of $f^{-}(p)$ as a functional of $\R^{3}$ for all $p\in U$. The lift $\tilde{\sigma}$ is not unique: indeed, for any smooth function $\mu:U \rightarrow \mathbb{R}$ the map $\tilde{\sigma}^{\mu}:=(e^{\mu}f^{+}, e^{-\mu}f^{-})$ still projects onto $\sigma$, satisfies $\mathbf{q}(\tilde{\sigma}^{\mu}, \tilde{\sigma}^{\mu})=-1$ and parameterizes an isotropic surface in $\R^{3}_{\tau}$. The main idea of the construction is to choose a best representative in the above family:

\begin{lemma}\label{lm:special_lift} There exists a smooth function $\mu:U \rightarrow \mathbb{R}$ such that the corresponding immersion $\tilde{\sigma}^{\mu}:U \rightarrow \R^{3}_{\tau}$ satisfies
\[
    \mathbf{q}(\tilde{\sigma}_{x}^{\mu},\tilde{\sigma}^{\mu})=\mathbf{q}(\tilde{\sigma}_{y}^{\mu}, \tilde{\sigma}^{\mu})=0 \ .
\]
\end{lemma}
\begin{proof} Using the same notation as in Lemma \ref{lem:Blaschke_metric}, we first observe that
\[
    -1=\mathbf{q}(\tilde{\sigma}^{\mu}, \tilde{\sigma}^{\mu}) = \eta(f^{+}, f^{-}) \ .
\]
Taking the derivative of this relation, we find that
\[
    0=\eta(f_{x}^{+}, f^{-})+\eta(f^{+}, f_{x}^{-})
\]
and similarly
\[
    \eta(f^{+}_{y}, f^{-})+\eta(f^{+}, f^{-}_{y})=0 \ ,
\]
where $x,y$ are local coordinates on $U$. In addition, since the immersion $\tilde{\sigma}$ is isotropic we know that
\begin{equation}\label{eq:isotropic}
    0=\bm{\omega}(\tilde{\sigma}_{x}, \tilde{\sigma}_{y}) = \Ima_{\tau}(\mathbf{q}(\tilde{\sigma}_{x}, \tilde{\sigma}_{y})) = \eta(f^{+}_{x}, f^{-}_{y})-\eta(f^{+}_{y}, f^{-}_{x}) \ . 
\end{equation}
Since
\[
    \tilde{\sigma}^{\mu}_{x}=(\mu_{x}e^{\mu}f^{+}+e^{\mu}f^{+}_{x}, -\mu_{x}e^{-\mu}f^{-}+e^{-\mu}f_{x}^{-})\ ,
\]
the condition $\mathbf{q}(\tilde{\sigma}_{x}^{\mu}, \tilde{\sigma}^{\mu})=0$ can be expressed as
\begin{align*}
    0&=\mathbf{q}(\tilde{\sigma}_{x}^{\mu}, \tilde{\sigma}^{\mu}) \\
    &= \eta(\mu_{x}e^{\mu}f^{+}+e^{\mu}f^{+}_{x}, e^{-\mu}f^{-})e_{+}+\eta(e^{\mu}f^{+}, -\mu_{x}e^{-\mu}f^{-}+e^{\-mu}f_{x}^{-})e_{-} \\
    &= -\mu_{x}e_{+}+\eta(f_{x}^{+}, f^{-})e_{+}+\mu_{x}e_{-}+\eta(f^{+}, f_{x}^{-})e_{-} \\
    &= \tau(-\mu_{x}+\eta(f^{+}_{x}, f^{-})) \ .
\end{align*}
Similarly,
\[
    0=\mathbf{q}(\tilde{\sigma}_{y}^{\mu}, \tilde{\sigma}^{\mu})=\tau(-\mu_{y}+\eta(f^{+}_{y}, f^{-})) \ .
\]
This means that the immersion $\tilde{\sigma}^{\mu}$ satisfies the required properties if and only if $\mu$ is a solution of the system of differential equations
\[
    (\mu_{x}, \mu_{y})=(\eta(f^{+}_{x}, f^{-}), \eta(f^{+}_{y}, f^{-})) \ ,
\]
in other words, if and only if $\mu$ is a potential of the vector field in $\R^{2}$ defined by
\[
    F=(\eta(f^{+}_{x}, f^{-}), \eta(f^{+}_{y}, f^{-})) \ .
\]
Since $U$ is simply connected, the vector field $F$ has a potential if and only if it is irrotational, i.e.
\[
    \frac{\partial}{\partial y}\eta(f^{+}_{x}, f^{-}) = \frac{\partial}{\partial x}\eta(f^{+}_{y}, f^{-}) \ .
\]
Expanding these derivatives, we see that this equality holds because of Equation \eqref{eq:isotropic}.
\end{proof}

\noindent With this special lift in hand, which we will still denote by $\tilde{\sigma}$, we are able to reverse the construction of Lemma \ref{lem:Blaschke_metric}:

\begin{theorem}\label{thm:from_minimal_to_affine_sphere}
Let $\tilde{\sigma}=(f^{+}, f^{-})$ be the lift of $\sigma$ found in Lemma \ref{lm:special_lift}. Then the immersion $f^+$ is a hyperbolic affine sphere with dual equal to $f^-$. Moreover, the Blaschke metric induced on $f^+(U)$ coincides with the first fundamental form of the space-like minimal Lagrangian $\sigma:U\to\mathbb H^2_\tau$.
\end{theorem}
\begin{proof} We first observe that the decomposition
\begin{align*}
    \tilde{\sigma}^{*}(T\R^{3}_{\tau}) &=\sigma^{*}(T\mathbb{H}^{2}_{\tau}) \oplus \R_{\tau}\cdot \tilde{\sigma} \\
    &= TU \oplus \mathcal{D}_{-} \oplus \R_{\tau}\cdot \tilde{\sigma}
\end{align*}
implies, by projecting onto $\R^{3}e_{+}$, that the immersion $f^{+}$ admits a transverse vector field that coincides with $f^{+}$. By standard affine differential geometry (\cite{nomizu1994affine}), this implies that $f^{+}$ is a hyperbolic affine sphere. The same holds for $f^{-}$ by considering the distribution $\mathcal{D}_{+}$ in the above decomposition and the projection onto $\R^{3}e_{-}$. 
Finally, the relations $\mathbf{q}(\tilde{\sigma}, \tilde{\sigma})=-1$ and $\mathbf{q}(\tilde{\sigma}_{x},\tilde{\sigma})=\mathbf{q}(\tilde{\sigma}_{y}, \tilde{\sigma})=0$ imply that
\begin{align*}
    &\eta(f^{+},f^{-})=-1 \\
    &\eta(f^{+}_{x},f^{-})e_{+}+\eta(f^{+},f_{x}^{-})e_{-}=0 \\
    &\eta(f^{+}_{y},f^{-})e_{+}+\eta(f^{+},f_{y}^{-})e_{-}=0 \ .
\end{align*}
Since $e_{\pm}$ are linearly independent over $\R$, we deduce, in particular, that $\eta(f^{+}_{x}, f^{-})=\eta(f^{+}_{y},f^{-})=0$, which means that the kernel of the functional $f^{-}(p)$ is exactly $T_{f^{+}(p)}f^{+}(U)$ for all $p \in U$. Since, moreover, $f^{-}(p)(f^{+}(p))=-1$, we conclude that $f^{-}$ is the affine sphere dual to $f^{+}$. The second part of the statement then follows from Lemma \ref{lem:Blaschke_metric}.
\end{proof}

\begin{cor}\label{cor:correspondence_cubic}
Under the above correspondence, given a space-like minimal Lagrangian $\sigma:U\to\mathbb H^2_\tau$, the connections $\nabla^+,\nabla^-$ on $\sigma(U)$ defined by the (real) para-K\"ahler structure of $\mathbb H^2_\tau$ are equal to the Blaschke connections of the hyperbolic affine spheres $$f^+=\pi_+\circ\tilde\sigma:U\to\R^3e_+ \ , \qquad f^-=\pi_-\circ\tilde\sigma:U\to\R^3e_-$$ which are dual to each other and satisfy $$\nabla^{g_B}=\frac{1}{2}\big(\nabla^++\nabla^-\big) \ ,$$ for $\nabla^{g_B}$ the Levi-Civita connection with respect to the Blaschke metric. In particular, the holomorphic cubic differentials $$q^+=C^+-iC^+(J\cdot,\cdot,\cdot) \ , \qquad q^-=C^--iC^-(J\cdot,\cdot,\cdot)=-q^+$$ on $\sigma(U)$, correspond to Pick differentials of $f^+, f^-$.
\end{cor}
\begin{proof}
The first part of the statement is just a consequence of the construction performed during the proof of Theorem \ref{thm:from_minimal_to_affine_sphere} and Lemma \ref{lm:dual_connection} point $ii)$. In particular, the tensors $A^+=\nabla^+-\nabla^{\sigma^*\hat q}=\bar\nabla^+-\nabla^{g_B}$ and $A^-=\nabla^--\nabla^{\sigma^*\hat q}=\bar\nabla^--\nabla^{g_B}$ are exactly the Pick forms of the affine spheres $f^+$ and $f^-$. On the other hand, the cubic differentials on $\sigma(U)$ determined by $C^\pm=(\sigma^*\hat q)^{-1}A^\pm=\pm\hat g(\II(\cdot,\cdot),\mathbf{P}\cdot)$ are holomorphic if and only if $\mathrm d^{\nabla^{\sigma^*\hat q}} A^\pm=0$ and $\trace(A^\pm)=0$, which is then equivalent to $\mathrm d^{\nabla^{\sigma^*\hat q}}\II=0$ and $\trace(\II)=0$ (since $A^\pm=\mp\mathbf P\II$ by Remark \ref{rmk:A_and_II} and $\nabla\mathbf P=0$).
\end{proof}

\begin{cor}\label{cor:Wang}
Under the above correspondence, the structural equations of a complex Lagrangian minimal surface $\sigma:U\to\mathbb H^2_\tau \subset \CH^{2}_{\tau}$, i.e.
\begin{align*}
    &d^{\nabla^h}\II = 0 \ \ \ \ \ \ \ \ \ \ \ &\text{(Codazzi equation)} \\
    &K_{h}+\| \II \|^{2} = -1 \ \ \ \ \ &\text{(Gauss equation),}
\end{align*}
are the same as the ones describing the associated hyperbolic affine sphere $f_+$, namely \begin{align*}
    &\bar\partial_J q=0 \ \ \ \ \ \ \ \ \ \ \qquad \qquad \text{(Codazzi equation)} \\ & K_{g_B}-2\|q\|^2_{g_B}=-1 \ \ \ \ \ \ \ \ \text{(Wang equation)}
\end{align*}
\end{cor}
\begin{proof}
As already noticed, we know that $A^\pm=\mp\mathbf{P\II}$, hence $\mathrm d^{\nabla^h}\II=0$ if and only if $\mathrm d^{\nabla^h}A^\pm=0$ if and only if $\bar\partial_Jq^\pm=0$, where $J$ is the complex structure on $\sigma(U)$ induced by the conformal class of $h$, which is well-defined since $h$ is the $\C$-linear extension of a Riemannian metric. According to Corollary \ref{cor:correspondence_cubic}, if $q$ denotes the holomorphic cubic differential of $f_+$, then $q=q^+$, and in particular $\bar\partial_Jq=0$ for the same $J$ since $h=g_B$. Regarding the second equation, given a local $h$-orthonormal basis $\{e_1,e_2\}$ of $\C TU$ we obtain a $g$-orthonormal basis $\{\mathbf{P}e_1,\mathbf{P}e_2\}$ of $\C NU$. Then by definition $$\|\II \|^{2} = g(\II(e_{1},e_{1}), \II(e_{1},e_{1}))+g(\II(e_{2},e_{1}), \II(e_{2},e_{1})) \ .$$ Given conformal local coordinates for the induced metric $h=2e^{2\psi}|\mathrm d z|^2$ we can choose $\{e_1:=\frac{\partial_x}{\sqrt 2 e^\psi},e_2:=\frac{\partial_y}{\sqrt 2 e^\psi}\}$ and by writing $$\II(e_i,e_j)=-g(\II(e_i,e_j),\mathbf{P}e_1)\mathbf{P}e_1-g(\II(e_i,e_j),\mathbf{P}e_2)\mathbf{P}e_2 \ ,$$ we obtain 
\begin{align*}
    g(\II(e_{1},e_{1}), \II(e_{1},e_{1}))&=g(\II(e_{1},e_{1}),\mathbf{P}e_1)^2g(\mathbf{P}e_{1}, \mathbf{P}e_{1})+g(\II(e_1,e_1),\mathbf{P}e_2)^2g(\mathbf{P}e_{2}, \mathbf{P}e_{2}) \\ 
    &=-g(\II(e_{1},e_{1}),\mathbf{P}e_1)^2-g(\II(e_1,e_1),\mathbf{P}e_2)^2 \\
    &=-\frac{1}{8e^{6\psi}}\big(C^{+}(\partial_x,\partial_x,\partial_x)^2+C^{+}(\partial_y,\partial_y,\partial_y)^2\big) \\ &=-\frac{1}{8e^{6\psi}}\big(\Ree(q^{+})^2+\Ima(q^{+})^2\big) \\ &=-\|q\|^2_{h} \ \tag{$q=q^{+}$}.
\end{align*}
In the end, with a similar computation we deduce that $\|\II\|^2=-2\|q\|^2_{h}$ and the proof is complete because $h=g_{B}$.
\end{proof}

\begin{remark}[Relation with minimal Lagrangian surfaces in $\mathbb{CH}^{2}$]\label{rmk:min_Lagrangian_complex}
As an observation, we note that if the minimal Lagrangian $\sigma(U)$ is contained in the complex hyperbolic space $\mathbb{C}\mathbb{H}^2=\CH^{2}_{\tau}\cap (\R \oplus i\tau \R)$ (see Section \ref{sec:submanifolds}), then its complex-valued first fundamental form $h$ is the $\C$-linear extension of the induced metric of $\sigma(U)$ as a submanifold of $\CH^{2}$. In particular, $h$ is a positive complex metric and its compatible $\C$-complex structure $\mathcal{J}$ corresponds to the pair $(J,J)$ for some complex structure $J$ on $U$.  In addition, the tensor $-iC^{+}(X,Y,Z)=-ig(\II(X,Y),\mathbf{P}Z)=g(\II(X,Y),\mathbf{J}Z)$ is the $\C$-linear extension of the analogously defined real tensor introduced by Loftin-McIntosh (\cite{loftin2013minimal}, see also \cite{loftin2013cubic}) in their study of minimal Lagrangian surfaces in $\CH^{2}$. In particular, since $-iC^{+}$ must attain real values on triples of real tangent vectors, in view of Theorem \ref{thm:pair_hol_diff}, we must have $C^{+}=q-\bar{q}$ for some $J$-holomorphic cubic differential $q$. 
\end{remark}

\subsection{The area functional of real minimal Lagrangians}
The aim of this section is to show that minimal Lagrangians in $\mathbb{H}^{2}_{\tau}$ are critical points of the area functional and in particular that they locally maximize area. This phenomenon is natural in the context of space-like surfaces immersed in pseudo-Riemannian manifolds of negative sectional curvature. Adapting the proof of Lemma \ref{lm:curvature_bicomplex}, it is easily seen that the para-complex hyperbolic space has constant and negative (real) para-holomorphic sectional curvature. Being in a pseudo-Riemannian context this is not sufficient to conclude that all sectional curvatures are negative (as is the case of $\mathbb{CH}^2$, for instance (\cite{goldman1999complex})). Nevertheless, by exploiting the fact that the surfaces are also Lagrangian we are able to obtain a similar result in this context where the sectional curvature has variable sign. \\ \\ In the following we will use the same notation as in the previous section. we will denote by $g_I$ the positive-definite metric induced on the tangent bundle and by $g_N$ the negative-definite metric induced on the normal bundle. Let $\xi$ be a normal vector field with compact support and let $\{\sigma_t\}_{t\in(-\varepsilon,\varepsilon)}$ be a compactly supported smooth deformation of $\sigma_0:=\sigma$ such that $\dot\sigma:=\frac{\mathrm d}{\mathrm dt}\big|_{t=0}\sigma_t(x)=\xi$. Suppose also that for small $t$ the image $\sigma_t(U)$ is space-like. The \emph{area functional} along the variation $\sigma_t$ is defined as $$\mathcal A_K(\sigma_t):=\int_K\mathrm d A_{g_t} \ ,$$ where $K\subset U$ is the compact subset over which $\sigma_t$ is supported, and $\mathrm dA_{g_t}$ is the area form of the induced metric on $\sigma_t(U)$. \begin{lemma}
A space-like surface $\sigma_0:U\to\mathbb{H}^2_\tau$ is a critical point of the area functional, i.e. $$\frac{d}{\mathrm dt}\Big|_{t=0}\mathcal A_K(\sigma_t)=0 \ ,$$ if and only if $\trace_{g_N}\II=0$.
\end{lemma}
\begin{lemma}\label{lm:secondordervariation}
Let $\sigma:U\to\mathbb H^2_\tau$ be a space-like minimal (i.e. $\trace_{g_N}\II=0$) surface and let $\xi\in\Gamma(NU)$ be a non-trivial normal deformation. Then, $$\frac{\mathrm d^2}{\mathrm dt^2}\Big|_{t=0}\mathcal{A}_K(\sigma_t)=2\int_K\trace_{g_I}(\mathcal T_\xi)\mathrm d A_{g_I} \ ,$$ where $\mathcal T_\xi$ is the following symmetric tensor $$\mathcal{T}_\xi(X,Y):=R^I(\xi,X,\xi,Y)-g_I\big(B_\xi(X),B_\xi(Y)\big)+g_N\big(\nabla^N_X\xi,\nabla^N_Y\xi\big) \ , \quad X,Y\in\Gamma(TU) \ .$$ The element $R^I$ is the $(0,4)$ Riemann tensor of the pull-back of the Levi-Civita connection of $\mathbb H^2_\tau$, $B_\xi(\cdot)$ is the shape operator and $\nabla^N$ is the connection induced on the normal bundle.
\end{lemma}
\noindent A detailed proof of the previous two results can be found in \cite[Lemma 3.22 and Proposition 3.23]{LTW} (see also \cite[\S 1.3]{anciaux2011minimal}). In fact, the argument is general and applies to space-like surfaces immersed in pseudo-Riemannian manifolds of signature $(2,n)$ without any assumptions about curvature. 
\begin{theorem}
Let $\sigma:U\to\mathbb H^2_\tau$ be a space-like minimal Lagrangian surface and let $\xi\in\Gamma(NU)$ be a non-trivial normal deformation supported on a compact subset $K\subset U$. Then,
$$\frac{\mathrm d^2}{\mathrm dt^2}\Big|_{t=0}\mathcal{A}_K(\sigma_t)<0 \ .$$
\end{theorem}\begin{proof}
Since the immersed surface is Lagrangian, the para-complex structure $\mathbf P$ of $\mathbb H_\tau^2$ induces an isomorphism  of vector bundles $NU\cong\mathbf P\big(TU\big)$ (see Lemma \ref{lm:normal_Lagrangian}). In particular, we can write $\xi=\mathbf P Z$ for some $0\neq Z\in\Gamma(TU)$. Moreover, being in a space of constant para-holomorphic curvature (Lemma \ref{lm:curvature_bicomplex} and Remark \ref{rmk:scaling}), the Riemann tensor is greatly simplified (Proposition \ref{prop:constant_parahol_curvature}). Therefore, given $X,Y\in\Gamma(TU)$ we compute \begin{align*}
    R^I(\xi,X,\xi,Y)&=R^I(\mathbf PZ,X ,\mathbf PZ,Y) \\ &=g(\mathbf PZ,\mathbf PZ)g(X,Y)-g(X,\mathbf PZ)g(\mathbf PZ,Y)+g(\mathbf PZ,\mathbf P^2Z)g(\mathbf PX,Y) \\ & \ \ \ -g(X,\mathbf P^2Z)g(\mathbf P^2Z,Y)+2g(\mathbf PZ,\mathbf PX)g(\mathbf P^2Z,Y) \\ &=-g_I(Z,Z)g_I(X,Y)-3g_I(X,Z)g_I(Y,Z) \ ,
\end{align*} where we used that $\mathbf P^2=\Id, \ g(\mathbf P\cdot,\mathbf P\cdot)=-g(\cdot,\cdot)$ and $g(\mathbf PX,Y)=0$ for all $X,Y\in\Gamma(TU)$. Taking the trace with respect to $g_I$, we get $$\trace_{g_I}\big(R^I(\mathbf PZ,\cdot,\mathbf PZ,\cdot)\big)=-2g_I(Z,Z)-3\big(g_I(e_1,Z)^2+g_I(e_2,Z)^2\big)<0 \ ,$$where $\{e_1,e_2\}$ is a local $g_I$-orthonormal basis. Regarding the other two terms appearing in the tensor $\mathcal{T}_{\mathbf PZ}$ of Lemma \ref{lm:secondordervariation}, we conclude $$\trace_{g_I}\big(-g_I(B_{\mathbf PZ}(\cdot),B_{\mathbf PZ}(\cdot))+g_N(\nabla^N_{\cdot}\mathbf PZ,\nabla^N_\cdot\mathbf PZ)\big)<0 \ ,$$ since $g_I$ (resp. $g_N$) is positive-definite (resp. negative definite). As a consequence, the second order variation of the area functional is strictly negative $$\frac{\mathrm d^2}{\mathrm dt^2}\Big|_{t=0}\mathcal{A}_K(\sigma_t)=2\int_K\trace_{g_I}(\mathcal T_\xi)\mathrm d A_{g_I}<0 \ .$$
\end{proof}
\begin{remark}
Unlike the case of surfaces immersed in Riemannian manifolds, space-like minimal Lagrangian in $\mathbb H^2_\tau$ locally maximize area. To be more correct, one should therefore speak of space-like \emph{maximal} Lagrangian surfaces in $\mathbb{H}^{2}_{\tau}$, but we will continue to call them minimal so as not to create too much confusion with the complex case inside $\CH^2_\tau$.
\end{remark}

\section{The geometry of quasi-Fuchsian representations into  \texorpdfstring{$\mathrm{SL}(3,\C)$}{SL(3,C)}}\label{sec:quasi-Hitchin}
\subsection{The representation space} 
Let $S$ be a smooth closed oriented surface of genus $g\ge 2$. We are interested in the set of all group homomorphisms from $\pi_1(S)$ to $\SL(3,\C)$, which will be denoted by $\Hom(\pi_1(S),\SL(3,\C))$. It is endowed with the subspace topology induced by the inclusion \begin{align*}
\Hom(\pi_1(S),&\SL(3,\C))\hookrightarrow\SL(3,\C)^{2g} \\ & \rho \mapsto (\rho(a_1),\dots,\rho(a_g),\rho(b_1),\dots,\rho(b_g) \ , \end{align*}where $a_1,\dots,a_g,b_1,\dots,b_g$ are the generators of $\pi_1(S)$ subject to the relation $\prod_{i=1}^g[a_i,b_i]=1$. There is a natural action of $\SL(3,\C)$ on $\Hom(\pi_1(S),\SL(3,\C))$ by conjugation: $$(X\cdot\rho)(\gamma):=X\rho(\gamma)X^{-1}, \qquad \text{for} \ X\in\SL(3,\C) \ \text{and} \ \gamma\in\pi_1(S) \ . $$ The quotient space endowed with the quotient topology will be denoted with $\chi_3(S)$ and called the \emph{representation space}. Although the topological space we have just introduced in general may not be Hausdorff, we are actually interested in a subspace of $\chi_3(S)$ whose associated representations have special properties. Recall that a representation $\rho:\pi_1(S)\to\SL(3,\C)$ is called \emph{irreducible} if there is no non-zero and proper subspace of $\C^3$ preserved by $\rho(\gamma)$, for all $\gamma\in\pi_1(S)$. Moreover, for any $\rho\in\chi_3(S)$ one can consider the centralizer of $\rho(\pi_1(S))$ defined by $$C\big(\rho(\pi_1(S))\big):=\{X\in\SL(3,\C) \ | \ X\rho(\gamma)X^{-1}=\rho(\gamma), \ \forall \gamma\in\pi_1(S)\} \ .$$ Let us denote by $\chi_3^s(S)$ the space of conjugacy class of representations into $\SL(3,\C)$ which are irreducible and whose centralizer is equal to the center of $\SL(3,\C)$, i.e. the discrete subgroup $\{\Id, \zeta\Id, \zeta^2\Id\}$ where $\zeta$ is a 3-rd root of unity. 

\begin{theorem}[{\cite[Corollary 50]{sikora2012character}}]\label{thm:smoothpointcomplex}
The space $\chi_3^s(S)$ has the structure of a smooth complex manifold of complex dimension $16g-16$.
\end{theorem}

\subsection{Quasi-Hitchin representations}\label{sec:quasiHitchin}
In this section we introduce another set of conjugacy classes of representations into $\SL(3,\R)<\SL(3,\C)$, which will be directly related to higher Teichm\"uller theory (see \cite[\S 2.5]{dumas2020geometry}). \\ 

\noindent Given a Fuchsian representation $\eta:\pi_1(S)\to\PSL(2,\R)$ (discrete and faithful) one can consider the composition $\iota\circ\eta:\pi_1(S)\to\SL(3,\R)$, where $\iota$ is the unique (up to conjugation) irreducible representation of $\PSL(2,\R)$ into $\SL(3,\R)$. In his seminal paper (\cite{hitchin1992lie}) Hitchin found a connected component $\Hit$ in the corresponding representation space for $\SL(3,\R)$ containing the space of Fuchsian representations. Nowadays it is called the \emph{Hitchin component} and it is known to have a structure of smooth manifold diffeomorphic to a ball of real dimension $16g-16$ (\cite{hitchin1992lie}). In particular, any representation in $\Hit$ is discrete and irreducible (\cite{Labourie2006anosov},\cite{fock2006moduli}). 
In fact, it was first observed by Labourie (\cite{Labourie2006anosov}) that Hitchin representations satisfy an additional property, which is now found in the literature under the name of $P$-Anosov, where $P$ is a parabolic subgroup. During the discussion we will focus only in the case where $P=B$ is the unique (up to conjugation) Borel subgroup of $\SL(3,\R)$, namely the subgroup in $\SL(3,\R)$ formed by upper triangular matrices. The quotient space $\SL(3,\R)/B$ is identified with the space of full flags $$\mathcal{F}_3(\R):=\{(l,V) \ | \ l\subset V\subset\R^3, \ \dim l=1, \ \dim V=2\} \ .$$   \begin{defi}[\cite{Labourie2006anosov}]
An irreducible representation $\rho:\pi_1(S)\to\SL(3,\R)$ is $B$-Anosov if there exists a unique injective continuous $\rho$-equivariant map $\xi:\partial_\infty\pi_1(S)\to\mathcal F_3(\R)$, called \emph{boundary map}, such that for any $x\neq y$ in $\partial_\infty\pi_1(S)$ the images $\xi(x)=(l_x,V_x),\xi(y)=(l_y,V_y)$ are transverse, namely $$l_x\oplus V_y=\R^3=l_y\oplus V_x \ .$$ 
\end{defi}
\begin{remark}
The definition of Anosov representation we introduced is a particular case of \cite[Definition 2.10]{guichard2012anosov} where as a pair of opposite parabolic subgroups $(P^+,P^-)$ we are taking $P^+=B$ and $P^-=gBg^{-1}$ for a suitable $g\in\SL(3,\R)$. Moreover, the usual condition of being dynamic preserving is automatically satisfied when considering irreducible and Anosov representations with respect to the minimal parabolic (see \cite[Proposition 4.10]{guichard2012anosov}).
\end{remark}
\begin{prop}[\cite{Labourie2006anosov},\cite{fock2006moduli}]
Every representation $\rho:\pi_1(S)\to\SL(3,\R)$ in the Hitchin component is $B$-Anosov
\end{prop}
\noindent Similarly, if $B^\C<\SL(3,\C)$ denotes the complexification of the Borel subgroup, we have a definition of irreducible $B^\C$-Anosov representations into $\SL(3,\C)$ whose boundary map takes values in the space of complex full flags, i.e. $$\mathcal{F}_3(\C):=\{(l,V) \ | \ l\subset V\subset\C^3, \ \dim_\C l=1, \ \dim_\C V=2\} \ .$$
The space of all (not necessarily irreducible) $B^\C$-Anosov representations into $\SL(3,\C)$ will be denoted by $\chi_3^{B^\C}(S)$. This subset is open in the representation space $\chi_3(S)$ (\cite{Labourie2006anosov}) and consists entirely of discrete and faithful representations (\cite[Theorem 5.3]{guichard2012anosov}).
\begin{defi}
The connected component of $\chi_3^{B^\C}(S)$ containing the (real) Hitchin representations into $\SL(3,\R)$ is called the $B^\C$-\emph{quasi}-\emph{Hitchin component} (or simply \emph{quasi}-\emph{Hitchin component}) and denoted with $\QHit$. In particular, any element $\rho\in\QHit$ will be called a \emph{quasi}-\emph{Hitchin} representation.
\end{defi}
\noindent In addition to the Hitchin component, this space also contains other types of representations that relate directly to some special cases of complex Lagrangian minimal surfaces in $\CH^2_\tau$. In fact, in Remark \ref{rmk:min_Lagrangian_complex}, we observed that if the immersed surface is contained in the totally-geodesic copy of $\CH^2$ then we recover the minimal Lagrangian surfaces studied by Loftin-McIntosh (\cite{loftin2013minimal},\cite{loftin2019equivariant}). In particular, they showed that if a representation $\rho:\pi_1(S)\to\SU(2,1)$ is sufficiently close to the Fuchsian locus, then there is a unique such surface, and they referred to those representations as \emph{almost $\R$-Fuchsian}.
\begin{lemma}\label{lm:subsets}
There is an open neighborhood of the real Fuchsian locus in the quasi-Hitchin component that contains representations of the following type:\begin{itemize}
    \item[$\bullet$] almost $\R$-Fuchsian representations into $\mathrm{SU}(2,1)$; \item[$\bullet$] quasi-Fuchsian representations into $\PSL(2,\C)$ post-composed with the irreducible principal embedding $\iota^\C:\PSL(2,\C)\hookrightarrow\PSL(3,\C)$.
    \end{itemize}
\end{lemma}
\begin{proof}
First notice that almost $\R$-Fuchsian representations into $\SU(2,1)$ are convex co-compact (\cite{loftin2019equivariant}). Since $\SU(2,1)$ is a Lie group of rank one, by a result of Guichard-Wienhard (\cite[Theorem 5.15]{guichard2012anosov}), this is equivalent for the representations of being $P$-Anosov with respect the unique (up to conjugation) parabolic subgroup. In particular, they are Anosov with respect to the Borel subgroup of $\SU(2,1)$ which is given by $B^\C\cap\SU(2,1)$. It is therefore clear that they land into $\chi_3^{B^\C}(S)$ when thought of with values in $\SL(3,\C)$. Regarding representations of the form $\iota^\C\circ\eta$, where $\eta:\pi_1(S)\to\PSL(2,\C)$ is quasi-Fuchsian, Dumas and Sanders \cite[Proposition 2.8]{dumas2020geometry} proved that they are $P$-Anosov for any parabolic $P<\SL(3,\C)$, therefore in particular when $P=B^\C$.
\end{proof}
\begin{prop}\label{prop:smooth_loci}
Any irreducible representation in $\QHit$ is a smooth point of the representation space $\chi_3(S)$. In particular, the Hitchin component and the set of representations described in Lemma \ref{lm:subsets} are smooth submanifolds of $\QHit$.
\end{prop}
\begin{proof}Any irreducible $B^\C$-Anosov representation is totally loxodromic, hence for all $\gamma\in\pi_1(S)$ the matrix $\rho(\gamma)$ is diagonalizable with different eigenvalues. 
 Therefore, given $\gamma\in\pi_1(S)$ there exists a basis of $\C^3$ where $\rho(\gamma)=\diag(\lambda_1^\gamma,\lambda_2^\gamma,\lambda_3^\gamma)$, with $\lambda_i^\gamma\neq\lambda_j^\gamma$ for $i\neq j$. Let $X$ be an element in the centralizer of $\rho(\pi_1(S))<\SL(3,\C)$, then $X$ commutes with the diagonal matrix $\rho(\gamma)$. In particular, $X$ is diagonal in the same basis. On the other hand, there must exist a $\gamma'\neq\gamma$ such that $\rho(\gamma')$ is not diagonal in the fixed basis but commutes with $X$. If all the elements of $X$ are different, then a simple calculation with matrices shows that $\rho(\gamma')$ must be diagonal in the same basis, but this is not possible because of the way $\gamma'$ was chosen. If there are two equal elements in $X$, repeating the computation for all $\gamma'\neq\gamma$ as above, would result in $\rho$ preserving a proper, non-trivial subspace of $\C^3$, which is not possible since $\rho$ is irreducible. So, the only possibility is that all elements on the diagonal of $X$ are equal, i.e. $X=\mu\mathrm{Id}\in\SL(3,\C)$, with $\mu=\zeta$ the 3-rd root of unity. In other words, we proved that the centralizer of $\rho(\pi_1(S))$ is contained in the center of $\SL(3,\C)$. According to Theorem \ref{thm:smoothpointcomplex}, the proof of the first part is completed, the other inclusion being trivial. Moreover, given a Hitchin representation $\rho$ thought with values in $\SL(3,\C)$, it is sufficient to notice that it is still irreducible. In fact, if there were a subspace $\{0\}\neq V\subsetneq\C^3$ preserved by $\rho$, then according to the decomposition $V=W\oplus iW$, one would have that $W$ is a non-zero real proper subspace of $\R^3$ preserved by the real Hitchin representation, but this is not possible $\rho$ being irreducible into $\SL(3,\R)$ (\cite{Labourie2006anosov}). Finally, almost $\R$-Fuchsian representations are irreducible (\cite{loftin2019equivariant}) and the same holds for $\iota^\C\circ\eta$ whenever $\eta:\pi_1(S)\to\PSL(2,\C)$ is quasi-Fuchsian.
\end{proof}

\subsection{Existence of solutions near the Hitchin locus}\label{sec:existencenearFuchsian} Let $\mathcal{J}$ be $\C$-complex structure on a closed surface $S$. We know that $\mathcal{J}$ can equivalently be described by a pair of complex structures $(J_{1}, J_{2})$ on $S$. By Lemma \ref{lm:C_compatible} and Theorem \ref{thm:pair_hol_diff}, the space of Codazzi tensors compatible with $\mathcal{J}$ consists of pairs $(q_{1}, \bar{q}_{2})$ where $q_{i}$ is a $J_{i}$-holomorphic cubic differential. We denote by $\mathcal{Q}_{3}^{\C}(S)$ the bundle over $\mathcal{C}^{\C}(S)$ where the fiber over $\mathcal{J}$ is given by all Codazzi tensors compatible with $\mathcal{J}$. It follows from the above discussion that $\mathcal{Q}^{\C}_{3}(S)$ is homeomorphic to the product $\mathcal{Q}_{3}(S)\times \mathcal{Q}_{3}(S)$ of two copies of the bundle of holomorphic cubic differentials over the space of complex structures on $S$. The \emph{Hitchin locus} inside $\mathcal{Q}^{\C}_{3}(S)$ is the pre-image of the diagonal in $\mathcal{Q}_{3}(S)\times \mathcal{Q}_{3}(S)$ under the isomorphism described previously. \\

\noindent By the results in Section \ref{sec:min_Lagrangian}, to any equivariant complex Lagrangian minimal immersion $\tilde{\sigma}:\tilde{S} \rightarrow \CH^{2}_{\tau}$ with positive first fundamental form $h$, we can associate a point $(\mathcal{J}, C) \in \mathcal{Q}^{\C}_{3}(S)$: the $\C$-complex structure $\mathcal{J}$ is the unique one compatible with $h$ and $C=-h(\mathbf{P}\II(\cdot, \cdot), \cdot)$, where $\II$ is the second fundamental form of $\tilde{\sigma}(\tilde{S})$. When $\tilde{\sigma}$ is actually contained inside the real para-complex hyperbolic space $\mathbb{H}^{2}_{\tau}$, then $(\mathcal{J},C)$ belongs to the Hitchin locus and all such pairs can be realized as embedding data. 

\begin{theorem}\label{thm:solutions} There is a neighborhood $\mathcal{U}$ of the Hitchin locus in $\mathcal{Q}_{3}^{\C}(S)$ such that all pairs $(\mathcal{J}, C) \in \mathcal{U}$
can be realized as embedding data of an equivariant minimal Lagrangian minimal surface in $\CH^{2}_{\tau}$. 
\end{theorem}
\begin{proof} For $\mathcal{J} \in \mathcal{C}^{\C}(S)$, let $g_{\mathcal{J}}$ be the unique positive complex metric of constant negative curvature $-1$ (\cite{ElEmam_quasiFuchsian}), called \emph{Bers metric}. It is sufficient to prove that, for any $(\mathcal{J}_{0},C_{0})$ in the Hitchin locus, there is a neighborhood $\mathcal{U}$ such that for all $(\mathcal{J},C)\in \mathcal{U}$ there is a unique conformal metric $h=e^{2\psi}g_{\mathcal{J}}$ that satisfies Gauss' equation
\begin{equation}\label{eq:Gauss-proof}
    K_{h} = -1 +8\|C\|^{2}_{h} \ . 
\end{equation}
Indeed, by Theorem \ref{thm:Gauss_Codazzi_sufficient}, this implies that the pair $(h,C)$ can be realized as embedding data of a complex Lagrangian minimal immersion. By Corollary \ref{cor:curvature_conformal} and Proposition \ref{prop:norm_C}, Equation \eqref{eq:Gauss-proof} can be re-written as
\[
    \Delta_{g_{\mathcal{J}}}\psi - e^{2\psi} + 8e^{-4\psi}\|C\|^{2}_{g_{\mathcal{J}}}+1 = 0 \ .
\]
Note that this equation has a unique smooth real solution $\psi_{0}$ at $(\mathcal{J}_{0},C_{0})$, because it reduces to Wang's equation on points in the Hitchin locus (see Corollary \ref{cor:Wang}). For $k>2$ and $p\in (1,+\infty)$, we consider the map
\begin{align*}
    F: \mathcal{Q}^{\C}_{3}(S) \times W^{k,p}(S,\C) &\rightarrow W^{k-2,p}(S,\C) \\
        (\mathcal{J},C, \psi) &\mapsto \Delta_{g_{\mathcal{J}}}\psi - e^{2\psi} + 8e^{-4\psi}\|C\|^{2}_{g_{\mathcal{J}}}+1 \ . 
\end{align*}
We observe that the Fr\'echet differential of $F$ with respect to the second variable at the point $(\mathcal{J}_{0}, C_{0}, \psi_{0})$ is
\begin{align*}
    dF_{2}(\dot{\psi}) &= \Delta_{g_{\mathcal{J}_{0}}}\dot{\psi} - 2\dot{\psi}e^{2\psi_{0}} -32\dot{\psi}e^{-4\psi_{0}}\|C_{0}\|^{2}_{g_{\mathcal{J}_{0}}} \\
    &= (\Delta_{g_{\mathcal{J}_{0}}} -2e^{2\psi_{0}}-32e^{-\psi_{0}}\|C_{0}\|^{2}_{g_{\mathcal{J}_{0}}}) \dot{\psi}
\end{align*}
Since $g_{\mathcal{J}_{0}}$ is a Riemannian metric, its Laplacian $\Delta_{g_{\mathcal{J}_{0}}}$ is a negative-definite operator. Hence, $dF_{2}$ is invertible at $(\mathcal{J}_{0}, C_{0}, \psi_{0})$ and, by the implicit function theorem for Banach spaces, there is a neighborhood $\mathcal{U}$ of $(\mathcal{J}_{0}, C_{0}, \psi_{0})$ and a function $f:\mathcal{U} \rightarrow W^{k,p}(S,\C)$ such that, for all $(\mathcal{J}, C) \in \mathcal{U}$, we have $F(\mathcal{J}, C, f(\mathcal{J},C))=0$. In other words, for every $(\mathcal{J},C) \in \mathcal{U}$ there is a weak solution $h=e^{2\psi}g_{\mathcal{J}}$ to Equation $\eqref{eq:Gauss-proof}$. The following lemma implies that $\psi$ is actually smooth, so that the pair $(h,C)$ is the embedding data of a smooth complex Lagrangian minimal surface in $\CH^{2}_{\tau}$. 
\end{proof}

\begin{lemma} Let $h$ be a positive complex metric on $S$. Then the Laplacian operator $\Delta_h: W^{k,p}(S,\C) \rightarrow W^{k-2,p}(S,\C)$ is elliptic of index $0$, for every $k>2$ and $p\in (1, +\infty)$.     
\end{lemma} 
\begin{proof} Let $\mathcal{J}=(J_{1},J_{2})$ be the $\C$-complex structure compatible with $h$. Let $z$ and $w$ be local holomorphic coordinates on $S$ for the complex structures $J_{1}$ and $J_{2}$ respectively. Let $\mu$ be the Beltrami differential of $w$ with respect to the local coordinate $z$. Note that $\|\mu\|_{\infty}<1$. Recalling that $\pw = (\pw z)(\pz +\bar{\mu}\pbz)$, we can re-write the local expression of the Laplacian $\Delta_{h}$ found in Proposition \ref{prop:laplacian} in terms of $\mu$:
\begin{align*}
    \Delta_{h}\phi &= \frac{e^{-2\psi}}{(\partial_{\bar{z}}\bar{w})(\partial_{w}z)}\bigg[\phi_{w\bar{z}}+\phi_{\bar{z}w}-\partial_{w}(\log(\partial_{\bar{z}}\bar{w}))\phi_{\bar{z}}-\partial_{\bar{z}}(\log(\partial_{w}z))\phi_{w}\bigg] \\
    &= \frac{e^{-2\psi}}{(\partial_{\bar{z}}\bar{w})(\partial_{w}z)}\bigg[2(\pw z)(\phi_{z\bar{z}}+\bar{\mu}\phi_{\bar{z}\bar{z}})-2\partial_{w}(\log(\partial_{\bar{z}}\bar{w}))\phi_{\bar{z}}\bigg]  \tag{Lemma \ref{lm:commutator}}\\
    &= \frac{2e^{-2\psi}}{\partial_{\bar{z}}\bar{w}}\bigg[\phi_{z\bar{z}}+\bar{\mu}\phi_{\bar{z}\bar{z}}-\bar{\mu}_{\bar{z}}\phi_{\bar{z}}\bigg] \ .
\end{align*}
If $z=x+iy$ and $\xi=\xi_{1}dx + \xi_{2}dy$ is a (real) $1$-form on $S$, then the principal symbol of $\Delta_{h}$ is equal, up to nonzero multiplicative constants, to
\begin{align*}
    \sigma(\Delta_{h})(\xi) &= \xi_{1}^{2}+\xi_{2}^{2}+\bar{\mu}(\xi_{1}^{2}-\xi_{2}^{2}+2i\xi_{1}\xi_{2}) \\
    &=[\xi_{1}^{2}+\xi_{2}^{2} +\Ree(\mu)(\xi_{1}^{2}-\xi_{2}^{2}) + 2\Ima(\mu)\xi_{1}\xi_{2}]+ i[\Ree(\mu)\xi_{1}\xi_{2}-\Ima(\mu)(\xi_{1}^{2}-\xi_{2}^{2})].
\end{align*}
It is an elementary computation to show that $\sigma(\Delta_{h})(\xi)=0$ only if $\xi=0$ or $|\mu|=1$. Therefore, under our assumptions, the operator $\Delta_{h}$ is elliptic. Moreover, it has index $0$, since we can connect $\Delta_{h}$ with the classic Laplacian of a Riemannian metric by considering the family of Beltrami differentials $t\mu$ with $t\in [0,1]$. 
\end{proof}

\begin{remark}\label{rmk:more_solutions} We observe that a solution to Equation \eqref{eq:Gauss-proof} is known to exist also at other points outside the neighborhood $\mathcal{U}$ provided by Theorem \ref{thm:solutions}. Indeed, when $C=0$, in other words at points in the zero section of $\mathcal{Q}^{\C}_{3}(S)$, Gauss' equation reduces to $K_{h}=-1$, i.e. it is sufficient to find a Bers metric compatible with the given $\C$-complex structure $\mathcal{J}$. It follows from \cite{bonsante2022immersions} that this is always possible. Moreover, when $\mathcal{J}=(J, J)$ and $C=q-\bar{q}$, then, by Remark \ref{rmk:min_Lagrangian_complex}, the complex Lagrangian minimal surface with these embedding data (if it exists) lives inside $\mathbb{CH}^{2}$. By work of Loftin-McIntosh (\cite{loftin2019equivariant}), a solution to Equation \eqref{eq:Gauss-proof} exists as long as the norm of the cubic differential $q$ is sufficiently small. Under those assumptions, one can actually repeat the argument of Theorem \ref{thm:solutions} and find that there is a neighborhood $\mathcal{U}'$ of the set $\{(\mathcal{J}, C) \ | \ \mathcal{J}=(J,J),  \ \ C=q-\bar{q}, \ \ \|q\|_{J}<\epsilon\}$ in which Equation \eqref{eq:Gauss-proof} has a unique smooth solution.    
\end{remark}

\subsection{Parameterization of quasi-Fuchsian representations}\label{sec:parametrFuchsian} By Theorem \ref{thm:solutions}, we can define a map
\[
    \widetilde{\mathrm{hol}}: \mathcal{U} \rightarrow \Hom(\pi_{1}(S), \SL(3,\C))/\SL(3,\C)
\]
that associates to any $(\mathcal{J}, C) \in \mathcal{U}$ the holonomy of the unique equivariant complex Lagrangian minimal surface in $\CH^{2}_{\tau}$ with complex first fundamental form $h$ compatible with $\mathcal{J}$ and complex second fundamental form $\II = -\mathbf{P}h^{-1}C$. \\

\noindent In order to promote this map to a local diffeomorphism, we need to restrict the domain. We consider the subset $\mathcal{R}\subset \mathcal{Q}^{\C}_{3}(S)$ of pairs $(\mathcal{J}, C)$ such that, if $\mathcal{J}=(J_{1},J_{2})$ and $g_{1}$ and $g_{2}$ denote the unique hyperbolic metric in the conformal class of $J_{1}$ and $J_{2}$, then the identity $(S,g_{1}) \rightarrow (S,g_{2})$ is harmonic. We call such pairs $(\mathcal{J}, C)$ \emph{renormalized pairs}. Note that the space of renormalized pairs is still infinite dimensional and the group $\Diff_{0}(S)$ of diffeomorphisms of $S$ isotopic to the identity acts diagonally on $\mathcal{R}$ by pull-back with quotient $Q^{\C}_{3}(S):=Q_{3}(S)\times Q_{3}(S)$ isomorphic to two copies of the bundle of holomorphic cubic differentials over the Teichm\"uller space of $S$. We call \emph{Fuchsian locus} the pre-image in $Q^{\C}_{3}(S)$ of the diagonal inside $\mathrm{Teich}(S)\times \mathrm{Teich}(S) \subset Q_{3}(S) \times Q_{3}(S)$ under the isomorphism described previously. The map $\widetilde{\mathrm{hol}}$ clearly descends to the quotient and defines a map
\[
    \mathrm{hol}: U:=(\mathcal{U}\cap\mathcal{R})/\Diff_{0}(S) \rightarrow \Hom(\pi_{1}(S), \SL(3,\C))/\SL(3,\C) \ . 
\]

\begin{theorem}\label{thm:holonomy_map} There exists a neighborhood $N\subset Q^{\C}_{3}(S)$ of the Fuchsian locus such that the holonomy map 
\[
    \mathrm{hol}: N \rightarrow \Hom(\pi_{1}(S), \SL(3,\C))/\SL(3,\C) 
\]
is a diffeomorphism onto its image. Moreover, all representations in the image of $\mathrm{hol}$ are $B^{\C}$-Anosov and hence quasi-Hitchin. 
\end{theorem}

\noindent The proof uses a symplectic form on the representation space due to Goldman \cite{goldman1984symplectic}, which we describe briefly. The representation space $\chi_{3}(S)$ is in bijection with
the moduli space of flat principal $\SL(3,\C)$-bundles $P$ over $S$. At a smooth point the tangent space can be identified with
the cohomology $H^{1}(S,ad(P))$, where $ad(P)$ is the associated flat $\mathfrak{sl}(3,\C)$-bundle. The cohomology is de-Rham cohomology with respect to the flat connexion $d_{P}$ on $ad(P)$. Goldman’s symplectic form is defined as follows. For any pair of $d_{P}$-closed 1-forms $Z, W \in \Omega^{1}(S, ad(P))$ representing cohomology classes $[Z], [W] \in H^1(S,ad(P))$, he shows that
\[
    \omega_{P}([Z],[W]) = \int_{S} \trace(Z\wedge W)
\]
is well-defined and symplectic at each smooth point $P$. In particular, suppose $P$ is the flat principal bundle whose connection is determined by the $\mathfrak{sl}(3,\C)$-valued $1$-form $\Omega$ as in Section \ref{sec:bicomplex_Higgs}. When $X \in T_{(\mathcal{J},C)}Q^{\C}_{3}(S)$ is
tangent to a curve $\gamma(t)$ in $Q^{\C}_{3}(S)$ at $t = 0$, its push-forward to $H^{1}(S,ad(P))$ is represented by the first variation $\delta_{X}\Omega$ of $\Omega$ along this curve, as a $\mathfrak{sl}(3,\C)$-valued $1$-form on
$S$. Thus we have
\[
    \mathrm{hol}^{*}\omega_{(\mathcal{J},C)}(X_{1},X_{2}) = \int_{S} \trace(\delta_{X_{1}}\Omega \wedge \delta_{X_{2}}\Omega) \ .
\]
\begin{proof} The proof proceeds by using Goldman’s symplectic form and the Inverse Function Theorem at the Fuchsian locus $\{ (\mathcal{J}, 0) \in U \ | \ \mathcal{J}=(J,J)\}$. First, we observe that, up to restricting the open set $U$, we can assume that $\mathrm{hol}(U)$ consists entirely of $B^{\C}$-Anosov representations, since being $B^{\C}$-Anosov is an open condition and the image under $\mathrm{hol}$ of the Fuchsian locus consists of discrete and faithful representations in $\SL(2,\R)$ seen inside $\SL(3,\C)$ via the unique irreducible embedding into $\SL(3,\R)$ and the natural inclusion $\SL(3,\R)\hookrightarrow \SL(3,\C)$ (see Section \ref{sec:submanifolds}, Remark \ref{rmk:X} and Lemma \ref{lm:subsets}), which are $B^{\C}$-Anosov. In particular, the image of $\mathrm{hol}$ takes values in the smooth locus $\chi_{3}^{s}(S)$ of $\chi_{3}(S)$ by Proposition \ref{prop:smooth_loci}. Note that $Q^{\C}_{3}(S)$ and $\chi_{3}^{s}(S)$ have the same complex dimension by an easy application of Riemann-Roch and Theorem \ref{thm:smoothpointcomplex}. Therefore, the statement follows if we show that at any point in $U$ of the form $(\mathcal{J},0)$, with $\mathcal{J}=(J,J)$, the differential of $\mathrm{hol}$ is a linear isomorphism. To this aim, we prove that $\mathrm{hol}^{*}\omega$ is non-degenerate. The tangent space $T_{(\mathcal{J},0)}Q^{\C}_{3}(S)$ can be split into a horizontal part, corresponding to variations $(\dot{\mathcal{J}},0)$ such that $\dot{\mathcal{J}}=(\dot{J}_{1}, \dot{J}_{2})$ are both non-trivial in Teichm\"uller space, and a vertical part consisting of variations of the form $(0,\dot{C})$ where $\dot{C}=\dot{q}_{1}+\dot{\bar{q}}_{2}$ where $\dot{q}_{i}$ are $J$-holomorphic cubic differentials. Non-degeneracy of $\mathrm{hol}^{*}\omega$ follows then from these three claims:
\begin{enumerate}
    \item given a vertical variation $X_{1}=(0,\dot{C})$, there exists another vertical variation $X_{2}$ such that the pairing $\mathrm{hol}^{*}\omega_{(\mathcal{J},C)}(X_{1},X_{2})$ is not zero;
    \item for all horizontal variations $X_{h}=(\dot{J},0)$ and vertical variations $X_{v}=(0,\dot{C})$, we have $\mathrm{hol}^{*}\omega_{(\mathcal{J},C)}(X_{1},X_{2})=0$;
    \item the pull-back $\mathrm{hol}^{*}\omega$ is non-degenerate when evaluated on pairs of horizontal variations.
\end{enumerate}
To prove the first claim, we need to compute the variations of the matrices $\hat{A}$ and $\hat{B}$ introduced in Section \ref{sec:bicomplex_Higgs} along a vertical variation. In particular, if we fix a background hyperbolic metric $g_{\mathcal{J}}$ compatible with $\mathcal{J}=(J,J)$, we need to understand the derivative at $t=0$ of the conformal factors $\psi_{t}$, which satisfy Gauss' equation
\[
    \Delta_{g_{\mathcal{J}}}\psi_{t} -e^{2\psi_{t}} + 8e^{-4\psi_{t}}\|C(t)\|^{2}_{g_{\mathcal{J}}}+1=0 \ ,
\]
along the path $C(t)=t\dot{C}$. Differentiating the above equation and evaluating at $t=0$, we obtain
\[
    0=\Delta_{g_{\mathcal{J}}} \dot{\psi} - 2e^{2\psi_{0}}\dot{\psi} = (\Delta_{g_{\mathcal{J}}} -2)\dot{\psi} 
\]
because $\psi_{0}=0$ by our choice of background metric. Since $\Delta_{g_{\mathcal{J}}}$ is negative definite, we deduce that $\dot{\psi}=0$. Therefore, the variation of the matrix connection $\Omega$ along the path $C(t)$ is given by
\begin{equation}\label{eq:vertical}
    \delta_{X_{1}}\Omega = \begin{pmatrix} 
                                0 & -\frac{1}{s^{2}}\tau\dot{q}_{1} & 0 \\
                                0 & 0 & 0 \\
                                0 & 0 & 0 
                            \end{pmatrix}dz +
                            \begin{pmatrix}
                                0 & 0 & 0 \\
                                -\frac{1}{s^{2}}\tau \dot{\bar{q}}_{2}& 0 & 0 \\
                                0 & 0 & 0 
                            \end{pmatrix}d\bar{z}
\end{equation}
Since $\dot{C}$ is non-trivial, we can assume, without loss of generality, that $\dot{q}_{1} \neq 0$. We consider then the vertical variation $X_{2}=(0,\dot{C}')$ with $\dot{C}'=i\dot{\bar{q}}_{1}$. We obtain
\begin{align*}
    \mathrm{hol}^{*}\omega_{(\mathcal{J},0)}(X_{1}, X_{2}) &= \int_{S} \trace(\delta_{X_{1}}\Omega \wedge \delta_{X_{2}}\Omega) = i\int_{S}  \frac{\dot{q}_{1}\dot{\bar{q}}_{1}}{g_{\mathcal{J}}^{2}} dz \wedge d\bar{z} \\ 
    & =2\int_{S} \|\dot{q}_{1}\|_{g_{\mathcal{J}}}^2 dA_{g_{\mathcal{J}}} >0
\end{align*}
Claim $(2)$ follows from the easy observation that for any horizontal variation $X_{h}$, since we are evaluating at a point on the zero section, the corresponding variation of the matrix connection is of the form
\[
    \delta_{X_{h}}\Omega = \begin{pmatrix} 
                                * & 0 & 0 \\
                                0 & * & * \\
                                * & 0 & 0 
                            \end{pmatrix}dz +
                            \begin{pmatrix}
                                * & 0 & * \\
                                0 & * & 0 \\
                                0 & * & 0 
                            \end{pmatrix}d\bar{z} \ ,
\]
where $*$ denotes an unknown entry, which pairs to zero with all vertical variations as in Equation \eqref{eq:vertical}. \\
Finally, when we consider two horizontal variations at a point in the Fuchsian locus, the corresponding complex Lagrangian minimal surfaces are embedded in $\mathbb{X} \subset \CH^{2}_{\tau}$ and the connections $\Omega_{t}$ are simply the Levi-Civita connections of their complex-valued first fundamental form, which is a Bers metric. The pull-back of Goldman's symplectic form in this context has already been computed by El-Emam (\cite{ElEmam_quasiFuchsian}), who showed that this pairing is non-degenerate.
\end{proof}

\noindent Representations in $\mathrm{hol}(N)$ will be called $\mathrm{SL}(3,\C)$-\emph{quasi-Fuchsian}, since the corresponding complex Lagrangian minimal surfaces have embedding data close to the Fuchsian locus in $Q_{3}(S)$, which corresponds to surfaces with Fuchsian holonomy in a copy of $\PSL(2,\R)$ embedded irreducibly into $\SL(3,\R)$ and then included canonically in $\SL(3,\C)$. Since $Q_{3}^{\C}(S)$ has a natural bi-complex structure and the map $\mathrm{hol}$ is a diffeomorphism in a neighborhood of the Fuchsian locus, we deduce the following: 

\begin{cor} The space of $\mathrm{SL}(3,\C)$-quasi-Fuchsian representations sufficiently close to the Fuchsian locus has a bi-complex structure.
\end{cor}

\noindent We plan on investigating the properties of this bi-complex structure and its relation with Goldman symplectic form in a future work. For now, we simply observe that the map $\mathrm{hol}$ can actually be promoted to a biholomorphism onto its image, if we equip $Q_{3}^{\C}(S)=Q_{3}(S) \times Q_{3}(S)$ with the complex structure that is given by the standard one in the first factor and by its conjugate in the second. Indeed, with this choice, the map $F$ defined in the proof of Theorem \ref{thm:solutions} is holomorphic between complex Banach spaces (see also the local expression of the Laplacian operator and of the norm of the tensor $C$ in Proposition \ref{prop:laplacian} and Proposition \ref{prop:norm_C}). Hence, by the implicit function theorem for holomorphic functions, the solution to Gauss' equation depends holomorphically on the data $\mathcal{J}=(J_{1},J_{2})$ and $C=q_{1}+\overline{q}_{2}$. Then, the same argument used in \cite[Theorem 5.1]{bonsante2022immersions} can be adapted word-by-word to this setting and conclude that the holonomy representation depends holomorphically on the pair $(\mathcal{J}, C)$. 

\subsection{Relation with Anosov representations and boundary maps}\label{sec:boundarymaps} Being $B^{\C}$-Anosov, an $\SL(3,\C)$-quasi-Fuchsian representation $\rho$ comes equipped with a boundary map $\xi:\partial_{\infty}(\pi_{1}(S)) \rightarrow \mathcal{F}_{3}(\C)$. By Section \ref{sec:incidence_geometry} and \ref{sec:boundary}, the space of projectivized flags $\mathbb{P}\mathcal{F}_{3}(\C)$ is naturally identified with the boundary at infinity of $\CH^{2}_{\tau}$ in an $\SL(3,\C)$-equivariant way via the map
\begin{align*}
    \beta : \partial_{\infty}\CH^{2}_{\tau} &\rightarrow \mathbb{P}\mathcal{F}_{3}(\C) \\
            (v, \varphi) &\mapsto (v, \Ker(\varphi)) \ . 
\end{align*}
On the other hand, if $\rho$ is in the neighborhood described in Theorem \ref{thm:holonomy_map}, then there is a unique $\rho$-equivariant complex Lagrangian minimal immersion $\tilde{\sigma}:\tilde{S} \rightarrow \CH^{2}_{\tau}$, whose boundary at infinity we relate to the image of $\xi$:

\begin{prop} The $\rho$-equivariant complex Lagrangian minimal immersion $\tilde{\sigma}$ bounds the image of the boundary map $\beta^{-1}\circ \xi$.    
\end{prop}
\begin{proof} For this proof it will be convenient to identify $\CH^{2}_{\tau}$ with $\mathbb{P}\mathcal{L}$ as in Proposition \ref{prop:incidence}. Let $\gamma \in \pi_{1}(S)$ be a non-trivial element. Since $\rho$ is $B^{\C}$-Anosov, the strong dynamical property of the boundary map implies that $\rho(\gamma)$ has a unique attracting fixed point $\xi(\gamma^{+}) \in \mathbb{P}\mathcal{F}_{3}(\C)$ and a unique repelling fixed point $\xi(\gamma^{-}) \in \mathbb{P}\mathcal{F}_{3}(\C)$. We denote by $\xi_{1}(\gamma^{\pm})$ and $\xi_{2}(\gamma^{\pm})$ the line and the plane in the flags $\xi(\gamma^{\pm})$, respectively. Let $\tilde{\Sigma}$ be the image of $\tilde{\sigma}$ in $\CH^{2}_{\tau}$. We claim that there is a point $x\in \tilde{\Sigma}$ such that 
\[
    \lim_{n \to +\infty} \rho(\gamma)^{n} \cdot x = \beta^{-1}(\xi(\gamma^{+})) \ .
\]
Indeed, by Proposition \ref{prop:isometries}, every point $(v,\varphi) \in \mathbb{P}\mathcal{L}$ when acted upon by $\rho(\gamma)^{n}$ converges to $\beta^{-1}(\xi(\gamma^{+}))$ unless 
\begin{equation}\label{eq:bad_points}
    v \in \xi_{2}(\gamma^{-}) \ \ \ \text{or} \ \ \ \xi_{1}(\gamma^{-}) \in \Ker(\varphi) \ .
\end{equation}
We now argue that there is at least a point $(v,\varphi) \in \tilde{\Sigma}$ which does not satisfy \eqref{eq:bad_points}. We actually prove more: there is no open set of $\tilde{\Sigma}$ whose points satisfy either of the two conditions in Equation \eqref{eq:bad_points}. In fact, assume that there is an open set $U \subset \tilde{\Sigma}$ (which we now think of as lifted in $\C^{3}_{\tau}$) such that for all $(v,\varphi) \in U$ we have $v\in \xi_{2}(\gamma^{-})$. Then the tangent space to $\tilde{\Sigma}$ at $(v, \varphi) \in U$ is generated by vectors of the form $(v_{1}, \varphi_{1})$ and $(v_{2}, \varphi_{2})$ in $\C^{3}e_{+}\oplus\C^{3}e_{-}$ such that $v_{i} \in \xi_{2}(\gamma^{-})$. By an adaptation of the proof of Lemma \ref{lm:special_lift}, we can assume that $\mathbf{q}((v,\varphi), (v_{i}, \varphi_{i}))=0$ for $i=1,2$. This implies that
\[
    \varphi_{i}(v)=\varphi(v_{i})=0 \ \ \ \ \text{for $i=1,2$} \ .
\]
In particular, we notice that if $v_{1}, v_{2}\in \C^{3}$ were linearly independent, then $\Ker(\varphi)=\Span_{\C}(v_{1},v_{2})=\xi_{2}(\gamma^{-})$, which would imply that $v \in \Ker(\varphi)$, which is impossible since $(v,\varphi) \in \tilde{\Sigma} \subset \CH^{2}_{\tau}$ and as such $\varphi(v)=-1$. Thus, $v_{1}$ is necessarily a multiple of $v_{2}$ but this means that the tangent space of $\tilde{\Sigma}$ at the point $(v,\varphi)$ is not transverse to the distribution $\mathcal{D}_{+}$, thus contradicting Lemma \ref{lm:transverse}. A similar argument shows that there is no open subset $U$ of $\tilde{\Sigma}$ such that for all $(v, \varphi) \in U$ the constraint $\xi_{1}(\gamma^{-}) \in \Ker(\varphi)$ holds. Hence we can find a point $x \in \tilde{\Sigma}$ such that
\[
     \lim_{n \to +\infty} \rho(\gamma)^{n} \cdot x = \beta^{-1}(\xi(\gamma^{+})) \ .
\]
\noindent On the other hand, the points $\rho(\gamma)^{n} \cdot x$ lie on $\tilde{\Sigma}$ for all $n$ by $\rho$-equivariance and thus their limit belongs to $\partial_{\infty}\tilde{\Sigma}:=\mathrm{cl}(\tilde{\Sigma})\cap \partial_{\infty}\CH^2_{\tau}$. Therefore, $\partial_{\infty}\tilde{\Sigma}$ contains all the points $\beta^{-1}(\xi(\gamma^{+}))$ for all $\gamma \in \pi_{1}(S)$. Since the set $\{ \gamma^{+} \ | \ \gamma\in \pi_{1}(S)\}$ is dense in $\partial_{\infty}\pi_{1}(S)$, we deduce that $\partial_{\infty}\tilde{\Sigma}$ contains $\beta^{-1}(\xi(\partial_{\infty}\pi_{1}(S)))$, thus $\tilde{\Sigma}$ bounds the image of $\beta^{-1}\circ \xi$, as claimed.

\end{proof}

\bibliographystyle{alpha}
\bibliography{bs-bibliography}

\newcommand{\etalchar}[1]{$^{#1}$}
\begin{thebibliography}{BCGR{\etalchar{+}}98}

\bibitem[ADL21]{alessandrini2021projective}
Daniele Alessandrini, Colin Davalo, and Qiongling Li.
\newblock Projective structures with (quasi-) hitchin holonomy.
\newblock {\em arXiv:2110.15407}, 2021.

\bibitem[AMTW23]{alessandrini2023fiber}
Daniele Alessandrini, Sara Maloni, Nicolas Tholozan, and Anna Wienhard.
\newblock Fiber bundles associated with anosov representations.
\newblock {\em arXiv:2303.10786}, 2023.

\bibitem[Anc11]{anciaux2011minimal}
Henri Anciaux.
\newblock {\em Minimal submanifolds in pseudo-Riemannian geometry}.
\newblock World Scientific, 2011.

\bibitem[BCGR{\etalchar{+}}98]{bonome1998paraholomorphic}
A~Bonome, R~Castro, E~Garc{\'\i}a-R{\'\i}o, L~Hervella, and R~V{\'a}zquez-Lorenzo.
\newblock On the paraholomorphic sectional curvature of almost para-hermitian manifolds.
\newblock {\em Houston J. Math}, 24(2):277--300, 1998.

\bibitem[BEE22]{bonsante2022immersions}
Francesco Bonsante and Christian El~Emam.
\newblock On immersions of surfaces into $\mathrm{SL}(2,\mathbb{C})$ and geometric consequences.
\newblock {\em International Mathematics Research Notices}, 2022(12):8803--8864, 2022.

\bibitem[BH13]{benoist2013cubic}
Yves Benoist and Dominique Hulin.
\newblock Cubic differentials and finite volume convex projective surfaces.
\newblock {\em Geometry \& Topology}, 17(1):595--620, 2013.

\bibitem[BW11]{baird2011harmonic}
Paul Baird and John~C Wood.
\newblock Harmonic morphisms and bicomplex manifolds.
\newblock {\em Journal of Geometry and Physics}, 61(1):46--61, 2011.

\bibitem[CTT19]{CTT}
Brian Collier, Nicolas Tholozan, and J\'{e}r\'{e}my Toulisse.
\newblock The geometry of maximal representations of surface groups into {${\rm SO}_0(2,n)$}.
\newblock {\em Duke Math. J.}, 168(15):2873--2949, 2019.

\bibitem[DGMY09]{davidov2009hyperhermitian}
Johann Davidov, Gueo Grantcharov, Oleg Mushkarov, and Miroslav Yotov.
\newblock Para-hyperhermitian surfaces.
\newblock {\em Bulletin math{\'e}matique de la Soci{\'e}t{\'e} des Sciences Math{\'e}matiques de Roumanie}, pages 281--289, 2009.

\bibitem[DS20]{dumas2020geometry}
David Dumas and Andrew Sanders.
\newblock Geometry of compact complex manifolds associated to generalized quasi-fuchsian representations.
\newblock {\em Geometry \& Topology}, 24(4):1615--1693, 2020.

\bibitem[DS21]{dumas2021uniformization}
David Dumas and Andrew Sanders.
\newblock Uniformization of compact complex manifolds by anosov homomorphisms.
\newblock {\em Geometric and Functional Analysis}, 31:815--854, 2021.

\bibitem[EE23]{ElEmam_quasiFuchsian}
Christian El-Emam.
\newblock A metric uniformization model for the {Q}uasi-{F}uchsian space.
\newblock {\em arXiv:2307.07388}, 2023.

\bibitem[EES22]{el2022gauss}
Christian El~Emam and Andrea Seppi.
\newblock On the gauss map of equivariant immersions in hyperbolic space.
\newblock {\em Journal of topology}, 15(1):238--301, 2022.

\bibitem[EES24]{Bers_SL3}
Christian El-Emam and Nathaniel Sagman.
\newblock On a bers theorem for $\mathrm{SL}(3,\mathbb{C})$.
\newblock {\em In preparation}, 2024.

\bibitem[FG06]{fock2006moduli}
Vladimir Fock and Alexander Goncharov.
\newblock Moduli spaces of local systems and higher teichm{\"u}ller theory.
\newblock {\em Publications Math{\'e}matiques de l'IH{\'E}S}, 103:1--211, 2006.

\bibitem[FPV24]{farre2024topological}
James Farre, Beatrice Pozzetti, and Gabriele Viaggi.
\newblock Topological and geometric restrictions on hyperconvex representations.
\newblock {\em arXiv:2403.13668}, 2024.

\bibitem[GMA89]{gadea1989spaces}
Pedro Gadea and {\'A}ngel Montesinos-Amilibia.
\newblock Spaces of constant para-holomorphic sectional curvature.
\newblock {\em Pacific Journal of Mathematics}, 136(1):85--101, 1989.

\bibitem[Gol84]{goldman1984symplectic}
William~M Goldman.
\newblock The symplectic nature of fundamental groups of surfaces.
\newblock {\em Advances in Mathematics}, 54(2):200--225, 1984.

\bibitem[Gol99]{goldman1999complex}
William~Mark Goldman.
\newblock {\em Complex hyperbolic geometry}.
\newblock Oxford University Press, 1999.

\bibitem[GW12]{guichard2012anosov}
Olivier Guichard and Anna Wienhard.
\newblock Anosov representations: domains of discontinuity and applications.
\newblock {\em Inventiones mathematicae}, 190(2):357--438, 2012.

\bibitem[Hil11]{hildebrand2011cross}
Roland Hildebrand.
\newblock The cross-ratio manifold: a model of centro-affine geometry.
\newblock {\em International Electronic Journal of Geometry}, 4(2):32--62, 2011.

\bibitem[Hit87]{hitchin1987self}
Nigel~J Hitchin.
\newblock The self-duality equations on a riemann surface.
\newblock {\em Proceedings of the London Mathematical Society}, 3(1):59--126, 1987.

\bibitem[Hit92]{hitchin1992lie}
Nigel~J Hitchin.
\newblock Lie groups and teichm{\"u}ller space.
\newblock {\em Topology}, 31(3):449--473, 1992.

\bibitem[HLL13]{huang2013holomorphic}
Zheng Huang, John Loftin, and Marcello Lucia.
\newblock Holomorphic cubic differentials and minimal lagrangian surfaces in $\mathbb{C}\mathbb{H}^2$.
\newblock {\em Mathematical Research Letters}, 20(3):501--520, 2013.

\bibitem[Kim07]{hol_ext}
Young-Heon Kim.
\newblock Holomorphic extensions of laplacians and their determinants.
\newblock {\em Advances in Mathematics}, 211(2):517--545, 2007.

\bibitem[KN96]{kobayashi1996foundations}
Shoshichi Kobayashi and Katsumi Nomizu.
\newblock {\em Foundations of Differential Geometry, Volume 2}, volume~61.
\newblock John Wiley \& Sons, 1996.

\bibitem[Kob87]{kobayashi1987differential}
Sh{\=o}shichi Kobayashi.
\newblock {\em Differential geometry of holomorphic vector bundles}.
\newblock Princeton University Press, 1987.

\bibitem[Lab06]{Labourie2006anosov}
Fran{\c{c}}ois Labourie.
\newblock Anosov flows, surface groups and curves in projective space.
\newblock {\em Inventiones Mathematicae}, 165(1):51--114, 2006.

\bibitem[Lab07]{Labourie_cubic}
Fran\c{c}ois Labourie.
\newblock Flat projective structures on surfaces and cubic holomorphic differentials.
\newblock {\em Pure Appl. Math. Q.}, 3(4, Special Issue: In honor of Grigory Margulis. Part 1):1057--1099, 2007.

\bibitem[LB83]{LeBrun}
Claude Le~Brun.
\newblock Spaces of complex null geodesics in complex-riemannian geometry.
\newblock {\em Transactions of the American Mathematical Society}, 278(1):209--231, 1983.

\bibitem[LESSV12]{luna2012bicomplex}
ME~Luna-Elizarraras, M~Shapiro, Daniele~C Struppa, and Adrian Vajiac.
\newblock Bicomplex numbers and their elementary functions.
\newblock {\em Cubo (Temuco)}, 14(2):61--80, 2012.

\bibitem[LM13]{loftin2013minimal}
John Loftin and Ian McIntosh.
\newblock Minimal lagrangian surfaces in $\mathbb{C}\mathbb{H}^2$ and representations of surface groups into $\mathrm{SU}(2, 1)$.
\newblock {\em Geometriae Dedicata}, 162(1):67--93, 2013.

\bibitem[LM16]{loftin2013cubic}
John Loftin and Ian McIntosh.
\newblock Cubic differentials in the differential geometry of surfaces.
\newblock {\em Handbook of Teichm\"uller Theory}, 6:231--274, 2016.

\bibitem[LM19]{loftin2019equivariant}
John Loftin and Ian McIntosh.
\newblock Equivariant minimal surfaces in $\mathbb{C}\mathbb{H}^2$ and their higgs bundles.
\newblock {\em Asian Journal of Mathematics}, 23(1):71--106, 2019.

\bibitem[Lof01]{loftin2001affine}
John~C Loftin.
\newblock Affine spheres and convex $\mathbb{R}\mathbb{P}^n$-manifolds.
\newblock {\em American Journal of Mathematics}, pages 255--274, 2001.

\bibitem[LS17]{loustau2017bi}
Brice Loustau and Andrew Sanders.
\newblock Bi-lagrangian structures and teichm\"uller theory.
\newblock {\em arXiv:1708.09145}, 2017.

\bibitem[LTW23]{LTW}
Fran{\c{c}}ois Labourie, J\`er\`emy Toulisse, and Michael Wolf.
\newblock Plateau problems for maximal surfaces in pseudo-hyperbolic spaces.
\newblock {\em Ann. Sci. \'Ec. Norm. Sup\'er. (to appear)}, 2023.

\bibitem[MW74]{marsden1974reduction}
Jerrold Marsden and Alan Weinstein.
\newblock Reduction of symplectic manifolds with symmetry.
\newblock {\em Reports on mathematical physics}, 5(1):121--130, 1974.

\bibitem[NS94]{nomizu1994affine}
Katsumi Nomizu and Takeshi Sasaki.
\newblock {\em Affine differential geometry: geometry of affine immersions}.
\newblock Cambridge university press, 1994.

\bibitem[Run23]{rungi2023pseudo}
Nicholas Rungi.
\newblock {\em Pseudo-K{\"a}hler geometry of Hitchin representations and convex projective structures}.
\newblock PhD thesis, SISSA, 2023.

\bibitem[Sik12]{sikora2012character}
Adam Sikora.
\newblock Character varieties.
\newblock {\em Transactions of the American Mathematical Society}, 364(10):5173--5208, 2012.

\bibitem[Sim92]{simpson_Higgs}
Carlos~T. Simpson.
\newblock Higgs bundles and local systems.
\newblock {\em Publications Mathématiques de l'Institut des Hautes Études Scientifiques}, 75:5--95, 1992.

\bibitem[Tre19]{trettel2019families}
Steve~J Trettel.
\newblock {\em Families of geometries, real algebras, and transitions}.
\newblock University of California, Santa Barbara, 2019.

\bibitem[Wei80]{weinstein1980symplectic}
Alan Weinstein.
\newblock The symplectic "category".
\newblock {\em Differential Geometric Methods in Mathematical Physics: Clausthal 1980 Proceedings of an International Conference Held at the Technical University of Clausthal}, pages 45--51, 1980.

\bibitem[Wie16]{wienhard2016representations}
Anna Wienhard.
\newblock Representations and geometric structures.
\newblock {\em arXiv:1602.03480}, 2016.

\bibitem[Wie18]{Wienhard_intro}
Anna Wienhard.
\newblock An invitation to higher {T}eichm\"{u}ller theory.
\newblock In {\em Proceedings of the {I}nternational {C}ongress of {M}athematicians---{R}io de {J}aneiro 2018. {V}ol. {II}. {I}nvited lectures}, pages 1013--1039. World Sci. Publ., Hackensack, NJ, 2018.

\end{thebibliography}

\end{document}